\documentclass[leqno]{amsart}
\usepackage{amsfonts,amssymb,amsmath,amsgen,amsthm,datetime}
\usepackage{hyperref}
\usepackage[scale=0.70, centering, headheight=14pt]{geometry}

\theoremstyle{plain}
\newtheorem{theorem}{Theorem}[section]

\newtheorem{assumption}[theorem]{Assumption}
\newtheorem{lemma}[theorem]{Lemma}
\newtheorem{corollary}[theorem]{Corollary}
\newtheorem{proposition}[theorem]{Proposition}

\theoremstyle{remark}
\newtheorem{remark}[theorem]{Remark}

\usepackage{latexsym,amssymb}
\usepackage{graphicx}
\usepackage{xcolor}
\definecolor{Lys}{RGB}{0,128,128}
\definecolor{Silver}{RGB}{192,192,192} 

\def\C{{\mathbb C}}

\def\N{{\mathbb N}}
\def\R{{\mathbb R}}

\def\eps{\varepsilon}
\def\op_#1{\mathrel{\mathop{{\rm op}_{#1}}}}

\def\build#1_#2^#3{\mathrel{
\mathop{\kern 0pt#1}\limits_{#2}^{#3}}}

\def\td_#1,#2{\mathrel{
\mathop{\build\longrightarrow_{#1\rightarrow #2}^{}}}}

\def\limsup_#1,#2{\mathrel{
\mathop{\build{\rm limsup}_{#1\rightarrow#2}^{}}}}

\def\liminf_#1,#2{\mathrel{
\mathop{\build{\rm liminf}_{#1\rightarrow#2}^{}}}}

\def\Tend#1#2{\mathop{\longrightarrow}\limits_{#1\rightarrow#2}^{}}

\def\eps{\varepsilon}

\def\1{{\bf 1}}
\def\0{{\bf 0}}

\def\op{{\rm op}}

\begin{document}

\title{Propagation of coherent states through conical intersections}
\author[C. Fermanian Kammerer]{Clotilde Fermanian Kammerer}
\address[C. Fermanian Kammerer]{
Univ Paris Est Creteil, CNRS, LAMA, F-94010 Creteil, France\\
Univ Gustave Eiffel, LAMA, F-77447 Marne-la-Vallée, France}
\email{clotilde.fermanian@u-pec.fr}
\author[S. Gamble]{Stephanie Gamble}
\address[S. Gamble]{
Department of Mathematics, Virginia Tech
Blacksburg, VA 24061-1026 \&  Savannah River National Laboratory, Aiken, South Carolina
}
\email{sgamble7@vt.edu \& stephanie.gamble@srnl.doe.gov}
\author[L. Hari]{Lysianne Hari}
\address[L. Hari]{Laboratoire de mathématiques de Besançon (LMB), Université Bourgogne Franche-Comté, CNRS - UMR 6623, F-25030 Besan\c con, France}
\email{lysianne.hari@univ-fcomte.fr}

\maketitle

\begin{abstract} In this paper, we analyze the propagation of a wave packet through a conical intersection. This question has been addressed for Gaussian wave packets in the 90s by George Hagedorn and we consider here a more general setting. We focus on the case of Schr\"odinger equation but our methods are general enough to be adapted to systems presenting codimension~2 crossings and to codimension~3 ones with specific geometric conditions.  
Our main Theorem gives explicit transition formulas for the profiles when passing through a conical crossing point, including precise computation of the transformation of the phase.  Its proof is based on a normal form approach combined with the use of superadiabatic projectors and the analysis of their degeneracy close to the crossing. 
\end{abstract}

\tableofcontents

\section{Introduction}

We consider a  system of two Schr\"odinger equations coupled by a matrix-valued potential
\begin{equation}\label{system}
i\eps\partial_t \psi^\eps = -\frac{\eps^2}{2}\Delta \psi^\eps+V(x) \psi^\eps,\;\;\psi^\eps_{|t=t_0}=\psi^\eps_0
\end{equation}
where $\psi^\eps_0$ is a bounded family in $L^2(\R^d,\C^2)$, 
and $V\in{\mathcal C}^\infty(\R^d, \C^{2,2} )$  is a self-adjoint matrix that we assume to be subquadratic: $\| \cdot\|_{\C^{2,2}}$ denotes a norm in the space of matrices  $\C^{2,2}$, the matrix $V$ satisfies
\begin{equation}\label{rangeV}
\forall \gamma\in \N^d,\;\;|\gamma|\geq 2,\;\;\exists c_\gamma>0,\;\;\sup_{x\in\R^d}\| \partial_x^\gamma V(x)\|_{\C^{2,2} }\leq c_\gamma.
\end{equation}

These assumptions  guarantee  
the existence of solutions to equation~(\ref{system}) in $L^2(\R^d,\C^2)$ or, more generally, in the  functional spaces $\Sigma^k_\eps:=\Sigma_\eps^k(\R^d,\C^2)$ 
containing functions $f\in L^2(\R^d,\C^2)$ such that 
$$\forall \alpha,\beta\in\N^d,\; |\alpha|+|\beta| \leq k,\; x^\alpha (\eps \partial_x)^\beta f\in L^2(\R^d,\C^2)$$
with a uniform control of the norm, with respect to $\eps\in(0,1]$
$$\| f\|_{\Sigma^k_\eps} = \sup_{|\alpha|+|\beta| \leq k}\| x^\alpha (\eps \partial_x)^\beta f\|_{L^2}.$$
For simplicity, we denote by $\Sigma^k$ the sets $\Sigma^k_\eps$ corresponding to $\eps=1$. The Schwartz space ${\mathcal S}(\R^d)$ then satisfies $\cap_{k\in\N} \Sigma^k={\mathcal S} (\R^d)$. The initial data that we will consider belong to $\Sigma^k_\eps$ for all $k\in\N$, as explained below. 
\smallskip

Let us first detail the assumptions we make on the matrix structure of the potential. As any symmetric matrix, 
the potential~$V$ can be decomposed as the sum of a scalar function and a trace-free matrix: we write 
\begin{equation}\label{eq:V}
V(x)=v(x) {\rm Id}_{\C^2} + \begin{pmatrix} w_1(x)  & w_2 (x)   \\ w_2(x)   & -w_1 (x) 
\end{pmatrix} 
\end{equation}
and  denote by $\lambda_-$ and $\lambda_+$ the eigenvalues of $V$ with $\lambda_-\leq \lambda_+$. We have 
$$\lambda_\pm(x)=v(x)\pm|w(x)|,\;\;|w(x)|=\sqrt{w_1^2(x)+w_2^2(x)}.$$
We associate with these eigenvalues the scalar Hamiltonians
\begin{equation}
\label{def:h+-}
h_\pm(z) = \frac{|\xi|^2}{2}+\lambda_\pm(x),\;\; z=(x,\xi).
\end{equation}
Since $V$ is smooth, the functions $v$ and $w=(w_1,w_2)$ are also smooth and the eigenvalues of~$V$ are smooth outside the set $\Upsilon$ of crossing points 
  $$\Upsilon=\{ z=(x,\xi)\in\R^{2d},\;\; h_+(z)=h_-(z)\}= \{z=(x,\xi)\in\R^{2d},\;\;  w(x)=0_{\R^2}\}. $$
  We shall also consider the eigenprojectors associated with each of the eigenvalues
$$\Pi_\pm(x)=\frac 12 \left( {\rm Id}_{\R^2} \pm \frac 1{|w(x)|} \begin{pmatrix} w_1(x)  & w_2 (x)   \\ w_2(x)   & -w_1 (x) 
\end{pmatrix} \right).$$

 Following~\cite{Hag94}, we will work in the case of conical crossing points by considering the following set of assumptions.
 
  \begin{assumption}\label{hypothesis} \
\begin{enumerate}
\item The crossing on $\Upsilon$ is a  conical crossing of codimension~2:  
\begin{equation*} \label{conical_hypothesis}
\forall q^\flat \in \Upsilon,\;\; {\rm Rank}\, dw(q^\flat) =2.
 \end{equation*}
 \item The conical crossing point $z^\flat=(q^\flat, p^\flat)$ is non-degenerate: 
 \begin{equation*} \label{nondeg_hypothesis}
 E(z^\flat) := (p^\flat\cdot \nabla w_1(q^\flat), p^\flat \cdot \nabla w_2(q^\flat) ) = dw(q^\flat) p^\flat \not =0_{\R^2}
  \end{equation*}
  We write $dw(q^\flat) p^\flat=r\omega$ with $r>0$ and $\omega\in\bf S^1$.
  \end{enumerate}
 \end{assumption}
 
In the notations above, we denote by $dw(q)$ the $2\times d$ matrix 
 $$dw(q)= (\partial_{q_j} w_i )_{1\leq i\leq 2,\;1\leq j\leq d},$$
 meaning that, when applied to a vector $p\in \R^d$, one gets a vector $r\omega= dw(q)p\in \R^2$. Note that Point~(1) of Assumption~\ref{hypothesis} implies that $\Upsilon$ is a submanifold of $\R^d$. Then, the points of~$\Upsilon$ are said to be {conical} crossing points because the eigenvalues $\lambda_+$ and $\lambda_-$ develop a conical singularity at those points.  This singularity induces special behaviors of the solution to Equation~\eqref{system} that has been  already studied in the literature (see~\cite{Hag94,FG02} for example) and that we want to analyze here for wave packets propagation. 
\smallskip

The eigenvalues $\lambda_+$ and $\lambda_-$ are also supposed to satisfy a polynomial gap condition at infinity: we assume  that there exist constants $c_0,n_0,r_0>0$ such that
\begin{equation}\label{hyp:gapinfinity}
|\lambda_+(x)-\lambda_-(x)| \ge c_0 \langle x\rangle^{-n_0}\ ,    \text{ when }\ |w(x)|\ge r_0,
\end{equation}
where $\langle x\rangle = (1+|x|^2)^{1/2}$.
This gap condition at infinity~\eqref{hyp:gapinfinity} ensures, that the derivatives 
of the eigenprojectors $\Pi_\pm(x)$ grow at most polynomially :  it is proved in \cite[Lemma~B.2]{CF11}  that for all $\beta\in\N^{d}$ there exists a constant $C_\beta>0$ such that
\begin{equation}\label{bound:projector}
\|\partial_x^{\beta} \Pi_\pm(x)\|_{\C^{2,2}} \le C_\beta \langle x\rangle^{|\beta|(1+n_0)}\ , \text{ when }\ |w(x)|\ge r_0.
\end{equation}

   \smallskip

We are interested in initial data that are wave packets as studied in~\cite{CR}.  Wave packets are  highly localized in position and impulsion, they are associated with a profile   $\varphi \in \mathcal S(\R^d)$ and a point $z=(q,p)\in\R^{2d}$ of the phase space according to 
\begin{equation}\label{wpdef}
{{\rm WP}}^{\eps}_{z}\varphi(x)= \eps^{-d/4} \,{\rm e}^{{i\over \eps} p\cdot(x-q)} 
\varphi \!\left(\tfrac{x-q}{\sqrt\eps}\right).
\end{equation}
Such families are uniformly bounded in all the spaces $\Sigma^k_\eps$ for any $k\in\N$.
Note that  Hagedorn's wave packets in~\cite{Hag94} are built by choosing $\varphi$ related to Hermite functions. Our set of data contains Hagedorn's ones. 
With these notations, we shall make the following set of assumptions on the initial data. 

\begin{assumption}\label{hyp:data}
	
The initial data of the system~\eqref{system} is given by
$$
\psi^\eps_0(x)={\vec Y_0}\, {{\rm WP}}^\eps_{z_0}\varphi(x),
$$
 where $\varphi \in{\mathcal S}(\R^d)$,  $z_0=(q_0,p_0)\in\R^{2d}\setminus \Upsilon$  and $\vec Y_0\in \R^2$ is a normalized eigenvector of the matrix $V$ in $q_0$ for the $minus$-mode:
  $$V(q_0) \vec Y_0=  \lambda_-(q_0)\vec Y_0.$$
  \end{assumption}
  
  Note that since $\vec Y_0$ is assumed to be a real-valued normalized eigenvector  of $V(q_0)$ with $w(q_0)\not=0$, one can replace the pair $(\vec Y_0,\varphi)$ by $(-\vec Y_0, -\varphi)$ without changing the wave packet. 
  \smallskip

 Wave packets satisfy localization properties that are recalled in Appendix~\ref{appB}. In particular, considering a function $\vec V_0\in {\mathcal C}^\infty(\R^d, \R^2)$ such that $\vec V_0(q_0)= \vec Y_0$, we have 
 \begin{equation}\label{eq:data}
 \psi^\eps_0(x)={\vec V_0}(x) {\rm WP}^\eps_{z_0}\varphi(x) +\mathcal{O}(\sqrt\eps)
 \end{equation}
 in $\Sigma_\eps^k$ for all $k\in\N$. Additionally, we can assume without loss of generality, that $\vec V_0(x)$ is an eigenvector of $V(x)$ associated with $\lambda_-(x)$ for all $x$ in a neighborhood $\Omega$ of $q_0$.
\smallskip

It is well-known (and we provide a detailed exposition of those results below) that, outside the crossing set, such a wave packet propagates along the classical trajectories associated with the mode $\lambda_-(x)$ (see~\cite{CR}).  We aim at precisely describing what happens when a wave packet reaches the crossing set, and  passes  through it. These results have been announced in~\cite{Ga}.
\smallskip

We provide a picture similar to the one involving Gaussian wave packets in~\cite{Hag94}:
as long as the gap remains large enough on the trajectory, the solution can be approximated by  a wave packet with a time dependent profile, an action $S_-(t,t_0,z_0)$ and a time dependent eigenvector $\vec Y_-(t)$
$$
\psi^\eps(t) = {\vec Y_-(t)}{\rm e}^{\frac i\eps S_-(t,t_0,z_0)} {\rm WP}^\eps_{\Phi^{t,t_0}_{-}} (u_-(t))+o(1),$$
in $\Sigma^k_\eps$. Besides, as soon as the gap shrinks, transitions occur on the other mode, leading to the birth of a quite similar wave packet on the other mode. The advantage of considering general wave packets lies in the fact that the transitions generate contributions on each mode that keep the more general structure, while the Gaussian one is not preserved (see~\cite{Hag94}).  
\smallskip

 We use the following ingredients:
\begin{enumerate}
\item The existence of generalized trajectories that exist despite the conical singularity (see~\cite{Hag94,FG02,FG03}).
\item The use of real-valued eigenvectors evaluated along the time-dependent classical trajectories.
\item The introduction of a profile equation along a trajectory and the analysis of this profile when the trajectory reaches a crossing point, proving precise estimates on its behavior close to the crossing time. This is performed in Section~\ref{sec:profile} and uses  ideas from~\cite{Hag94,HJ}. \item The definition of a thin layer close to the crossing point of the trajectory and the reduction to a model problem in this thin layer. 
\end{enumerate}
In the next Section~\ref{sec:classical}, we introduce the main objects (classical trajectories, actions, eigenvectors and profiles) that characterize the approximate solution, and our result is stated in Section~\ref{sec:result}. 
\smallskip

We point out that this transfer has been precisely described in terms of Wigner measures by the results of~\cite{FG03} when one single wave packet reaches a crossing point. However, if two wave packets reach simultaneously a point of the crossing set, the Wigner measure information is not enough and a phase information is needed to describe the Wigner measure of the outgoing wave packets. One of our aim here is to get this phase information.  
\smallskip

Even though our results are inspired by those of~\cite{Hag94}, they differ on several aspects. First,  
  the way we handle the problem is different and easier to generalize to other Hamiltonians.  Secondly, the results obtained are more general in terms of the data that are considered. Thirdly, the method we develop also allow to treat data passing close to the crossing set and not exactly through it (see Remark~\ref{rem:gap}) and more general Hamiltonian (see Appendix~\ref{app:generalization}).  
  The latter point opens the way to further development and proofs of the convergence of numerical methods mixing surface hopping approaches~\cite{FL08,FL08bis,FL12,FL17,Lu} and thawed or frozen Gaussian algorithms (also called Herman-Kluk approximation) as introduced in chemical literature in~\cite{HK,Kay1,Kay2} and studied from a mathematical point of view in~\cite{R,RS} (see also~\cite{FLR2}).


\subsection{The parameters of the approximate solution}\label{sec:classical}

The aim of this paper is to give a precise description of how one can approximate solutions to equation~\eqref{system} in the framework of  Assumptions~\ref{hypothesis} and~\ref{hyp:data}. This result is presented in the next section and we begin here by introducing the parameters of the wave packets that are involved in the process. We give a description of their centers, profiles and phase factor, which are $\eps$-independent and related with classical quantities. 

\subsubsection{Classical trajectories and actions}
 For $(t_0,z_0)\in\R\times (\R^{2d}\setminus \Upsilon)$ we consider the classical
trajectory $ (q_\pm(t),p_\pm(t))$ issued from $z_0=(q_0,p_0)$ at time~$t_0$, and defined by the ordinary differential equation
$$\dot q_\pm(t)=p_\pm(t),\;\;\dot p_\pm(t)=-\nabla \lambda_\pm(q_\pm(t))$$
with 
$$q_\pm(t_0)=q_0\;\;\mbox{and}\;\;p_\pm(t_0)=p_0.$$
The associated flow map is then denoted by $\Phi_\pm^{t,t_0}(z_0)= (q_\pm(t),p_\pm(t))$ and we have 
\begin{equation}\label{def:flot}
\partial_t\Phi_\pm^{t,t_0}= J\nabla_z h_\pm\circ\Phi_\pm^{t,t_0},\quad\Phi_\pm^{t_0,t_0} = \1_{\R^{2d}},
\end{equation}
where
\begin{equation}\label{def:J}
J = \begin{pmatrix}0 & {\rm Id}_{\R^d}\\ -{\rm Id}_{\R^d} & 0\end{pmatrix}
\end{equation}
and the Hamiltonians $h_\pm$ are defined in~\eqref{def:h+-}.
It will be convenient in the following to denote by $\{f,g\}$ the Poisson bracket of two  functions 
$f,g\in{\mathcal C}^\infty(\R^{2d})$, that might be scalar-, vector- or matrix-valued as soon as the product $fg$ makes sense:
$$
\{f,g\}:= J\nabla f\cdot \nabla g= \sum_{j=1}^d \left(\partial_{\xi_j} f \partial_{x_j} g- \partial_{x_j} f \partial_{\xi_j} g\right).
$$

Of course, since $w(q_0)\not=0_{\R^2}$, the existence of these Hamiltonian trajectories is guaranteed by Cauchy-Lipschitz theorem, as long as they do not reach~$\Upsilon$. Moreover, one can prove that there exist trajectories passing through~$z^\flat=(q^\flat,p^\flat) \in\Upsilon$ that are piecewise smooth, as soon as Assumptions~\ref{hypothesis}  hold at point $(q^\flat,p^\flat)$.  We point out that we will make the convenient abuse of notations of saying indistinctly that $z=(q,p)\in \Upsilon$ or $q\in\Upsilon$. 

\begin{proposition}\label{prop:traj}[\cite{FG03}, Proposition~1]
Let $z^\flat\in\Upsilon$ satisfying Assumptions~\ref{hypothesis}, the notations of which  we use.
Then,  there exist two continuous maps 
$$t\mapsto \Phi_\pm^{t,t^\flat}(z^\flat)=(q_\pm(t),p_\pm(t))$$
 defined in a neighborhood of~$t^\flat$ and which satisfy~\eqref{def:flot} for $t\not=t^\flat$ with moreover $\Phi_\pm^{t^\flat,t^\flat}(z^\flat)=z^\flat$. Besides, we have for all $t\sim t^\flat$
\begin{align}\label{eq:asympflot}
w(q_\pm(t)) &= (t-t^\flat)r\omega +\mathcal{O}((t-t^\flat)^2).
\end{align}
\end{proposition}

We shall call {\it generalized trajectories} these continuous maps passing through points $z^\flat\in \Upsilon$ satisfying Assumptions~\ref{hypothesis}. We associate with  $\Phi^{t,t_0}_\pm(z_0)=(q_\pm(t),p_\pm(t)) $ the  
 action integral 
\begin{equation}
\label{def:S} 
S_{\pm}(t,t_0,z_0) = \int_{t_0}^t \left(p_\pm(s)\cdot \dot q_\pm(s)-h_\pm(z_\pm(s)) \right) ds.
\end{equation}
We analyze in Section~\ref{sec:actionproof} the behavior of both these  generalized trajectories and their actions close to a crossing point.

\subsubsection{Real-valued time-dependent eigenvectors along the trajectory} 
\label{sec:eigenvec}

We introduce the matrix-valued function $B_\pm\in\mathcal C^\infty(\R^{2d}\setminus~\Upsilon)$ defined by
 \begin{equation}\label{def:B+-}
 B_\pm(x,\xi)=\pm\Pi_\mp (x)\xi\cdot \nabla_x\Pi_+(x)\,\Pi_\pm(x)=\mp\Pi_\mp (x)\xi\cdot \nabla_x\Pi_-(x)\,\Pi_\pm(x)=-B_\mp(x,\xi)^*.
 \end{equation}

\begin{proposition} \label{prop:ingoeigen}
Let $(t_0,z_0)\in \R^{2d+1}$ be such that the trajectory $\Phi_-^{t,t_0}=(q_-(t),p_-(t))$ reaches~$\Upsilon$ at time~$t^\flat>t_0$ and point $z^\flat =\Phi^{t^\flat,t_0}_-(z_0)$ satisfying Assumption~\ref{hypothesis}.  Let $\vec Y_0$ such that $\Pi_-(q_0) \vec Y_0= \vec Y_0$.
Then, the solution $\vec Y_-(t)$ of the  differential system 
\begin{equation}\label{eq:traj-}
\left\{ \begin{array} l
\partial_t \vec Y_-(t)= B_-(\Phi^{t,t_0}_-(z_0)) \vec Y_-(t),\;\; t\in[t_0,t^\flat)\\
\vec Y_-(t_0)=\vec Y_0 
\end{array}\right.
\end{equation}
satisfies the following properties:
\begin{enumerate}
\item for all $t\in [t_0,t^\flat)$, $\vec Y_-(t)$
is an eigenvector for the $minus$-mode along the trajectory: 
$$\Pi_-(q_-(t))\vec Y_-(t)= \vec Y_-(t).$$
\item There exists a normalized real-valued vector $\vec V_\omega$ such that 
$$
 \lim_{t\rightarrow t^\flat, \, t<t^\flat} \vec Y_-(t)= \vec  V_\omega
$$
and
\begin{equation}\label{def:Vomega'}
\begin{pmatrix} \omega_1  & \omega_2    \\ \omega_2   & -\omega_1 \end{pmatrix} 
\vec V_\omega =\vec V_\omega  \;\;\mbox{with}\;\; \vec V_\omega\in\R^2,\;\;
|\vec V_\omega|=1.
\end{equation}
\item 
There exist $\tau>0$ and a function $x\mapsto \vec V_-(x)$ smooth in a neighborhood of $(\Phi^{t,t_0}_-(z_0))_{t\in [t^\flat-\tau,t^\flat)}$ such that $\Pi_- \vec V_-= \vec V_-$ and 
 $\vec Y_-(t)=\vec V_-(q_-(t))$. 
 \end{enumerate}
 \end{proposition}
 
 Note that since $\vec V_\omega$ is real-valued, the relation~\eqref{def:Vomega'}
 fixes $\vec V_\omega$ up to its sign. Its sign depends on the value of $\vec Y_0$. 

\smallskip

 Of course, a similar result holds for the $plus$-mode. More generally, one can construct ingoing and outgoing eigenvectors along the trajectories arising from a non-degenerate conical crossing point~$z^\flat$.
 
 \begin{proposition}\label{prop:eigencon}
 Let $t^\flat\in\R$ and $z^\flat$ satisfying Assumption~\ref{hypothesis} and $\vec V_\omega\in\R^2$ satisfying~\eqref{def:Vomega'}. Let~$\vec V_\omega^\perp$ obtained by the rotation of angle $\frac\pi 2$. There exist two families of eigenvectors $\vec Y_\pm(t)$ defined in a neighborhood $I$ of $t^\flat$, such that 
\begin{equation}\label{eq:traj_flat}
\partial_t \vec Y_\pm(t)= B_\pm(\Phi^{t,t^\flat}_\pm(z^\flat)) \vec Y_\pm(t),\;\; t\in I\setminus \{t^\flat\}\\
\end{equation}
and 
\begin{equation}\label{eq:limvect}
 \lim_{t\rightarrow t^\flat, \, t>t^\flat} \vec Y_+(t)=
 \lim_{t\rightarrow t^\flat, \, t<t^\flat} \vec Y_-(t)= \vec  V_\omega,\;\;
\lim_{t\rightarrow t^\flat, \, t>t^\flat} \vec Y_-(t)=
 \lim_{t\rightarrow t^\flat, \, t<t^\flat} \vec Y_+(t)= \vec  V_\omega^\perp.
\end{equation}
 \end{proposition}

As a consequence, starting at time $t_0$ from a trajectory $\Phi^{t,t_0}_-(z_0)$ for the $minus$-mode that reaches $\Upsilon$ at a non-degenerate conical crossing point $z^\flat$, we
 are left with a family of time-dependent eigenvectors~$\vec Y_-(t)$    that reaches the crossing and defines a vector $\vec V_\omega$. One can then continuously pass  through the crossing, while hopping from the $minus$-mode to the $plus$-mode at time $t^\flat$. 
Similarly, with the generalized  trajectory arriving at time $t^\flat$ in $z^\flat$ for the $plus$-mode, one can associate a  family of time-dependent eigenvector for the $plus$-mode  that will pass continuously through the crossing with~\eqref{eq:limvect}, while hopping from the $plus$-mode to the $minus$-mode at time $t^\flat$.


\subsubsection{Profile equations}\label{sec:profile}
The profiles of the approximate solutions are linked with the scalar Hamiltonians~$h_\pm$ - see~\eqref{def:h+-} - and the associated trajectories. We consider trajectories $\Phi^{t,t_0}_\pm(z_0)$ that do not meet $\Upsilon$ on some time interval $I$ containing $t_0$ and associate with them the Schr\"odinger equations with time-dependent harmonic potential 
\begin{equation} \label{def:profile}
 i\partial_t u_\pm =-\frac 12 \Delta u_\pm + {1\over 2} {\rm Hess}\, \lambda _\pm(\Phi_\pm^{t,t_0}(z_0))y\cdot y \,u^\pm,
\end{equation}
with initial data in ${\mathcal S}(\R^d)$.
In view of~\cite{MaRo}, these equations have a solution in $\Sigma^k(\R^d)$ on the time interval $I$ for any $k\in\N^*$. Moreover, we have the following proposition.

\begin{proposition}\label{prop:profile}
Let $(t_0,z_0)\in \R^{2d+1}$ be such that the trajectory $\Phi_\pm^{t,t_0}$ reaches~$\Upsilon$ at time~$t^\flat>t_0$ and point $z^\flat =\Phi^{t^\flat,t_0}_\pm(z_0)$ satisfying Assumption~\ref{hypothesis}. Then, there exists a solution $u_\pm(t)$ to~\eqref{def:profile} on $[t_0,t^\flat)$ with initial data $u^\pm(t_0)=\varphi _\pm\in{\mathcal S}(\R^d)$. Moreover, for any $t\in[t_0,t^\flat)$, $u_\pm(t)\in\mathcal S(\R^d)$, $\| u_\pm(t)\|_{L^2} =\| \varphi_\pm\|_{L^2}$ and  if  $k\in\N^*$, 
 there exists a constant $C_k>0$ such that 
\begin{equation}\label{est:profileHk}
\sup _{t\in [t_0,t^\flat)} \| u_\pm(t) \|_{\Sigma^k}\leq C_k \left( 1+\left|\ln |t-t^\flat|\right| \right).
\end{equation}
\end{proposition}

The result of Proposition~\ref{prop:profile},   implies that the  time derivatives of the profile functions~$u_+$ and~$u_-$ are  integrable, up to a  phase.
With the notations of Assumptions~\ref{hypothesis}, we consider the $d\times d$
 matrix~$\Gamma_0$ defined by 
\begin{equation}\label{eq:Gamma0}
 \Gamma_0= r^{-1}  \, ^t dw(q^\flat) ({\rm Id}_{\R^2} -\omega\otimes \omega) dw(q^\flat)
 \end{equation}
where $\omega\otimes \omega$ is the $2$ by $2$ matrix $(\omega_i\omega_j)_{i,j}$ and $dw$ is the $2\times d$ matrix $(\partial_{x_j} w_i)_{i,j}$
 (note that ${\rm Id}_{\R^2} -\omega\otimes \omega$ is the orthogonal projector on $\R \,\omega^\perp$). 

\begin{corollary}\label{cor:profile}
Under the assumptions of Proposition~\ref{prop:profile}, there exists $u^{\rm in}_\pm\in {\mathcal S}(\R^d)$ such that for all $k\in \N$, there exists $C_k>0$ with
\begin{equation}\label{eq:profile_in}
\left\| {\rm Exp} (\mp \frac i 2\Gamma_0 y\cdot y\, \ln|t-t^\flat| ) u_\pm(t) 
-u^{\rm in}_\pm\right\|_{\Sigma^k} \leq C_k |t-t^\flat| \,\left(1+\left |\ln |t-t^\flat| \right|\right).
\end{equation}
Moreover, once given $u^{\rm out}_\pm\in {\mathcal S}(\R^d)$, there exists a unique pair $u_\pm(t)$ for $t>t^\flat$ satisfying~\eqref{def:profile} and such that for all $k\in \N$, there exists $C_k>0$ with\begin{equation}\label{eq:profout}
\left\|{\rm Exp} (\pm \frac  i 2  \Gamma_0 y\cdot y\,\ln|t-t^\flat|) u_\pm(t)
- u^{\rm out}_\pm\right\|_{\Sigma^k} \leq C_k |t-t^\flat| \,\left(1+\left |\ln |t-t^\flat| \right|\right).
\end{equation}
\end{corollary}

Let us consider an initial data as in~\eqref{eq:data} and assume that $\Phi^{t,t_0}_-(z_0)$ passes through~$\Upsilon$ at time $t^\flat$ at a point $z^\flat$ that satisfies Assumption~\ref{hypothesis}. Then one can associate a profile $u_-(t)$ with the ingoing trajectory $\Phi^{t,t_0}_-$  for $t\in [t_0, t^\flat)$; this generates an ingoing profile
$u^{\rm in}_-\in {\mathcal S}(\R^d)$. We shall see later how to build an approximate solution to the system~\eqref{system} thanks to $u^{\rm in}_-$, and how to associate two outgoing profiles, $u^{\rm out}_-$ and $u^{\rm out}_+$, with $u^{\rm in}_-$  in an adequate manner; these outgoing profiles  then
generate two profiles $u_+(t)$ and $u_-(t)$ when $t>t^\flat$, one for each mode, by solving equation~\eqref{def:profile} with initial data  at time $t^\flat$ given by $u^{\rm out}_-$ and $u^{\rm out}_+$ respectively.

\subsection{Main results} \label{sec:result}

Let us consider an initial data at time $t_0$ satisfying~\eqref{eq:data} and assume that the trajectory $\Phi^{t,t_0}_-(z_0)$ does not reach $\Upsilon$ on the interval $[t_0,t_0+T]$ because  $\Phi^{t,t_0}_-(z_0)\in\{|w(x)|\geq \delta\}$ for some $\delta>0$. Then, there is adiabatic propagation of the wave packet: at leading order, the solution remains in the same eigenspace and can be approximated by a wave packet whose parameters are determined by the classical quantities associated with the related eigenvalue. This type of results are already present in the literature, see~\cite{bi} for the case of wave packets and~\cite{MS,Te} for more general results. Our contribution here is intended to emphasize the dependence of the approximation on the parameter $\delta$, encoding the minimum gap along the trajectory, which is  a crucial ingredient in the proof of our next result. 

\begin{theorem}\label{theo:adiabatic}[Propagation with a gap of size $\delta$]
Let $k\in\N$. 
Assume $\psi^\eps_0$ is chosen as in Assumption~\ref{hyp:data}.
Let $\delta>0$ and assume that  $\Phi^{t,t_0}_-(z_0)\in\{|w(x)|\geq \delta\}$ for all $t\in [t_0,t_0+T]$. Consider the time-dependent eigenvector  $\vec Y_-(t)$ given by Proposition~\ref{prop:ingoeigen} and the profile $u_-$ associated with $\varphi_-$ by Proposition~\ref{prop:profile}.
 Then, there exists $C_k>0$ independent of $\delta$ such that 
$$\left\| \psi^\eps(t) - \vec Y_-(t)\,  {\rm e}^{\frac i\eps S_-(t,t^\flat,z^\flat ) } {\rm WP}_{\Phi^{t,t^\flat}_-(z^\flat )} u_-(t)\right\|_{\Sigma^k_\eps} \leq C_k  \left(1+ \left| \ln \delta \right| \right)\left(\frac{\eps^{3/2}}{\delta^4} +\frac{ \sqrt\eps}\delta \right).$$ 
\end{theorem}

\smallskip

Of course, this result easily extends by linearity to the case of data which have components on both modes with wave packet structures. 
Theorem~\ref{theo:adiabatic} only gives information when the 
 gap along the trajectory is large enough. 

\smallskip

Let us now assume that the trajectory $\Phi^{t,t_0}_-(z_0)$ passes  through  $\Upsilon$ at time $t^\flat\in(t_0,t_0+T)$, $T>0$ at point $z^\flat$ where Assumption~\ref{hypothesis} is satisfied. We consider:
\begin{itemize}
\item The {\it trajectories} $\Phi_-^{t,t_0}(z_0)$ and  $\Phi_+^{t,t^\flat}(z^\flat)$ built in Proposition~\ref{prop:traj}. 
\item The {\it time-dependent eigenvectors} $\vec Y_-(t)$ associated with $\vec Y_0$ by Proposition~\ref{prop:ingoeigen} for $t\in [t_0,t^\flat]$, 
 and the pair of time-dependent eigenvectors $(\vec Y_+(t), \vec Y_-(t))$ of Proposition~\ref{prop:eigencon}  on $[t^\flat, t_0+T]$ with~\eqref{eq:limvect}
\item The {\it profile} $u_-(t)$ built for $t\in [t_0,t^\flat)$, thanks to  Proposition~\ref{prop:profile} with data $u_-(t_0)=\varphi_-$. We define $u^{\rm in}_-$ by Corollary~\ref{cor:profile},
and associate with $u^{\rm in}_-$ the profiles $u_\pm(t)$ defined for $t>t^\flat$ thanks to the outgoing limiting profiles $u^{\rm out}_-$ and $u^{\rm out}_+$  given by
\begin{equation}\label{transition}
\begin{pmatrix}u^{\rm out}_+\\ u^{\rm out}_-\end{pmatrix} 
= \begin{pmatrix}
  {\rm e}^{-i\theta_\eps(\eta)}  b(\eta_2) &   a(\eta_2)  \\ 
 a(\eta_2)  & -{\rm e}^{i\theta_\eps(\eta) } \bar b(\eta_2) 
  \end{pmatrix}
\begin{pmatrix}0\\ {\rm e}^{\frac i\eps S^\flat_-}u^{\rm in}_-\end{pmatrix}. 
\end{equation}
where,  
$S^\flat_-= S_-(t^\flat,t_0, z_0)$ and, with the notations of Assumption~\ref{hypothesis},
\begin{align}
\label{def:eta}
\eta(y)&=\left( \omega \cdot (dw(q^\flat) y), \omega^\perp \cdot (dw(q^\flat) y)\right) =(\eta_1(y),\eta_2(y)),\;\;\forall y\in\R^d,\\
\label{coef.scat}
a(\eta_2)&={\rm e}^{-{\pi\eta_2^2\over 2}},\;\;
b(\eta_2)=
{2i\over \sqrt\pi\eta_2}2^{-i\eta_2^2/2}{\rm
e}^{-\pi\eta_2^2/4}\,
\Gamma\left(1+i{\eta_2^2\over 2}\right) \, {\rm \sinh}\left({\pi\eta_2^2\over 2}\right),\\
\label{def:theta}
\theta_\eps (\eta)& =\frac { \eta_2^2}{2r} \ln \left(\frac{r}{\eps}\right)  +\frac { \eta_1^2}{r}
\end{align}
\end{itemize}
We recall the Gamma function and hyperbolic sine function we use:
$$\Gamma(z)=\int_0^1 \left( \ln \frac{1}{t}\right)^{z-1} \, dt=\int_0^\infty t^{z-1} e^{-t} \, dt,\;\;
\sinh(z)=\frac{e^z-e^{-z}}{2}.$$
We then have the following result. 

\begin{theorem}\label{theo:main}[Propagation of a single wave packet] 
Let $k\in\N$.
Assume $\psi^\eps_0$ is chosen as in Assumption~\ref{hyp:data}  and that the trajectory $\Phi^{t,t_0}_-(z_0)$ reaches~$\Upsilon$ at some time $t^\flat$ and some point $z^\flat$ satisfying Assumption~\ref{hypothesis}. Consider the above-mentioned classical quantities.
Then, as~$\eps$ tends to~$0$, the solution to equation~\eqref{system} with initial data $\psi^\eps_0$ satisfies in $\Sigma^k_\eps(\R^d)$:  if $t\in[t_0,t^\flat)$  
$$\psi^\eps(t)={\rm e}^{\frac i\eps S_-(t,t_0,z_0)}  \vec Y_-(t) {\rm WP}_{\Phi^{t,t_0}_-(z_0)} u_-(t) +\mathcal{O}\left( (1+\left|\ln \eps\right| )\eps^{{\frac 1{14}}^-} \right)$$
and if $t\in(t^\flat, t_0 +T]$, 
\begin{align}\label{eq:appt>tflat}
\psi^\eps(t)& = \vec Y_-(t))\,  {\rm e}^{\frac i\eps S_-(t,t^\flat,z^\flat ) } {\rm WP}_{\Phi^{t,t^\flat}_-(z^\flat )} u_-(t)  \\
\nonumber
&\qquad+\vec Y_+(t) \, {\rm e}^{\frac i\eps S_+(t,t^\flat,z^\flat ) }{\rm WP}_{\Phi^{t,t^\flat}_+(z^\flat)} u_+(t) 
 +\mathcal{O}\left((1+ \left|\ln \eps\right|) \eps^{{\frac 1{14}}^-} \right)
\end{align}
\end{theorem}

By $ \eps^{{\frac 1{14}}^-}$, we mean $\eps^{\frac 1{14}-\varsigma}$ for some $\varsigma\in(0,\frac 1{14})$ small enough. 
Note that the constants involved in the approximation result of Theorem~\ref{theo:main} depend on the initial data, the potential $V$ and the time length $T$ of the approximation.  This is also the case in the next results.

\begin{remark}\label{rem:profil_ln}
 The presence of the phase-shift driven by the function $\theta_\eps(\eta)$ in the transfer formula implies that  if the $L^2$-norms of the outgoing profiles are still uniformly bounded with respect to $\eps$, it will not be the case for their Schwartz semi-norms, that will grow as powers of $\ln(\eps)$. However, setting $f^\eps = {\rm WP}_{z_0}({\rm e}^{i S(y) \ln (\eps) }\varphi(y))$ with $\varphi\in\mathcal S(\R^d)$ and $S\in\mathcal C^\infty(\R^d)$ with polynomial growth together with its derivatives, one can check that the $\eps$-derivatives of $f^\eps$ are uniformly bounded. Indeed, one can prove by a recursive argument that for $\alpha\in\N^d$, 
$$\eps^{|\alpha|}\partial_x^\alpha f^\eps(y)= {\rm WP}_{z_0}({\rm e}^{i S(y) \ln (\eps) }\varphi^\eps_\alpha(y))$$ with for $k\in\N$,
$\displaystyle{\| \varphi^\eps_\alpha\|_{\Sigma^k} \leq c_k(1+(\sqrt\eps \ln(\eps))^{|\alpha|+k},\;\;c_k>0.}$
The wave packet structure is not excessively deteriorated by this phase shift and the approximate solution in~\eqref{eq:appt>tflat} is uniformly bounded in $\Sigma^k_\eps$ with respect to~$\eps$ for all~$k\in\N$.
\end{remark}

The result extends, by superposition principles, to the case where two wave packets 
 interact at a crossing point $z^\flat$. 
 Assume 
 \begin{equation}\label{data:interact}
\psi^\eps_0(x)={\vec Y_{0,-}}\, {{\rm WP}}^\eps_{z_{0,-}}\varphi_-(x)+{\vec Y_{0,+}}\, {{\rm WP}}^\eps_{z_{0,+}}\varphi_+(x) ,
\end{equation}
 where $\varphi_\pm \in{\mathcal S}(\R^d)$,  $z_{0,\pm}=(q_{0,\pm},p_{0,\pm})\in\R^{2d}\setminus \Upsilon$ with 
 $ \Phi^{t^\flat,t_0}_\pm(z_{0,\pm})=z^\flat,$
 and $\vec Y_{0,\pm}\in \R^2$ are normalized real-valued  eigenvectors of the matrix $V$:
  $$V(q_{0,\pm}) \vec Y_{0,\pm}=  \lambda_\pm(q_{0,\pm})\vec Y_{0,\pm}.$$

 We associate with each mode classical quantities:
 \begin{itemize}
 \item One first computes the time-dependent eigenvectors along the trajectories $\vec Y_\pm(t)$, carefully handling the fact that if $\vec V_\omega$ is the vector associated with $\vec Y_{0,-}$, the vector associated with $\vec Y_{0,+}$ by (2) of Proposition~\ref{prop:ingoeigen} adapted to the $plus$-mode  is $\vec V_\omega^\perp$ or $-\vec V_\omega^\perp$.
 If one gets $-\vec V_\omega^\perp$, one has to turn the pair $(\vec Y_{0,+}, \varphi_+)$ into $(-\vec Y_{0,+}, -\varphi_+)$.
 \item Once this issue is fixed, one computes the profiles 
  $u_\pm(t)$ for $t<t^\flat$ associated with the trajectories and the initial data $\varphi_\pm$. Note that the change of initial data $(\vec Y_{0,+}, \varphi_+)$ into $(-\vec Y_{0,+}, -\varphi_+)$ corresponds to changing the ingoing profiles~$u_+^{\rm in}$ into~$-u_+^{\rm in}$. 
  \item  This generates 
 incoming profiles $u_-^{\rm in}$ and $u_+^{\rm in}$ on the $minus$-mode and $plus$-mode, and incoming actions 
$S^\flat_\pm=S_\pm(t^\flat, t_0,z_{0,\pm})$ respectively. Then, we set 
\begin{equation}\label{transitionbis}
\begin{pmatrix}u^{\rm out}_+\\ u^{\rm out}_-\end{pmatrix} 
= \begin{pmatrix}
  -{\rm e}^{i\theta_\eps(\eta)} \bar b(\eta_2) & a(\eta_2)   \\ 
  a(\eta_2)  &  {\rm e}^{-i\theta_\eps(\eta) } b(\eta_2) 
  \end{pmatrix}
\begin{pmatrix} {\rm e}^{\frac i\eps S^\flat_+}u^{\rm in}_+ \\ {\rm e}^{\frac i\eps S^\flat_-}u^{\rm in}_-\end{pmatrix}
\end{equation}
and one computes the outgoing profiles $u_\pm(t)$ for $t>t^\flat$ along the trajectories and with initial data $u^{\rm out}_\pm$ at time $t=t^\flat$.
\end{itemize}

Then, the following result is a straightforward consequence of Theorem~\ref{theo:main} and of the linearity of the equation. 

\begin{corollary}\label{cor:main} [Interactions of wave packets at conical intersections]
The solution of equation~\eqref{system} with initial data~\eqref{data:interact} is given for $t\in(t^\flat , t_0+T]$ by
\begin{align*}
\psi^\eps(t)& = \vec Y_-(t) {\rm e}^{\frac i\eps S_-(t,t^\flat,z^\flat)} {\rm WP}_{\Phi^{t,t^\flat}_-(z^\flat)} u_-(t) 
+\vec Y_+(t)  {\rm e}^{\frac i\eps S_+(t,t^\flat,z^\flat )} {\rm WP}_{\Phi^{t,t^\flat }_+(z^\flat)} u_+(t) +\mathcal{O}((1+ \left|\ln \eps\right| )\eps^{{\frac 1{14}}^-} )
\end{align*}
in $\Sigma^k_\eps(\R^d)$.
\end{corollary}

\begin{remark}
Several remarks are of interest :
\begin{enumerate}
\item The adjustment of the-time dependent eigenvectors is a crucial issue. It is connected with the choice of the basis $(\vec V_\omega, \vec V_\omega^\perp)$ at the level of the transition. This basis plays the role of what is sometimes called a {\it diabatic} basis and the process that we describe above gives a way of choosing a diabatic basis close to a non-degenerate conical crossing point.
\item The actions accumulated during the transport to the conical intersection play a part in the transition process and the new profiles are affected by a $\eps$-dependent phase.
\item The analysis performed above extends to the case of time-dependent symmetric Hamiltonian $H(t,z)$ presenting conical intersections. Appendix~\ref{app:generalization} is devoted to the generalization of the process.
\end{enumerate}
\end{remark}

It is interesting to compute the Wigner measure of the function $\psi^\eps(t)$ ($t> t^\flat$) of Corollary~\ref{cor:main}.

\begin{corollary}\label{cor:wigner}
The (matrix-valued) Wigner measure of the solution to equation~\eqref{system} with initial data~\eqref{data:interact} is given for $t>t^\flat$ by
\begin{equation}
\label{eq:mut}\mu(t,z)= c_+ \delta\left (z-\Phi^{t,t^\flat}_+(z^\flat)\right)  \vec Y_+(t)\otimes \vec Y_+(t)+ c_- \delta\left (z-\Phi^{t,t^\flat}_-(z^\flat)\right) \vec Y_-(t)\otimes \vec Y_-(t)
\end{equation}
 and (with the notations of~\eqref{def:eta} and ~\eqref{coef.scat})
$$c_\pm=\| a(\eta_2) u^{\rm in}_\pm\|^2 + \| \sqrt{1-a(\eta_2)^2} \, u^{\rm in}_\mp \|^2.$$
\end{corollary}

Let us conclude this section with a parallel between our main result and Theorem 6.3 of~\cite{Hag94}. The latter deals with the propagation of a Hagedorn's wave packet through the conical intersection; it corresponds to our Theorem~\ref{theo:main} for $\varphi$ being a Gaussian multiplied by a polynomial function, which implies that the ingoing profile $u_-^{\rm in}$ has the same structure. The outgoing  profiles are decomposed on the basis of Hagedorn's wave packets in formula (6.53). One sees that the component that switches from one mode to the other one only has a finite number of components. In fact, it still has the structure of a Gaussian multiplied by a polynomial function, while the one that keeps going on the same mode has a full decomposition, which is due to the presence of the function~$\Gamma$ in the coefficient $b(\eta)$.  The comparison with our result is easier page 100 (last formula of the page): one can observe the oscillating phase and the exponential transition coefficient in the part of the approximate solution that switches of mode, together with a decomposition on Hermite functions at the top of page 101. The other mode is treated page 102 and 103, where the Gamma function can be spotted. The phase shift itself is more visible in~\cite{HJ} where $\lambda(\eps)$ of Theorem~3.1 is the analogue of our~$\theta_\eps(y)$. The phase~$\lambda(\eps)$ does not depend on $y$ but does depend on the parameters of the avoided crossing that is the subject of~\cite{HJ}. Note that in both references \cite{Hag94} and \cite{HJ}, the scaling of the equation is not the same, as $\eps$ in this present article corresponds to~$\eps^2$ in those contributions.

\subsection{Ideas of the proof, organization of the paper and notations}

An important part of the proof consists in the construction of the approximate solutions and, in particular, in the resolution of equation~\eqref{def:profile}, as well as the analysis of the properties of its solutions. This part is performed in Section~\ref{sec:approxsol}, together with results on the classical quantities. Then the proof proceeds in two steps. We first show that the approximate solution fits outside~$\Upsilon$, which corresponds to times $t\notin (t^\flat-\delta,t^\flat+\delta)$ for some $\delta$ that will be chosen small. In this region - that can be qualified as {\it adiabatic} - the solutions of~\eqref{system} decouples on each of the modes. Using techniques arising from~\cite{bi,Te} for example,  as spelled out in~\cite{FLR}, we carefully analyze the order of the approximation (which involves negative powers of~$\delta$ combined with powers of~$\eps$) in Section~\ref{sec:adiab}. Then, in $(t^\flat-\delta,t^\flat+\delta)$, we are able to reduce to a local model of Landau-Zener's type and exhibit the transitions relations~\eqref{transitionbis} in Section~\ref{sec:trans}.  This allows us to fix the ansatz for times $t>t^\flat+\delta$.  All along the proof, it will be convenient to use the notation
\begin{equation}\label{def:A}
A(w)= \begin{pmatrix} w_1 & w_2 \\ w_2 & -w_1\end{pmatrix},\;\; w\in\R^2.
\end{equation}
Besides, with a vector $V= \begin{pmatrix} v_1\\ v_2 \end{pmatrix} \in\R^2$, we associate the vector $V^\perp = \begin{pmatrix} -v_2\\ v_1 \end{pmatrix}$. Moreover, if $U= \begin{pmatrix} u_1\\ u_2 \end{pmatrix} \in\R^2$, we set $U\wedge V=U^\perp\cdot V=u_1v_2-u_2v_1$. Finally, we will use the notation $D_y=\frac 1i \nabla_y$.

\smallskip

{\bf Acknowledgements.} CFK thanks 
Caroline Lasser and  Didier Robert for several stimulating discussions about this paper. CFK and LH acknowledge support form the CNRS 80|Prime program {\it AlgDynQua}, LH from the regional ANER project {\it ClePh-M} and the ANR JCJC {\it ESSED}. The authors thank George Hagedorn for his stimulating pioneer works and SG and CFK wish to thank him for the kindness  he has always shown them when they have been interacting with him.


\section{Analysis of classical quantities and construction of the approximate solution}\label{sec:approxsol}

In this section, we first focus on the properties of the classical trajectories and actions in the neighborhood of the crossing set. Then, the next subsections are intended to construct the time-dependent eigenvectors along the trajectories and the solutions of the profile equation~\eqref{def:profile}, together with a careful analysis of their properties.

\subsection{The classical trajectories and actions}\label{sec:actionproof}

It is interesting to compare a generalized classical trajectory $\Phi^{t,t_0}_\pm(z_0)$ reaching the crossing set~$\Upsilon$ at time~$t^\flat$ and point~$z^\flat$ with the trajectory 
$\Phi^{t,t^\flat}_0(z^\flat)=(q_0(t),p_0(t))$ associated with the (smooth) Hamiltonian 
\begin{equation}\label{def:h0}
h_0(z)= \frac{|\xi|^2}{2} +v(x).
\end{equation}
A simple Taylor expansion close to $t=t^\flat$ gives the following lemma. 

\begin{lemma}\label{prop:traj_asymp}
	Under the assumptions stated in Proposition~\ref{prop:traj}, we have
	\[ \left\{
	\begin{array}{rcl}
	q_\pm(t) & =  &q_0(t)
	\mp  \,\frac 12 {\rm  sgn}(t-t^\flat) (t-t^\flat)^2 \,^t dw(q^\flat) \omega  +\mathcal{O}((t-t^\flat)^3),\\
	p_\pm(t) &=  & p_0(t)  \mp  \, |t-t^\flat|  ^t dw(q^\flat) \omega + \mathcal{O}((t-t^\flat)^2).
	\end{array}
	\right.\]
\end{lemma}

We recall the notation $S_\pm^\flat =S_\pm(t^\flat,t_0,z_0)$ introduced in Section \ref{sec:result}.

\noindent The next lemma provides a comparison between the action $ S_\pm(t,t^\flat,z^\flat )=  S_\pm(t,t_0,z_0)-\nolinebreak S_\pm^\flat$ associated with a generalized trajectory $\Phi^{t,t_0}_\pm(z_0)$ and the action
\begin{equation}\label{def:S0}
S_{0}(t,t^\flat,z^\flat ) = \int_{t^\flat}^t \left(p_0(s)\cdot \dot q_0(s)-h_0(z_0(s)) \right) ds
\end{equation}
associated with the trajectory $\Phi^{t,t^\flat}_0(z^\flat)$.

\begin{lemma}\label{lem:action}
	Using the notations of Proposition~\ref{prop:traj}
	we have the following asymptotics 
	\begin{align*}
	S_\pm(t,t^\flat,z^\flat )&= 
	S_0(t,t^\flat, z^\flat)   \mp\, {\rm sgn}(t-t^\flat)  r (t-t^\flat)^2 +\mathcal{O}((t-t^\flat)^3),    
	\end{align*}
	and 
	$$S_0(t,t^\flat, z^\flat)= (t-t^\flat )   \left(\frac 12 |p^\flat|^2 - v(q^\flat)\right)  - p^\flat \cdot \nabla v(q^\flat) (t-t^\flat)^2 +\mathcal{O}((t-t^\flat)^3).$$
\end{lemma} 

\begin{proof}[Proof of Lemma~\ref{lem:action}]
	We use that $h_\pm(z_\pm(t))$ is conserved along the trajectory and we write 
	$$\dot S_\pm(t,t_0,z_0)= p_\pm(t) \cdot \dot q_\pm(t) -h_\pm(  z^\flat)= |p_\pm(t)|^2 -h_\pm(z^\flat).$$
	Lemma~\ref{prop:traj} gives
	$$\dot S_\pm(t,t_0,z_0)=   |p^\flat|^2 - 2 p^\flat \cdot \nabla v(q^\flat) (t-t^\flat)  \mp 2 \, dw(q^\flat) p^\flat \cdot \omega  |t-t^\flat| -h_\pm(z^\flat)+\mathcal{O}((t-t^\flat)^2).$$
	Integrating between $t$ and $t^\flat$ and using $ |p^\flat|^2-h_\pm(z^\flat)= \frac 12 |p^\flat|^2 -v(q^\flat)$, we obtain
	\begin{align*}
	S_\pm(t,t_0,z_0) =  &\; S^\flat+ (t-t^\flat )   \left(\frac 12 |p^\flat|^2 -v(q^\flat)\right)  - p^\flat \cdot \nabla v(q^\flat) (t-t^\flat)^2 \\
	& \qquad  \mp \,{\rm sgn} (t-t^\flat) dw(q^\flat) p^\flat \cdot \omega  (t-t^\flat)^2 +\mathcal{O}((t-t^\flat)^3),
	\end{align*}
	and we identify the terms $(t-t^\flat )   \left(\frac 12 |p^\flat|^2 -v(q^\flat)\right)  - p^\flat \cdot \nabla v(q^\flat) (t-t^\flat)^2$ with the first terms of the Taylor expansion of $S_0(t,t^\flat, z^\flat)$ close to~$t^\flat$. 
\end{proof}

\subsection{Parallel transport} \label{sec:eigenvectorproof} 

In this subsection, we prove Propositions~\ref{prop:ingoeigen} and~\ref{prop:eigencon}. We begin with preliminary conditions in order to prepare the elements required for the proof. 
We use the crucial observation that for all $(x,\xi)\in (\R^d\setminus \Upsilon)\times \R^{d}$, the matrix $\xi\cdot \nabla \Pi_+(x)$ is off-diagonal (see Lemma \ref{lem:computation} for details), that is 
\begin{equation}\label{eq:propPi}
\Pi_\pm(x) \xi\cdot \nabla \Pi_+(x) \Pi_\pm(x)=0
\end{equation}
and that for $\alpha\in\N^d$,  there exist constants $C_\alpha>0$, $n_\alpha\in\N$ such that 
\begin{equation}
\label{est:proj}
\|\partial^\alpha_x \Pi_\pm (x) \|_{\C^{2,2} } \leq C_\alpha |w(x)|^{-|\alpha|} \langle x\rangle^{n_\alpha}  ,
\end{equation}
which is obtained by combining the estimate~\eqref{bound:projector} at infinity and the analysis of the singularity close to~$\Upsilon$. 

\smallskip

A simple calculus shows that the pair $(\vec V_+,\vec V_-)$ given by  
$\vec V_\pm(x) = \begin{pmatrix}
\varsigma_\pm(x) \\ \eta_\pm(x)
\end{pmatrix} $
with
$$\varsigma_\pm(x) = \dfrac{w_2(x)}{\sqrt{2}\sqrt{|w(x)|(|w(x)|\mp w_1(x))}} \qquad ; \qquad \eta_\pm(x) = \dfrac{\pm \sqrt{|w(x)|\mp w_1(x)}}{\sqrt{2 }}$$
is a pair of real-valued eigenvectors  of the matrix $V(x)$ given in~\eqref{eq:V}. These functions are smooth in $\{ w_2\not=0\}$ (indeed, one has $|w|\not= \pm w_1$ when $w_2\not=0$). Actually, one cannot construct pairs of eigenvectors that are smooth in $\R^{2d}\setminus \Upsilon$. However, it is possible to construct pairs of eigenvectors that are smooth in $\R^{2d}\setminus \{ w(x)\cdot \vec e\not=0\} $ for all $\vec e\in\R^2$, $|\vec e|=1$. Indeed, we introduce the rotation matrix 
\begin{equation}\label{def:Rtheta}
\mathcal{R}(\theta) = \left(\begin{array}{c c} \cos \frac{\theta}{2} & -\sin \frac{\theta}{2} \\  \sin \frac{\theta}{2} & \cos \frac{\theta}{2} \end{array} \right),\;\;\theta\in \R
\end{equation}
which satisfies 
\begin{equation}\label{eq:matrixKR}
\mathcal{R}(\theta)^*A(w)\mathcal{R}(\theta)=\left( \begin{array}{c c} \vec e_\theta \cdot w & \vec e_\theta \wedge w \\ \vec e_\theta \wedge w & -\vec e _\theta\cdot w \end{array}\right)
\end{equation}
where $\vec e_\theta=(\cos \,\theta,\sin\,\theta)$ (recall $w\wedge w'=w_1w_2'-w_2w_1'$ for $w,w'\in\R^2$). \\
Then, consider the vectors
$\vec V_\pm^\theta(x) = \begin{pmatrix}
\varsigma^\theta_\pm(x) \\ \eta^\theta_\pm(x)
\end{pmatrix} $
with
$$\varsigma^\theta_\pm(x) = \dfrac{w(x)\wedge\vec e_\theta }{\sqrt{2}\sqrt{|w(x)|(|w(x)|\mp w(x)\cdot \vec e_\theta)}} \qquad ; \qquad \eta^\theta_\pm(x) = \dfrac{\pm \sqrt{|w(x)|\mp w(x)\cdot \vec e_\theta}}{\sqrt{2 }},$$  
the pair $(\mathcal R(\theta)^*\vec V_+^\theta,\mathcal R(\theta)^*\vec V_-^\theta)$ with  
gives a pair of eigenvectors of $V(x)$ that are smooth in the region $\R^d\setminus \{w(x)\cdot \vec e_\theta^\perp \not=0\}$.

\smallskip

\begin{lemma}[Control of real-valued eigenvectors outside $\Upsilon$]
 		\label{eigenvectors}
		Let $( \vec V_+, \vec V_-))$ be a pair of  normalized eigenvectors of the matrix $V(x)$ that are smooth in $\Omega=\R^{2d}\setminus \{w(x)\cdot \omega\not=0\}$ for $\omega\not=0$. Then, for all $\alpha \in \N^d$, there exist $C_\alpha >0$, and $n_\alpha \in \N$ such that for $x \in \Omega$
 		\begin{equation}
\label{eq:eigencontrol}
\left\| \partial_x^\alpha \vec V_\pm(x)\right\|_{\C^2} \leq C \left\langle x \right\rangle^{n_\alpha} |w(x)|^{-\alpha}.
 		\end{equation}
 		\newline
 		Moreover, with the notation of~\eqref{def:B+-}, the following relation holds in $\Omega$
 		\begin{equation}
 		\label{xi.nablaY}
 		\xi \cdot \nabla_x \vec V_\pm(x) = B_\pm(x,\xi) \vec V_\pm(x).
 		\end{equation}
 	\end{lemma}

\begin{proof}
\textbullet \; {\bf Proof of} \eqref{eq:eigencontrol}. 
We proceed by induction on $|\alpha|\geq 1$, using the relations
 $$\Pi_\pm \vec Y_\pm = \vec Y_\pm \quad \textrm{and} \quad |\vec Y_\pm|^2=1.$$
 When $|\alpha|=1$ with $\alpha={\bf 1}_j$, we derive the second relation in $x_j$ and using the fact that the vectors are real-valued, we obtain that
$$\partial_{x_j} \vec V_\pm \cdot \vec V_\pm = 0 $$
which implies that $\partial_{x_j} \vec V_\pm$ is colinear to $\vec V_\mp$.
Deriving the first relation, we have
$$\partial_{x_j} \Pi_\pm \vec V_\pm + \Pi_\pm \partial_{x_j} \vec V_\pm = \partial_{x_j} \vec V_\pm$$
that is
\begin{equation}\label{eq:mo}
\partial_{x_j} \Pi_\pm \vec V_\pm  =({\rm Id}_{\R^2}- \Pi_\pm) \partial_{x_j} \vec V_\pm = \Pi_\mp \partial_{x_j} \vec V_\pm = \partial_{x_j} \vec V_\pm
\end{equation}
since $\partial_{x_j} \vec V_\pm$ is colinear to $\vec V_\mp$.
Using \eqref{est:proj}, we obtain \eqref{eq:eigencontrol} for  all $\alpha \in \N^d$ such that $|\alpha|=1$.

\smallskip

We now fix $\alpha \in \N^d$ and suppose that for some $C_\beta>0, \; n_\beta \in \N$, we have
$$\forall \beta \in \N^d, \; |\beta|\leq |\alpha|-1, \; \;  |\partial^\beta_x \vec V_\pm (x) | \leq C_\beta |w(x)|^{-|\beta|} \langle x\rangle^{n_\beta}. $$
Let $j\in\{1,\cdots, d\}$ such that $\alpha_j\not=0$. 
We apply $\partial^{\alpha-{\bf 1}_j}$ to  the relation ``$\partial_{x_j} \Pi_\pm \vec V_\pm = \partial_{x_j} \vec V_\pm$'' from~\eqref{eq:mo}.
The chain rule implies  that $\partial_x^\alpha \vec V_\pm$ is a linear combination of terms $\partial^{\beta} \Pi_\pm \partial ^\gamma \vec V_\pm $ for $\beta+\gamma = \alpha$ with $|\beta|>1$ so that $|\gamma|<|\alpha|$. Using \eqref{est:proj} and the assumption on lower order derivatives of $\vec Y_\pm$, we infer that there exist a constant $C_\alpha$ and an integer $n_\alpha$ (taking the $\sup$ on $(m,\ell)$) such that \eqref{eq:eigencontrol} holds.
\smallskip

\textbullet \; {\bf Proof of\eqref{xi.nablaY}}. 
We write the proof for the $plus$-mode, since the other mode is dealt in the same manner. We first notice that 
$$\xi \cdot \nabla_x \vec V_+ = 
( \xi \cdot \nabla_x \Pi_+) \vec V_+ +  \Pi_+ (\xi \cdot\nabla_x \vec V_+). $$
Since $\vec V_+$ is normalized and real-valued,  $\Pi_+ (\xi \cdot\nabla_x \vec V_+)=0$  and we are left with the relation 
$$\xi \cdot \nabla_x \vec V_+ =( \xi \cdot \nabla_x \Pi_+) \vec V_+ =   ( \xi \cdot \nabla_x \Pi_+)\Pi_+ \vec V_+= \Pi_- ( \xi \cdot \nabla_x \Pi_+)\Pi_+ \vec V_+= B_+\vec V_+.$$ 
\end{proof}

We can now prove Propositions~\ref{prop:ingoeigen} and~\ref{prop:eigencon}.

\begin{proof}[Proof of Proposition~\ref{prop:ingoeigen}]
1- Differentiating in time the expression $\Pi_+(q_-(t)) \vec Y_-(t)$, we obtain 
$$\frac d{dt} (\Pi_+(q_-(t) )\vec Y_-(t))= p_-(t)\cdot \nabla \Pi_+(q_-(t) ) \vec Y_-(t) + \Pi_+(q_-(t) ) B_-(q_-(t))\vec Y_-(t)=0.$$
Indeed, 
\begin{align*}
p_-(t)\cdot  \nabla \Pi_+(q_-(t) ) \vec Y_-(t) & = \Pi_+ (q_-(t))p_-(t)\cdot  \nabla \Pi_+(q_-(t) )\Pi_-(q_-(t))\vec Y_-(t)\\
&= - \Pi_+(q_-(t) ) B_-(q_-(t))\vec Y_-(t)
\end{align*}
by~\eqref{def:B+-}, \eqref{eq:propPi} and $\vec Y_-=\Pi_-\vec Y_-$. Therefore, 
$\Pi_+(q_-(t)) \vec Y_-(t)= \Pi_+(q_-(t_0)) \vec Y_-(t_0)=0$
\smallskip

2-  We start by analyzing $\Pi_-(\Phi^{t,t_0}_- (z_0)) $ and $(\xi\cdot \nabla \Pi_\pm)(\Phi^{t,t_0}_\pm (z_0)) $ when $t$ goes to $t^\flat$ with $t<t^\flat$.
We recall 
	$$\Pi_-(x)=\frac 12 \left( {\rm Id} - |w(x)|^{-1} A(w(x)\right).$$
	By equation~\eqref{eq:asympflot} setting  $\omega=(\omega_1,\omega_2)^t$, We obtain
	\begin{equation}\label{obs1}
	\Pi_-(\Phi^{t,t_0}_-(z_0)) 
	=\ \frac 12 \left( {\rm Id} +  A(\omega) \right) +O(t-t^\flat).
	\end{equation}
	We now consider the limit of $B_-(\Phi_-^{t,t_0}(z_0))$.
Using Lemma~\ref{lem:computation}, we obtain 
\begin{align}
\label{formule:B}
B_-(x,\xi)&\,=- \frac {\xi\cdot\nabla w(x) \wedge w(x)}{2|w(x)|^3} \Pi_+(x) 
\begin{pmatrix} w_2(x) & -w_1(x) \\ -w_1(x) &- w_2(x) \end{pmatrix}
\Pi_-(x) \\
&=-\frac{1}{2}\ \Pi_+(x) \Biggl[
\frac{1}{|w(x)|}
\begin{pmatrix} \xi \cdot \nabla_xw_1 &\xi \cdot \nabla_x w_2 \\ 
\nonumber
\xi \cdot \nabla_xw_2 & -\xi \cdot \nabla_xw_1
\end{pmatrix}
\\
\nonumber
&\qquad
-\frac{ (\xi \cdot \nabla_x w_1 )w_1+(\xi \cdot \nabla_x w_2)  w_2}{|w(x)|^3}
\begin{pmatrix} w_1 & w_2 \\ w_2 & -w_1\end{pmatrix} 
\Biggr] \Pi_-(x)
\end{align}
We now specify this relation to $(x,\xi)=\Phi^{t,t_0}_-(z_0)$.  By definition 
$$p_-(t) \cdot \nabla_x w(q_-(t))=  r\omega +O(|t-t^\flat|),$$
and, using~\eqref{eq:asympflot}, we obtain 
$$\left. \frac{1}{|w(x)|}
\begin{pmatrix} \xi \cdot \nabla_xw_1 &\xi \cdot \nabla_x w_2 \\ 
\xi \cdot \nabla_xw_2 & -\xi \cdot \nabla_xw_1
\end{pmatrix}\right|_{(x,\xi)=\Phi^{t,t_0}_-(z_0)}
=
\frac {1}{|t-t^\flat| } \begin{pmatrix} \omega_1 & \omega_2 \\ \omega_2 & -\omega_1\end{pmatrix} +O(1)$$
and 
$$\left.
\frac{ (\xi \cdot \nabla_x w_1 )w_1+(\xi \cdot \nabla_x w_2)  w_2}{|w(x)|^3}
\begin{pmatrix} w_1 & w_2 \\ w_2 & -w_1\end{pmatrix} \right|_{(x,\xi)=\Phi^{t,t_0}_-(z_0)}=
\frac {1}{|t-t^\flat| } \begin{pmatrix} \omega_1 & \omega_2 \\ \omega_2 & -\omega_1\end{pmatrix} +O(1),$$
that is 
$$\left. 
\frac {\xi\cdot\nabla w(x) \wedge w(x)}{2|w(x)|^2}
\right|_{(x,\xi)=\Phi^{t,t_0}_-(z_0)}
=O(1)
$$
and 
the singularity in $|t-t^\flat|^{-1}$ disappears in the expression of $B_-(\Phi_-^{t,t_0}(z_0))$. We obtain that $B_-(\Phi_-^{t,t_0}(z_0))$ is uniformly bounded in a neighborhood of $ t^\flat$. 

As a consequence of the last   observation, we deduce  the boundedness  of $\partial_t \vec Y_-(t)$ for $t\in [t_0, t^\flat)$, which - in turn - implies that $\vec Y_-(t)$ has a limit  $\vec V_\omega$ when $t$ goes to $(t^\flat)^-$ which is  normalized  and real-valued. Besides, by~\eqref{obs1},  $\vec V_\omega$ is in the range of $\frac 12({\rm Id_{\C^2}}+A(\omega)) $, thus an eigenvector of $A(\omega)$.

\smallskip

3- One checks that the function $w(x) \cdot \omega$ is non zero along the curves $\Phi^{t,t_0}_-(z_0)$ for $t$ close to $t^\flat$. Therefore, we choose the function $\vec V_-(x)$ that is a smooth real-valued eigenvector of $V(x)$ for the $minus$-mode in the region $\{w(x)\cdot \omega \not=0\}$ and so that $\vec V_-(q_-(t))$ has the same limit $\vec V_\omega$ than $\vec Y_-(t)$ as $t$ goes to $t^\flat$ with $t<t^\flat$ by turning $\vec V_-$ into $-\vec V_-$ if necessary. Then, the result comes from the observation
\begin{align*}
\frac d{dt} \vec V_-(q_-(t))  =& \,p_-(t)\cdot \nabla \vec V_- (q_-(t)) = \Pi_+(q_-(t))p_-(t)\cdot \nabla \vec V_- (q_-(t))\\
=&\, \Pi_+(q_-(t))p_-(t)\cdot \nabla \Pi_-(q_-(t)) \vec V_- (q_-(t))= B_-(\Phi^{t,t_0}_-(z_0)) \vec Y_-(t),
\end{align*}
where we have used $(\xi\cdot \nabla \Pi_-) \vec Y_-= \Pi_+ \, \xi\cdot \nabla \vec V_-=  \xi\cdot \nabla \vec V_-$.
\end{proof}

\begin{proof}[Proof of Proposition~\ref{prop:eigencon}]
The proposition follows the same ideas than in the preceding one and is based on the following observations
\begin{eqnarray}
	\label{key12}
	&\Pi_-(\Phi^{t,t^\flat}_- (z^\flat)) \Tend {t}{(t^\flat)^-} V_\omega\otimes \overline V_\omega,\;\;
	\Pi_+(\Phi^{t,t^\flat}_+ (z^\flat)) \Tend {t}{(t^\flat)^-} V_\omega^\perp\otimes \overline V_\omega^\perp, & \\
	\label{key13}
	&\Pi_+(\Phi^{t,t^\flat}_+ (z^\flat)) \Tend {t}{(t^\flat)^+} V_\omega\otimes \overline V_\omega,\;\;
	\Pi_-(\Phi^{t,t^\flat}_- (z^\flat)) \Tend {t}{(t^\flat)^+} V_\omega^\perp\otimes \overline V_\omega^\perp, & \\
	\label{key11}
	&(\xi\cdot \nabla \Pi_\pm)(\Phi^{t,t^\flat}_\pm (z^\flat)) =O(1)\;\;\mbox{when}\;\; t\sim t^\flat.&
	\end{eqnarray}
\end{proof}


\subsection{Resolution of the  profile equations} \label{sec:profileproof}

In this section, properties of the solutions of equation~\eqref{def:profile} are discussed and Proposition~\ref{prop:profile} and Corollary~\ref{cor:profile} are proved. A crucial element of the proof is a good understanding of the singularity of the Hessian of the function $\lambda_\pm$ along the trajectories. We start by a technical Lemma that we shall use later. 

\begin{lemma}\label{lemma_Mpm}
There exist smooth matrices $M_\pm(t)$ defined on $[t_0, t^\flat]$ (resp. $[t^\flat, t^\flat+\tau]$) such that when $t$ tends to $t^\flat$ with $t<t^\flat$ (resp. $t>t^\flat$),
\begin{equation}\label{def:M}
{\rm Hess}\, \lambda_\pm(q_\pm(t))=  M_\pm(t) \pm |t-t^\flat|^{-1} \Gamma_0
\end{equation}
with $\Gamma_0$ given by~\eqref{eq:Gamma0}.
\end{lemma}

\begin{proof}
We have ${\rm Hess}\, \lambda_{\pm}= {\rm Hess} \, v \pm {\rm Hess} (|w|)$ and 
$$\partial^2_{x_ix_j} (|w|)= \partial_{x_ix_j}^2 w \cdot \frac {w}{|w|}+ \frac{\partial_{x_i} w\cdot \partial_{x_j} w}{|w|} - \frac{(\partial_{x_i} w\cdot w)( \partial_{x_j} w\cdot w)}{|w|^3}.$$
We deduce  from~\eqref{eq:asympflot} that
$${\rm Hess} \, \lambda_{\pm}(q_\pm(t) )= \pm \frac{1}{|t-t^\flat|} \Gamma_0 \pm {\rm sgn}(t-t^\flat)  d^2w(z^\flat) \omega +  {\rm Hess} \,v(q^\flat ) 
 +\mathcal{O}(t-t^\flat)$$
with
$$\Gamma_0= r^{-1}(\partial_{x_i} w\cdot \partial_{x_j} w - (\partial_{x_j} w\cdot \omega) (\partial_{x_i} w\cdot \omega))_{1\leq i,j\leq d},$$
whence~\eqref{eq:Gamma0}
\end{proof}

We now prove Proposition~\ref{prop:profile}.

\begin{proof}[Proof of Proposition~\ref{prop:profile}]
Let us consider the operator 
\begin{equation}\label{def:Q}
Q_\pm(t)= - \frac 12 \Delta_y +\frac 12 {\rm Hess}\, \lambda_\pm(q_\pm(t))y\cdot y.
\end{equation}
 This operator has a classical symbol $(y,\xi)\mapsto \frac 12 |\xi|^2 +  \frac 12 {\rm Hess}\left(\lambda_\pm(\Phi_\pm^{t,t_0}(z_0))\right)y\cdot y$ that enjoys  subquadratic estimates in the interval $[t_0,t^\flat[$, which guarantees the existence of the solution (see~\cite{MaRo}): the solution $u_\pm(t)$ exists for all $t\in [t_0,t^\flat[$ and is in all spaces $\Sigma ^k$ for $k\in\N$. Since we know that the $L^2$-norm is conserved, we focus on  $\|u_\pm(t)\|_{\Sigma^k}$ for $k\geq 1$. 
 \smallskip
 
For convenience, we fix a mode, the  \textit{plus}-mode, and  choose $t<t^\flat$. So we drop any mention of the mode as it will cause no confusion in this part of the paper: $$Q(t)=Q_+(t), \; \lambda(q(t))=\lambda_+(q_+(t)), \; u=u_+.$$ The proofs for the $minus$-mode or for $t>t^\flat$ are similar.
We set the following notation,
$$U=\, ^t(yu,D_y u),$$
and our aim is to prove that the norms $\|U\|_{\Sigma^k}$ are bounded for all $k\in\N$.
Using 
\begin{equation}\label{com:Q(t)}
[Q(t),D_y]= - i {\rm Hess}\, \lambda(q(t))y \;\;\mbox{and}\;\;
[Q(t),y]= -\nabla_y=-iD_y,
\end{equation}
we obtain 
\begin{align*}
[ Q(t) ,\,^t(y,D_y)] u_\pm & = \,^t ( -iD_y u, i {\rm Hess}\, \lambda(q(t)) y u)
 =  -i \begin{pmatrix} 0 & {\rm Id}_{\R^d} \\  {\rm Hess}\, \lambda(q(t)) & 0 \end{pmatrix}
\, ^t(yu,D_y u) U.
\end{align*}
We deduce the equation
\[ i\partial_t \, U - Q(t)U=  {\left[Q(t) , \begin{pmatrix} y\\ D_y\end{pmatrix} \right] u}
=  {-i \begin{pmatrix} 0 & {\rm Id}_{\R^d} \\  {\rm Hess}\, \lambda(q(t)) & 0 \end{pmatrix}}
 U
\]
This system is closed and by 
Lemma~\ref{lemma_Mpm} it is a system of the form 
$$i\partial_t \, U - Q(t) \,U= (M(t) +{i} (t-t^\flat)^{-1} \Gamma)\,U,$$
where $t\mapsto M(t)$ 
smoothly depends on $t$ for $t\in [t_0,t^\flat]$ (meaning that it has -  as its derivatives - limits when $t$ goes to $t^\flat$ from below) and $$\Gamma = \begin{pmatrix} 0 & 0\\ -\Gamma_ 0  & 0 \end{pmatrix},\;\;$$
for $\Gamma_0$ defined in \eqref{eq:Gamma0}. Our aim is to 
 prove the following claim :
\begin{center}
\textbf{Claim: } {\it For all $k\in\N$, there exists $C_k>0$ such that  for all $t\in[t_0, t^\flat[$
$$\| U(t)\|_{\Sigma^k(\R^d,\C^{2d})} \leq C_k \left(1+ \left|\ln |t-t^\flat|\right|  \right)	.$$}
\end{center}

	For that purpose, we introduce the following projector of rank~$d$
	$$\mathbb P=\begin{pmatrix} 0 & 0 \\ 0 & {\rm Id}_{\R^d} \end{pmatrix} \qquad ; \qquad \textrm{satisfying \; }(1-\mathbb P)\Gamma=0\;\;\mbox{and} \;\; \mathbb P \Gamma=\mathbb P \Gamma (1-\mathbb P).$$
	
\noindent \textbf{Step one:  $k=0$. }
We set $V=(1- \mathbb P) U$ and $ W=\mathbb P U$.
Then, because $(1-\mathbb P)\Gamma=0$,   
$$i\partial_t V  - Q(t)  V=  (1- \mathbb P  )M(t) (V+W) $$
and
$$ i\partial_t W  - Q(t)  W= i (t-t^\flat)^{-1} \mathbb P\Gamma V +  \mathbb P M(t) (V+W).$$
We then introduce the variable 
$$\tilde V= W- \ln |t-t^\flat|  \mathbb P\Gamma V $$ so that 
$\tilde V$ satisfies 
\begin{align*}
i\partial_t \tilde V  - Q(t)  \tilde V
&= \mathbb P M(t) (V+W) -   \ln |t-t^\flat|  \mathbb P\Gamma (i\partial_t V  - Q(t)  V)  \\
&= \mathbb P M(t) (V+W) -   \ln |t-t^\flat|  \mathbb P\Gamma (1- \mathbb P  )M(t) (V+W)\\
&=(\mathbb P  -   \ln |t-t^\flat|  \mathbb P\Gamma ) M(t)  (V+\tilde V +\ln |t-t^\flat|  \mathbb P\Gamma V ).
\end{align*}
To conclude, $V$ and $\tilde V$ satisfy the system
\begin{equation}\label{eq:VtildeV}
\left\{ \begin{array}{rcl}
i\partial_t V  - Q(t)  V&=&A(t) V+ B(t) \tilde V   \\
i\partial_t \tilde V  - Q(t)  \tilde V&= & \tilde A (t) V+\tilde B(t)\tilde V   ,
\end{array}
\right.
\end{equation}
with $t\mapsto A(t), B(t), \tilde A(t), \tilde B(t) $  are smooth on $[t_0, t^\flat)$ and integrable on $[t_0, t^\flat]$. The change of unknown has contributed to improve the integrability of the functions of the right-hand side of the system. It allows us to conclude thanks to an energy estimate and Gr\"onwall lemma. As a consequence, there exists a constant $C>0$ such that 
$$\forall t\in [t_0, t^\flat),\;\;\| V(t) \|_{L^2} + \| \tilde V(t) \|_{L^2} \leq C.$$
Since we can write $$U=V+W=V+\tilde{V}+ \ln |t-t^\flat|\mathbb P\Gamma V, $$
this implies the existence of $C_1>0$ such that for all $t\in [t_0, t^\flat)$
\begin{align*}
\| U(t) \|_{L^2} & \leq \;\; \| V(t) \|_{L^2}+\| \tilde{V}(t) \|_{L^2}+ \| \ln |t-t^\flat|\mathbb P\Gamma V(t) \|_{L^2}\\
& \leq \; \; C_1\, \left(1+ \left| \ln |t-t^\flat| \right| \right).
\end{align*}
 
 \smallskip

\noindent  \textbf{Step two: $k=1$.} In view of~\eqref{com:Q(t)},
  the quantities 
  $$ y_j V,  \;\; y_j \tilde V,\;\; D_{y_j} V, \;\; D_{y_j} \tilde V,\;\;1\leq j\leq d$$
    satisfy a closed system of  equations of the form 
\begin{align*}
& i\partial_t  (y_j V)- Q(t) (y_j V)= A(t)  (y_j V )+ B(t)( y_j \tilde V)  +i D_{ y_j} V,\\
&i\partial_t (y_j \tilde V)   - Q(t)  (y_j \tilde V) =\tilde A(t) (y_j  V)+ \tilde B(t) (y_j  \tilde V)  +i D_{ y_j} \tilde V,\\
 &i\partial_t  (D_{y_j} V)- Q(t)  (D_{y_j} V) = A(t)  (D_{y_j} V) + B(t)(D_{y_j}  \tilde V) +C(t) \cdot y\,  V\\
&\qquad\qquad\qquad 
{+i |t-t^\flat|^{-1}(e_j\cdot \Gamma_0 y)  \, V,}\\
&i\partial_t(D_{y_j} \tilde  V) - Q(t)  (D_{y_j} \tilde  V) = \tilde A (t)D_{y_j}  V + \tilde B(t) (D_{y_j} \tilde  V) + \tilde C(t) \cdot y \tilde V\\
&\qquad\qquad\qquad 
{+i |t-t^\flat|^{-1}(e_j\cdot \Gamma_0 y)  \,  \tilde V},
\end{align*}
where $A(t), \tilde A(t), B(t), \tilde B(t), C(t)$ and $\tilde C(t)$ are smooth maps.  Again, this system presents the non-integrable singularity $|t-t^\flat|^{-1}$ in the right-hand side that calls for a change of unknown, as we previously did. 
 We write $V_1=V\in \C^{d}$, $\tilde V_1= \tilde V\in\C^{d}$ and consider the derivatives and momenta of $V_1$ and $\tilde V_1$. We set 
 $$V_2= (y_1 V, \cdots , y_d V, y_1 \tilde V,\cdots , y_d \tilde V)$$
 and 
 \begin{align*}
 \tilde V_2=
&  ((D_{y_j} V + \ln|t-t^\flat|  (e_j\cdot \Gamma_0 y) V)_{1\leq j\leq d}, 
( D_{y_j} \tilde V  + \ln|t-t^\flat| (e_j\cdot \Gamma_0 y)\tilde V)_{1\leq j\leq d}),
\end{align*}
where $(e_j)_j$ is the canonical basis of $\R^d$.
We have: $V_2,\tilde V_2\in\C^{(2d)^2}$ and 
 the functions $t\mapsto V_2(t), \tilde V_2(t)$   satisfy a system of the form 
 \begin{equation*}
\left\{ \begin{array}{ll}
i\partial_t V_2  - Q(t)  V_2&=A_2(t) V_2+ B_2(t) \tilde V_2   \\
i\partial_t \tilde V_2  - Q(t)  \tilde V_2&=\tilde A_2(t) V_2+\tilde B_2(t) \tilde V_2
\end{array}
\right.
\end{equation*}
 with $A_2(t),B_2(t), \tilde A_2(t), \tilde B_2(t) $ are integrable.

 Arguing as above by using an energy estimate and Gr\"onwall lemma, together with the control established for $V_1, \tilde V_1$, we obtain a control of the $L^2-$norm of $(V_2(t), \tilde V_2(t))$ of the form 
  $$\| V_2(t) \|_{L^2(\R^d,\C^{(2d)^2} )}+ \| \tilde V_2(t) \|_{L^2(\R^d,\C^{(2d)^2})} \leq C'_2.$$
We then write
\begin{align*}
  \| U(t)\|_{\Sigma^1} & \leq \|U(t)\|_{L^2} + \|V_2(t)\|_{L^2} + \|\tilde V_2(t)\|_{L^2} + \|\ln|t-t^\flat|(e_j\cdot \Gamma_0 y)V_1(t)\|_{L^2} \\
  &\quad+ \|\ln|t-t^\flat|(e_j\cdot \Gamma_0 y)\tilde V_1(t)\|_{L^2}  \\
  & \leq C_2 \left(1+\left|  \ln|t-t^\flat|\right|\right),
\end{align*}
where we have noticed that $\|(e_j\cdot \Gamma_0 y)V_1(t)\|_{L^2}$ is controlled by $\|V_2(t)\|_{L^2}$, and the same holds with the \textit{tilda-}term. 
  \smallskip

\noindent\textbf{Step three: from  $k$ to $k+1$.} At the  $(k-1)$-th step, we are left with a   vector
  $$(V_k(t), \tilde V_k(t))\in \C^{(2d)^{k}}$$
satisfying a system of the form 
 \begin{equation*}
\left\{ \begin{array}{ccc}
i\partial_t V_k  - Q(t)  V_k&=&A_k(t) V_k+ B_k(t) \tilde V_k   \\
i\partial_t \tilde V_k  - Q(t)  \tilde V_k&=& \tilde A_k(t) V_k+ \tilde B_k (t) \tilde V_k 
\end{array}
\right.
\end{equation*}
 with $A_k(t),B_k(t), \tilde A_k(t), \tilde B_k(t) $ are integrable. 
 This leads to the construction of a vectors of $(2d)^k  = d (2d)^{k-1} + d(2d)^{k-1}$ variables. Re-organizing the equation in order to cancel the  singularity generated by the commutator $[D_y, Q(t)]$:
we set 
 $$V_{k+1}= (y_1 V_{k}, \cdots , y_d V_{k }, y_1 \tilde V_{k},\cdots , y_d \tilde V_{k })$$
  \begin{align*}
&\mbox{and}\;\; \tilde V_{k +1}=
  ((D_{y_j} V_k + \ln|t-t^\flat|  (e_j\cdot \Gamma_0 y) V_k)_{1\leq j\leq d}, 
   ( D_{y_j} \tilde V_k +\ln|t-t^\flat| (e_j\cdot \Gamma_0 y) \, \tilde V_k )_{1\leq j\leq d}).
\end{align*}
One can proceed as before and one obtains the boundedness of $(V_{\ell+1}, \tilde V_{\ell+1})$ in $L^2$, whence the existence of $C_{k+1}, C_{k+1}'>0$ such that 
$$\| (V_k, \tilde V_k)\|_{\Sigma^1} \leq c_k \| (V_{k+1}, \tilde V_{k+1})\|_{L^2} \leq C'_{k+1},$$
which implies 
$\displaystyle{\| U\|_{\Sigma^k } \leq C_{k+1} \left(1+|\ln|t-t^\flat||\right). }$
\end{proof}

With Proposition~\ref{prop:profile}, we have a precise information on the behavior of the~$\Sigma^k-$norms of the solutions to the system~\eqref{def:profile}. This allows to characterize their behaviors on the crossing set and to solve the equation~\eqref{def:profile} after the crossing time. This is the subject of  
Corollary~\ref{cor:profile} that we now prove. 

\begin{proof}[Proof of Corollary~\ref{cor:profile}]
Let us assume $t<t^\flat$ and set
$$v_\pm(t)= {\rm Exp} \left[\mp \dfrac{i}{2}\Gamma_0 y\cdot y\, \ln|t-t^\flat| \right] u_\pm(t).$$
We have
\begin{align*}
i\partial_t v_\pm(t)&=
 {\rm Exp} \left[\mp \dfrac{i}{2}\Gamma_0 y\cdot y\, \ln|t-t^\flat| \right]\times \left(i\partial_t u_\pm \mp   \dfrac{1}{2|t-t^\flat|}\Gamma_0 y\cdot y\, u_\pm\right),\\
&= {\rm Exp} \left[\mp \dfrac{i}{2}\Gamma_0 y\cdot y\, \ln|t-t^\flat| \right]\times \\
&\qquad \left(-\frac 1 2 \Delta _y u_\pm(t) + \frac 12  ({\rm Hess} \lambda_\pm(q_\pm(t) )y\cdot y) u_\pm(t) \mp \dfrac{1}{2|t-t^\flat|}\Gamma_0 y\cdot y\,u_\pm\right)\\
& = {\rm Exp} \left[\mp \dfrac{i}{2}\Gamma_0 y\cdot y\, \ln|t-t^\flat| \right]\times \left(-\frac 1 2 \Delta _y u_\pm(t) + \frac 12 M_\pm(t)y\cdot y u_\pm\right)
\end{align*}
where the matrix $M_\pm(t)$ is defined in Lemma~\ref{lemma_Mpm} and is smooth on  $[t_0,t^\flat]$ (the term  $\pm(t-t^\flat)^{-1} \Gamma_0 y\cdot y$ compensates for the singularity of the potential of the operator $Q(t)$ (see~\eqref{def:Q}). We now use   Proposition~\ref{prop:profile}. 
Therefore, for all $t\in[t_0,t^\flat)$, $\partial_t v_\pm(t)\in\Sigma^k$ for all $k\in\N$. Besides, for each $k\in\N$, in view of the control \eqref{est:profileHk}, there exist constants $C_k, \tilde C_k>0$ and $N_k, \tilde N_k\in\N$  such that 
\[
\|\partial_t  v_\pm(t)\|_{\Sigma^k} \leq C_k \left(1+\left|\ln |t-t^\flat| \right|\right)^{N_k} \| u_\pm(t)\|_{\Sigma^{k+2}}
\leq \widetilde C_k \left(1+\left| \ln |t-t^\flat|\right|\right)^{\tilde N_k} 
\]
We deduce that $\int_{t_0}^{t^\flat} \partial_t v_\pm(s) ds$ is well-defined as a function of $\Sigma^k$ and we denote by $u_\pm^{\rm in}$ this function that satisfies~\eqref{eq:profile_in}.  

\smallskip

We now want to use $u^{\rm in}_\pm$ as an initial data at time $t^\flat$.
We observe that the function $v_\pm(t)$ solves an equation of the form 
	\begin{equation}\label{eq:Mtilde}
	i\partial_t v_\pm= H(t) v_\pm
	\end{equation}
	with
	\begin{align*}
	H(t)&=-\frac 12 \Delta \pm a(t) y\cdot D_y \pm c(t) +  b(t) y\cdot y ,\\
	a(t)&= \Gamma_0 \ln |t-t^\flat|, \;\;
	c(t)= - \frac i 2 {\rm  tr}(\Gamma_0)\ln |t-t^\flat | \\
	b(t)y\cdot y &= \frac 12 M_\pm(t)y\cdot y +
	\frac 12  (\ln |t-t^\flat|)^2 |\Gamma_0 y |^2. 
	\end{align*}
	Note that 
	$$a(t)y\cdot D_y +c(t) = \frac 12 (a(t) y\cdot D_y) + \frac 12 (a(t) y\cdot D_y)^*.$$
The operator $H(t)$ is a self-adjoint quadratic operator with time-integrable coefficients to which we can associate a two-parameters
	propagator 
	$\widetilde {\mathcal U}(t,s)$  defined for $t,s\in [t_0, t^\flat)$ (see~\cite{MaRo}). our aim is to construct $ \mathcal U(s,t^\flat)$.
We use the following facts:

\begin{enumerate}
	\item It is equivalent to say that $u_\pm(t)$ solves~\eqref{def:profile} and to say that $v_\pm(t)$ solves~\eqref{eq:Mtilde}.
	\item There is conservation of the $L^2$-norm and
	$$\| v_\pm(t)\|_{L^2}= \| v_{\pm}(t_0)\|_{L^2}= \|u_\pm(t_0)\|_{L^2}.$$
	\item When $t$ tends to $t^\flat$, $\widetilde {\mathcal U}(t,s) u_\pm(t_0)$ has a limit $u^{\rm in} $ with $\|u_\pm(t_0)\|_{L^2}=\|u_\pm^{\rm in}\|_{L^2}$.
	Let us denote by	$\widetilde {\mathcal U}(t^\flat,s)$ the operator mapping $ u_\pm(t_0)$ to $u_\pm^{\rm in}.$
	\item For all $f\in \mathcal S(\R^d)$, $k\in\N$ there exists $C_k>0$ such that
	$$\forall f\in \mathcal S(\R^d),\;\;\| \widetilde {\mathcal U}(t^\flat,s) f\|_{\Sigma^k} \leq C_k \|f\|_{\Sigma^{k+3}}.$$
\end{enumerate}
We claim that for $t,s\in[t_0, t^\flat)$ we have $\widetilde {\mathcal U}(s,t)= \widetilde {\mathcal U}(t,s)^*$, which allows
to  define the operator~$\widetilde {\mathcal U}(s,t^\flat)$ by
$$
\widetilde {\mathcal U}(s,t^\flat):= \widetilde {\mathcal U}(t^\flat,s)^*.
$$
Indeed,  from the definition of $\widetilde {\mathcal U}(t,s)$ as solving
\begin{equation}\label{totoU(s,t)}
i\partial_t \widetilde {\mathcal U}(t,s)= H(t) \widetilde {\mathcal U}(t,s),\;\;\widetilde {\mathcal U}(s,s)={\rm Id}_{\R^2},
\end{equation}
we deduce on one hand, that
$$i\partial_t \widetilde {\mathcal U}(t,s)^*= - \widetilde {\mathcal U}(t,s)^* H(t),\;\;\widetilde{ \mathcal U}(s,s)={\rm Id}_{\R^2} ,$$
and on the other hand, differentiating in $s$ the relation~\eqref{totoU(s,t)}, we obtain
that $V(t,s)=\partial_s\widetilde {\mathcal U}(t,s)$ satisfies
$$i\partial_t V(t,s) =H(t) V(t,s),\;\; V(s,s)= - \partial_t\widetilde {\mathcal U}(s,s)= i H(s) .$$
Therefore, $V(t,s)=\widetilde {\mathcal U}(t,s)iH(s)$, which gives $ i\partial_s\widetilde {\mathcal U}(t,s) =-\widetilde {\mathcal U}(t,s)H(s)$. Exchanging the roles of $t$ and $s$ we obtain that $\widetilde
{\mathcal U}(s,t)$ solves the same equation as $\widetilde {\mathcal U}(t,s)^* $ with the same initial data and thus, they are equal.

\smallskip

Therefore, we have proved that we can build a function
$u^\pm(t)$ solving~\eqref{def:profile} for $t\leq t^\flat$, starting from a profile $u_\pm^{\rm in}$ on $t^\flat$ with enough regularity, in particular for $u_\pm^{\rm in}\in \mathcal S(\R^d)$.
\smallskip
Arguing in a similar way in the zone $t>t^\flat$, we deduce that there exists a unique solution to ~\eqref{def:profile}
satisfying~\eqref{eq:profout} for some given $u_\pm^{\rm out}\in\mathcal S(\R^d)$.
\end{proof}


\section{Adiabatic transport outside the gap region}\label{sec:adiab}

This section is inspired by~\cite{FLR} and discussions with Caroline Lasser and Didier Robert. We focus here on zones that are far enough from the gap region in the sense that $|w(x)|>\delta$, along the trajectories concerned by the process.
 In this adiabatic region, we prove the following  result showing that one can approximate the solution of the system~\eqref{system} by solutions  of  scalar type equations. 

\begin{proposition} \label{prop:propagationt_out}
Let $k\in\N$ and $\delta\in (0,1)$ such that $\sqrt\eps\delta^{-1} \ll 1$. 
Consider $s_1,s_2\in\R $, $s_1< s_2$  and two classical  trajectories $z_\pm(t))_{t\in [s_1,s_2]}$ that reach the crossing set~$\Upsilon$ at time $t^\flat$ at a point where Assumptions~\ref{hypothesis} are satisfied. We assume 
 $[s_1,s_2]\subset \{|t-t^\flat|>\delta\}$ and that 
 at initial time $s_1$, 
$$\left\|\psi^\eps(s_1)-{ \vec Y_+(s_1)} v^\eps_+(s_1) - {\vec Y_-(s_1)} v^\eps_-(s_1) \right\|_{\Sigma^k_\eps}\leq C\sqrt \eps,$$
$$v^\eps_\pm(s_1)=
{\rm WP}^\eps_{z_\pm(s_1)} (u_\pm(s_1)),\;\; u_\pm(s_1)\in{\mathcal S}(\R^d),\;\; z_\pm(s_1)=(q_\pm(s_1), p_\pm(s_1))\in\R^{2d},$$
and 
with $\Pi_\pm (q_\pm(s_1)) \vec Y_\pm(s_1)= \vec Y_\pm(s_1)$.
 Then, for all $k\in\N$, one has
$$\sup_{t\in [s_1,s_2]} \left\|  \Pi_\pm \psi^\eps(t) -{\vec Y_\pm(t)} v^\eps_\pm(t)
 \right\| _{\Sigma^k_\eps} \leq C_k (1+ | \ln \delta| )\left( {\eps^{3/2} \over \delta^4}  +{ \sqrt \eps\over\delta} \right),$$
where the constant $C_k$ is uniform in $\delta$ and $\eps$, and  for $t\in[s_1,s_2]$  
\begin{itemize}
\item the functions $v^\eps_\pm(t)$ are wave packets:
\begin{equation}\label{def:veps}
 v^\eps_\pm(t)= {\rm e}^{\frac i\eps S_\pm(t)} {\rm WP}_{z_\pm(t)}^\eps \left(u_\pm(t) \right),
 \end{equation}
\item the trajectory $z_\pm(t)$ is the classical trajectory $z_\pm(t)= \Phi^{t,s_1}_\pm(z_\pm(s_1))$ and $S_\pm(t)$ is the related action   $S_\pm(t)= S_\pm(t,s_1,z_\pm(s_1))$(see~\eqref{def:S}),
\item the functions $u_\pm(t)$ satisfy~\eqref{def:profile} with data $u_\pm(s_1)$ at time $s_1$ and their norms in spaces $\Sigma^k$ satisfy~\eqref{est:profileHk},
\item the vectors $\vec Y_\pm(t)$ are defined in Section \ref{sec:eigenvec}
 and satisfy $\Pi_\pm(z_\pm(t))\vec Y_\pm(t)=\vec Y_\pm(t)$, together with  $\partial_t \vec Y_\pm(t) = B_\pm(z_\pm(t)) \vec Y_\pm(t)$.
\end{itemize}
\end{proposition}

Note first that, by the results of Section~\ref{sec:approxsol},  all the quantities involved in Proposition~\ref{prop:propagationt_out}  are well defined for $t\in [s_1,s_2]$. Besides, the solution at time~$t\in[s_1,s_2]$ on each mode only depends on the data on the same mode at time~$s_1$. This is the reason why one may say that the approximation is of ``scalar type'' as mentioned before. 

\smallskip

Note also that the assumptions of Proposition~\ref{prop:propagationt_out} imply that there exists $c>0$ such that
$$\forall t\in[s_1,s_2],\;\;|w(z_\pm(t))|>c \delta.$$

\smallskip

In the proof of Theorem~\ref{theo:main}, we will use Proposition~\ref{prop:propagationt_out} twice: first between $s_1=t_0$ and $s_2=t^\flat-\delta$ with $u_+(t_0)=0$ and $u_-(t_0)=a$, then, between $s_1=: t^\flat+\delta$ and $s_2$ equal to some final time $t$ with the profiles $u_\pm(t^\flat+\delta)$ arising from the process of passing through the crossing.

\smallskip

For proving Proposition~\ref{prop:propagationt_out}, we use  the semi-classical formalism of Appendix~\ref{app:pseudo} and the pseudo\-differential operators introduced therein: with $a\in{\mathcal C}^\infty(\R^{2d}, \C^N)$ ($N=1$ or $2$), we associate the operator ${\rm op}_\eps (a)$ defined by~\eqref{def:oppseudo}. We shall use the matrices $\mathbb P$, $\mathbb P_\pm^{(2)}$, $\Omega$ and $\Omega_\pm^{(2)}$ of Section~\ref{sec:superadiab}.
We work close enough to the crossing time $t^\flat$ so that the curves $z_\pm(t))$ are included in $\{w(x)\cdot \omega\not=0\}$ for all $t\in[s_1,s_2]$. Indeed, 
far from $t^\flat$, the proof is easier since one does not see the singularities of the involved quantities. 
The proof is divided into two steps: we first identify an approximate solution satisfied by an auxiliary ansatz that is close to the function ${\vec Y_\pm(t)} v^\eps_\pm(t)$ (Lemma~\ref{lem:ansatz} in Section~\ref{subsec:eqansatz}), then we prove that $ \Pi_\pm \psi^\eps(t)$ (up to some remainder) satisfies the same equation (Section~\ref{subsec:projected}).

\subsection{The adiabatic ansatz}\label{subsec:eqansatz}
For proving Proposition~\ref{prop:propagationt_out}, we first
introduce cut-off functions that allow us to restrict the analysis close to the trajectories, where the functions $\lambda_\pm$ and related quantities are smooth. 
Let $I$ be an  interval containing $[s_1,s_2]$. We construct  $\chi_\pm^\delta \in\mathcal C(I,{\mathcal C}_0^\infty(\R^{2d}))$, compactly supported in 
$\{|w(x)|>\delta\}$, equal to $1$ close to the curve $(z_\pm(t))_{t\in[s_1,s_2]}$ and satisfying
\begin{equation}
\label{eq:chidelta}
\partial_t \chi_\pm^\delta + \left\{\frac {|\xi|^2}{2} + \lambda_\pm, \chi_\pm^\delta \right\}=0.
\end{equation}
\begin{remark}\label{rem:chidelta}
	Let $s_1, s_2 $ as in Proposition~\ref{prop:propagationt_out}. 
	The functions $\chi_\pm^\delta$ can be taken for $t\in [s_1,s_2]$ as 
	$$\chi_\pm^\delta(t,x,\xi)= \chi\left( \frac {\Phi^{t,s_2}_\pm(x,\xi)-z_\pm(s_2)}{\delta}\right)$$
	where $0\leq \chi\leq 1$ with $\chi=1$ close to $0$ and $\chi=0$ far from $0$.
\end{remark}

We also introduce 
	$\tilde \chi_\pm^\delta\in \mathcal{C}(I, {\mathcal C}_0^\infty(\R^{2d}))$ compactly supported, such that for all $t\in[s_1,s_2]$, we have $\tilde \chi_\pm^\delta(t)=1$ on
	${\rm supp}\ \chi_\pm^\delta (t)$.
	We have 
	\newline
	\begin{equation}
	\label{est:chidelta} 
	\textrm{for } \alpha\in\N^{2d}, \; \; 
	\partial^\alpha \chi_\pm^\delta =\mathcal{O}(\delta^{-|\alpha|}) \; \; ; \; \; 
		\partial^\alpha \tilde \chi_\pm^\delta =\mathcal{O}(\delta^{-|\alpha|})
	\end{equation}
\smallskip

\noindent {\bf Step one: reduction to an auxiliary ansatz.}  Let $\vec V_\pm$ be the smooth functions defined in (3) of Proposition~\ref{prop:ingoeigen}, that is a smooth eigenvector of the matrix $V(x)$ satisfying $\vec Y_\pm(t)=\vec V_\pm(q_\pm(t))$. 

\begin{lemma}\label{lem:auxansatz}
Let $\delta\in (0,1]$ be such that $\sqrt\eps \delta^{-1}\ll 1$. Then, we have for $t\in[s_1,s_2]$, 
$$\vec Y_\pm(q_\pm(t)) v^\eps_\pm(t)= {\rm op}_\eps\left( \chi^\delta_\pm(t,x,\xi)\vec V_\pm(x)\right) v^\eps_\pm(t)
 + \mathcal{O}\left(\sqrt{\eps}\, \delta^{-1}(1+ |\ln \delta |)  \right) .
 $$
\end{lemma}

\begin{proof} The proof relies on  the application of Lemma~\ref{lem:localisation} with $n_0=0$ to the symbol  $a(t,x,\xi)=\chi^\delta_\pm(t,x,\xi)\vec V_\pm(x)$, which requires the computation of the semi-norms 
	$$
	N_{d+k+1}^\eps(\partial_{z_j}a)=\sum_{\alpha\in\N^{2d},|\alpha|\leq n_{d+k+1}} \eps^{|\alpha|\over 2}\sup_{\R^{2d}}|\partial_{z}^\alpha \partial_{z_j}a|, \;\; 1\leq j\leq 2d,
	$$
	where for $\ell\in \N$, $n_\ell = M \ell$, $M\geq 1$. With $a=\chi^\delta_\pm\vec V_\pm$ and in view of Lemma~\ref{eigenvectors} and equation~\eqref{est:chidelta},  we obtain  
\begin{align*}
\sqrt{\eps}N_{d+k+1}^\eps(\partial_{z_j} a) 
& \leq C\sqrt{\eps} \left( \sum_{|\alpha|\leq n_{d+k+1}} \eps^{|\alpha|/2} \sup_{\R^{2d}} \left|\partial ^\alpha_z ((\partial_{z_j} \chi^\delta_\pm)\vec V_\pm)\right| + \sum_{|\alpha|\leq n_{d+k+1}} \eps^{|\alpha|/2} \sup_{\R^{2d}} \left|\partial^\alpha_z(\chi^\delta_\pm \partial_{z_j} \vec V_\pm)\right| \right)\\
& \leq C'\sqrt{\eps} \sum_{|\alpha|\leq n_{d+k+1}}\eps^{|\alpha|/2} \delta^{-|\alpha|-1} = C'  \sum_{|\alpha|\leq n_{d+k+1}}(\sqrt\eps \delta^{-1})^{|\alpha|+1}=O( \sqrt\eps \delta^{-1})
\end{align*}
since  $\sqrt{\eps}\delta^{-1}\ll 1$. 

	\smallbreak
	
We can now see, thanks to Lemma \ref{lem:localisation} and the previous computation, that we have in $\Sigma_\eps^k$, for some integer $N'$
\begin{align*}
&{\rm op}_\eps\left( \chi^\delta_\pm(t,x,\xi)\vec V_\pm(x)\right) v^\eps_\pm(t)  
=  {\rm e}^{\frac i\eps S_\pm(t)}{\rm op}_\eps\left( \chi^\delta_\pm(t,x,\xi)\vec V_\pm(x)\right) {\rm WP}_{z_\pm(t)}^\eps \left(u_\pm(t) \right) \\
&\qquad = {\rm e}^{\frac i\eps S_\pm(t)} {\rm WP}_{z_\pm(t)}^\eps  \left(\chi^\delta_\pm(t,z_\pm(t)) \vec V_\pm(q_\pm(t))\right) u_\pm(t)    + \mathcal{O}\left(\sqrt{\eps}\delta^{-1} \|u_\pm(t)\|_{\Sigma^{N'}} \right)
\\
&\qquad  =\vec Y_\pm(t)  {\rm e}^{\frac i\eps S_\pm(t)} {\rm WP}_{z_\pm(t)}^\eps u_\pm(t)   
 + \mathcal{O}\left(\sqrt{\eps} \delta^{-1} (1+ |\ln \delta |)  \right) \\
& \qquad =\vec Y_\pm(t) v^\eps_\pm(t) + \mathcal{O}\left(\sqrt{\eps} \delta^{-1}(1+ |\ln \delta |)  \right) 
\end{align*} 
where we have used to definition of $\chi^\delta_\pm$, the estimation on the profiles \eqref{est:profileHk}, $\chi^\delta_\pm(t,z_\pm(t))=1$ and Lemma \ref{lem:prelimWP}.
\end{proof}

\noindent{\bf Step two: Analysis of the ansatz.} We now study the properties of the ansatz
\begin{equation}\label{def:psiepsapp}
\psi^\eps_{\pm, \rm app} (t)={\rm op}_\eps\left( \chi^\delta_\pm(t,x,\xi)\vec V_\pm(x)\right) v^\eps_\pm(t).
\end{equation}
We  analyze the equations satisfied by $\psi^\eps_{\pm,\rm app}$ and use the notations of Section~\ref{sec:superadiab}.

\begin{lemma}\label{lem:ansatz}
Let $k\in\N$ and $\delta\in (0,1]$ be such that $\sqrt\eps \delta^{-1}\ll 1$.
	With the notations of Proposition~\ref{prop:propagationt_out} and Equation~\eqref{def:psiepsapp}, for $t\in [s_1,s_2]$, we have 
		\begin{align*}
	i\eps \partial_t \psi^\eps_{\pm, \rm app}& = -\frac{\eps^2}{2}\Delta \psi^\eps_{\pm,\rm app} + \lambda_\pm(x) \psi^\eps_{\pm,\rm app}
	+ \eps \, {\rm op}_\eps (\Omega \tilde \chi^\delta_\pm)  \psi^\eps_{\pm,\rm app} +\eps^2 \,{\rm op}_\eps(\Omega^{(2)}_\pm \tilde \chi^\delta_\pm) \psi^\eps_{\pm,\rm app}\\
	&\qquad 
	+\mathcal{O}\left( (\eps^{3/2} + \eps^{2}\delta^{-2} + \eps^{5/2}\delta^{-4})(1+\left|  \ln\delta \right|)\right)
	\end{align*}
	in $\Sigma^k_\eps$, where $\Omega_\pm^{(2)}$ is given in~\eqref{def:Omega2} and $\Omega$ is the self-adjoint matrix
	\begin{align}\label{def:Omega}
	\Omega& = i(B_++ B_-)
	= i(\Pi_-\xi\cdot \nabla\Pi_+ \Pi_+ -\Pi_+\xi\cdot\nabla \Pi_+\Pi_-) \\
	\nonumber
	&= - \frac i{2|w(x)|} \xi\cdot \nabla w(x) \wedge \frac {w(x)}{|w(x)|}  
	\begin{pmatrix} 0 & 1 \\ -1 & 0\end{pmatrix}.
	\end{align} 
\end{lemma} 

We recall that the matrices $B_\pm$ are defined in~\eqref{def:B+-} and we point out that  
$\Omega$ is self-adjoint because 
$$\Omega^*=-i(B_+^*+B_-^*)=i(B_++B_-)=\Omega.$$
Moreover, by~\eqref{def:B+-} and~\eqref{bound:projector}, the operator ${\rm op}_\eps(\Omega)$ is a differential operator of order~$1$ with matrix-valued coefficients that are growing polynomially at infinity and are singular on~$\Upsilon$. The various expressions of the matrix $\Omega$ are proved in Lemma~\ref{lem:computation}. 

\begin{remark}\label{rem:alphachoice}
	We shall use $\delta=\eps^\alpha$ with $3/2-4\alpha>0$, that is $\alpha \leq 3/8$. We shall see in the next section that the analysis requires $\delta^3 \eps^{1} \ll 1$ (see Remark~\ref{rem:constraint2}), which is possible since one has $1/3<3/8$.  Besides, $\eps^2\delta^{-2} \ll \eps^{5/2}\delta^{-4}$ as soon as $\alpha>1/4$, which is satisfied when $\alpha\in (1/3,3/8)$. An optimal choice of $\delta$ will then consist in choosing $\delta =\eps^{\frac 5{14}}$, leading to $\eps^{3/2} \delta^{-4} = \delta^3 \eps^{-1}= \eps^{\frac 1{14}}$ (and of course $\sqrt\eps \ll \eps^{5/14}$). 
\end{remark}

\begin{proof}
		We begin by considering for  $t\in [s_1,s_2]$ the family  
		$(v^\eps_\pm(t))$ defined in~\eqref{def:veps}.
		It comes from a computation (see~\cite{CR} for example) that $v^\eps_\pm(t)$ solves in $\Sigma^k$,
		\begin{equation}\label{eqveps(t)}
		i\eps \partial_t v_\pm^\eps(t)= -\frac{\eps^2}{2}\Delta v^\eps_\pm(t) +\lambda_\pm(x) v^\eps_\pm(t) +\mathcal{O}(\eps^{3/2}\| u_\pm(t)\|_{\Sigma^{k+3}}).
		\end{equation}
		Since the profiles satisfy~\eqref{est:profileHk},   we have 
		$$\mathcal{O}( \eps^{3/2}\| u_\pm\|_{\Sigma^{k+3}})
		=\mathcal{O}(\eps^{3/2} (1+|\ln \delta|)).$$
Considering $\psi^\eps_{\pm,\rm app}$, we write in $\Sigma^k_\eps$,
\begin{equation}\label{2terms}
i\eps \partial_t \psi^\eps_{\pm,\rm app}  = {\rm op}_\eps (i \eps \partial_t \chi^\delta_\pm (t)\vec V_\pm )v^\eps_\pm  + {\rm op}_\eps\left( \chi^\delta_\pm(t) \vec V_\pm\right) i \eps \partial_t v^\eps_\pm.\end{equation}
Using \eqref{eq:chidelta} in the first term of~\eqref{2terms}, we get
\[
{\rm op}_\eps (i \eps \partial_t \chi^\delta_\pm(t) \vec V_\pm )v^\eps_\pm  = -i\eps {\rm op}_\eps \left(  \left\lbrace \dfrac{|\xi|^2}{2}+\lambda_\pm ,\chi^\delta_\pm(t)\right\rbrace \vec V_\pm \right)v^\eps_\pm.
\]
Writing $\displaystyle{
 \left\lbrace \dfrac{|\xi|^2}{2}+\lambda_\pm(x) ,\chi^\delta_\pm(t,z)\right\rbrace \vec V_\pm(x)
 = \left\lbrace \dfrac{|\xi|^2}{2}+\lambda_\pm(x) ,\chi^\delta_\pm(t,z) \vec V_\pm(x) \right\rbrace + \xi \cdot \nabla_x \vec V_\pm(x)  \chi^\delta_\pm(t,z)}$,
and using   \eqref{xi.nablaY} together with $\Pi_\mp \vec V_\pm =0$, we deduce 
$${\rm op}_\eps (i \eps \partial_t \chi^\delta_\pm(t) \vec V_\pm )v^\eps_\pm =-i\eps {\rm op}_\eps \left(
  \left\lbrace \dfrac{|\xi|^2}{2}+\lambda_\pm,\chi^\delta_\pm(t) \vec V_\pm \right\rbrace\right)v^\eps_\pm + \eps\, {\rm op}_\eps \left( \Omega \vec V_\pm\chi^\delta_\pm(t)\right)v^\eps_\pm ,
.$$
On the other hand, using~\eqref{eqveps(t)}, the second term of~\eqref{2terms} can be handled as
\begin{align*}
{\rm op}_\eps\left( \chi^\delta_\pm(t)\vec V_\pm\right) i \eps \partial_t v^\eps_\pm &=\, 
  {\rm op}_\eps\left(\dfrac{|\xi|^2}{2} + \lambda_\pm \right) {\rm op}_\eps\left( \chi^\delta_\pm(t) \vec V_\pm\right)v^\eps _\pm\\
 &\qquad - \left[{\rm op}_\eps \left(\dfrac{|\xi|^2}{2} + \lambda_\pm\right) , {\rm op}_\eps \left(\chi^\delta_\pm(t)\vec V_\pm \right) \right]v^\eps_\pm 
+ \mathcal{O}(\eps^{3/2} (1+|\ln \delta|)) \\
& = {\rm op}_\eps\left(\dfrac{|\xi|^2}{2} + \lambda_\pm \right) {\rm op}_\eps\left( \chi^\delta_\pm(t)\vec V_\pm\right)v^\eps _\pm- \dfrac{\eps}{i}{\rm op}_\eps \left(\left\lbrace\dfrac{|\xi|^2}{2} + \lambda_\pm, \chi^\delta_\pm(t)\vec V_\pm \right\rbrace \right) v^\eps_\pm \\
& \qquad +  \mathcal{O}(\eps^{3/2} (1+|\ln \delta|)) + \mathcal O (\eps^2 \delta^{-2} )
\end{align*}
thanks to Proposition~\ref{prop:symbol}.

\smallskip

As a consequence of these two computations, we obtain 
\[
i\eps \partial_t \psi^\eps_{\pm,\rm app}  = 
\left(-\dfrac{\eps^2}{2}\Delta + \lambda_\pm \right)\psi^\eps_{\pm,\rm app} + \eps\, {\rm op}_\eps \left(\Omega \vec V_\pm \chi^\delta_\pm(t)\right)v^\eps_\pm  +  \mathcal{O}(\eps^{3/2} (1+|\ln \delta|)) + \mathcal O (\eps^2 \delta^{-2} (1+|\ln \delta|)) )
\]
We then  use 
\[ {\rm op}_\eps \left(\Omega \vec V_\pm \chi^\delta_\pm(t)\right)= {\rm op}_\eps \left(\Omega \tilde \chi^\delta_\pm (t)\right)  \, {\rm op}_\eps \left( \chi^\delta_\pm (t) \vec V_\pm\right)
 - \dfrac{i\eps}{2} {\rm op}_\eps \left(\left\lbrace \Omega \tilde \chi^\delta_\pm(t) , \vec V_\pm \chi^\delta_\pm(t) \right \rbrace\right) 
 + \mathcal O(\eps^3\delta^{-5}).
 \]
 Therefore, 
 \begin{align*}
 i\eps \partial_t \psi^\eps_{\pm,\rm app}  &= 
\left(-\dfrac{\eps^2}{2}\Delta + \lambda_\pm + \eps {\rm op}_\eps \left(\Omega \tilde \chi^\delta_\pm(t) \right) 
+ \eps^2 \,{\rm op}_\eps(\Omega^{(2)}_\pm \tilde \chi^\delta_\pm(t)) \right)
\psi^\eps_{\pm,\rm app} \\
&  \;  - \eps^2 \,{\rm op}_\eps(\Omega^{(2)}_\pm \tilde \chi^\delta_\pm(t)) {\rm op}_\eps \left(\vec V_\pm \chi_\pm^\delta (t)\right)v^\eps_\pm 
- \dfrac{i\eps^2}{2} {\rm op}_\eps \left(\left\lbrace \Omega \tilde \chi^\delta_\pm (t), \vec V_\pm \chi^\delta_\pm (t)\right \rbrace\right)  v^\eps_\pm\\
&\;\;+
  \mathcal{O}(\eps^{3/2} (1+|\ln \delta|)) + \mathcal O (\eps^2 \delta^{-2} (1+|\ln \delta|)) ) + \mathcal O(\eps^4\delta^{-5} (1+|\ln \delta|)) )
\end{align*}
To handle the last terms, we rely on Proposition \ref{prop:symbol}, Remark \ref{rem:C3} and estimates \eqref{est:chidelta}, \eqref{eq:eigencontrol}, together with \eqref{est:profileHk}. We   write in $\Sigma^k_\eps$
	$${\rm op}_\eps(\Omega^{(2)}_\pm \tilde \chi^\delta_\pm(t)) \, {\rm op}_\eps (\chi^\delta_\pm (t)\vec V_\pm)v^\eps_\pm = {\rm op}_\eps(\Omega^{(2)}_\pm    \chi^\delta_\pm(t) \vec V_\pm)v^\eps_\pm + \mathcal{O}\left({\eps}\delta^{-5}(1+ | \ln \delta|)\right). $$
We finally use Lemma \ref{lem:localisation}, noticing that $\chi^\delta_\pm, \tilde \chi^\delta_\pm$ are equal to one close to the curve $z_\pm(t)$ and write in $\Sigma^k_\eps$
\begin{align*}
{\rm op}_\eps(\Omega_\pm^{(2)} \chi^\delta_\pm(t) \vec V_\pm)v^\eps_\pm & =   \Omega_\pm^{(2)}(z_\pm(t)) \vec V_\pm(q_\pm(t)) v^\eps_\pm+ \mathcal{O}({\eps} ^{5/2}\delta^{-4}  \|v^\eps\|_{\Sigma_\eps^{k+1}}) \\
& = \mathcal{O}\left( \delta^{-2}(1+ |\ln \delta |)\right) +  \mathcal{O}\left( \eps ^{1/2}\delta^{-4}(1+ |\ln \delta |)\right)
\end{align*}
thanks to Lemma \ref{lem:Omega2}.
We treat the term ${\rm op}_\eps \left(\left\lbrace \Omega \tilde \chi^\delta_\pm (t), \vec V_\pm \chi^\delta_\pm(t)\right \rbrace\right)  v^\eps_\pm$ in a similar way.
One notices 
$$\left\lbrace \Omega \tilde \chi^\delta_\pm (t), \vec V_\pm \chi^\delta_\pm (t)\right \rbrace (z_\pm(t))= \{ \Omega, V_\pm\} (z_\pm(t)) =\nabla_\xi\Omega\cdot \nabla  V_\pm (z_\pm(t))=O(\delta^{-2})$$
because $\partial_{z_j} \chi^\delta_\pm (z_\pm(t))=\partial_{z_j}\tilde  \chi^\delta_\pm (z_\pm(t))=0$. Then  Lemma \ref{lem:localisation} gives 
\begin{align*}
{\rm op}_\eps \left(\left\lbrace \Omega \tilde \chi^\delta_\pm (t), \vec V_\pm \chi^\delta_\pm(t)\right \rbrace\right)  v^\eps_\pm =&
\left\lbrace \Omega \tilde \chi^\delta_\pm (t), \vec V_\pm \chi^\delta_\pm (t)\right \rbrace (z_\pm(t)) v^\eps_\pm \\
&\qquad
 + \mathcal O\left(\sqrt\eps(1+|\ln\delta|)N^\eps_{d+k+1} ( d(\{ \Omega \tilde \chi^\delta_\pm(t),\vec V_\pm \chi^\delta_\pm(t)\}))\right) .
\end{align*}
One has $\left\lbrace \Omega \tilde \chi^\delta_\pm (t), \vec V_\pm \chi^\delta_\pm (t)\right \rbrace (z_\pm(t))= \mathcal O(\delta^{-4} )
$, which gives 
$${\rm op}_\eps \left(\left\lbrace \Omega \tilde \chi^\delta_\pm (t), \vec V_\pm \chi^\delta_\pm(t)\right \rbrace\right)  v^\eps_\pm=\mathcal 
O\left( (\eps^{1/2} \delta^{-4}+\delta^{-2})(1+|\ln \delta|)\right).$$

One the concludes by observing that $\eps^4 \delta^{-5}\ll \eps^{5/2} \delta^{-4}$ since $\sqrt\eps \delta^{-1} \ll 1$. 

 \end{proof}

\subsection{Superadiabatic correctors of the projectors}\label{subsec:projected}
In this section, we proceed with the study of the equation satisfied by the projections of $\psi^\eps(t)$ on the modes, the functions $\Pi_\pm \psi^\eps(t)$. 
We use ideas issued from~\cite{Te,bi,N1,N2,MS}, aiming at improving the projectors $\Pi_\pm(x)$ into operators called   super\-adiabatic projectors that are pseudodifferential operators with symbols that are series in $\eps$. For our purpose, we only need the first two terms of these series. We set 
$$H(x,\xi)= \frac {|\xi|^2}{2} +V(x),\;\;
h_\pm(x,\xi)=\frac {|\xi|^2}{2} +\lambda_\pm (x),$$
and consider for $x\notin\Upsilon$, the matrices 
$\Omega(x,\xi)$ defined by \eqref{def:Omega},
$\mathbb P(x,\xi)$ given by
\begin{align}\label{def:bigP}
\mathbb P(x,\xi) & = \frac{i}{2|w(x)|} (\Pi_-(x) \xi\cdot \nabla \Pi_+(x)  -   \Pi_+(x)\xi\cdot \nabla \Pi_+(x))\\
\nonumber
&=  \dfrac{-i}{4|w(x)|^2}\; \; \xi\cdot \nabla w \wedge\frac{ w}{|w|} \begin{pmatrix} 0 & 1\\ -1 & 0\end{pmatrix}
\end{align}
and
 $\mathbb P_\pm^{(2)}$, $\Omega_\pm^{(2)}$ written in details in Section~\ref{sec:superadiab}.
The superadiabatic projectors at order~$2$ are the functions 
$$ \Pi^\eps_\pm(x,\xi)= \Pi_\pm(x) \pm \, \eps \mathbb P(x,\xi) +\eps^2 \mathbb P^{(2)}_\pm(x,\xi).$$
These matrices  are smooth outside $\Upsilon$.
From Lemma~\ref{lem:C2}, outside $\Upsilon$, we have equation~\eqref{eq:sharp}, i.e.
 \begin{equation*}
 \Pi^\eps_\pm\, \sharp_\eps \, H = (h_\pm +\eps \,\Omega^{(1)}_\pm +\eps^2\,\Omega^{(2)}_\pm)\, \sharp_\eps \, \Pi^\eps_\pm + \eps^3 R^\eps
\end{equation*}
where $R^\eps$ satisfies  the estimate~\eqref{estimate:Reps}.
Besides, Estimate \eqref{est:proj} and  Remark \ref{rem:C3} give precise information about these matrices at infinity and close to $\Upsilon$.
Because these corrected projectors may grow in the variables $x$ and $\xi$, we shall localize them by use of the cut-off functions of Section~\ref{subsec:eqansatz}. It will also allow to restrict the analysis to the zone where they are smooth. 
By construction, we have the following Lemma. 

\begin{lemma}\label{lem:superadiab}
Let $k\in\N$, $\delta\in(0,1)$ such that $\sqrt \eps\delta^{-1}\ll 1$. Then, in $\mathcal L(\Sigma^k_\eps) $, we have for all function 
$\check \chi^\delta_\pm(t)\in\mathcal C^\infty(\R^{2d})$ satisfying~\eqref{est:chidelta} and supported on $\{ \chi^\delta(t))=1\}$ for $t\in [t_0,t^\flat)$,  
$$
{\rm op}_\eps (\check \chi_\pm^\delta (t)) {\rm op}_\eps (
\Pi^\eps_\pm\chi_\pm ^\delta(t)) \left(-\frac {\eps^2} 2 \Delta + V(x)\right)= 
{\rm op}_\eps (\check \chi_\pm^\delta (t)) {\rm op}_\eps (H^\eps_{\rm adiab,\pm}) 
{\rm op}_\eps (
\Pi^\eps_\pm\chi_\pm ^\delta(t))+\mathcal{O}(  \eps^3\delta^{-5}).
$$
with
$$H^\eps_{\rm adiab,\pm}(t):= h_\pm +\eps\Omega \tilde\chi^\delta_\pm(t)+\eps^2 \Omega_\pm^{(2)} \tilde\chi^\delta_\pm(t)$$
\end{lemma}

\begin{remark}\label{rem:alphachoice2}
Note that if $\delta=\eps^\alpha$ with $\alpha\in (\frac 13, \frac 38)$, as suggested in Remark~\ref{rem:alphachoice}, then $\eps^2\delta^{-5}\ll \eps^{3/2}\delta^{-4}$. 
\end{remark}

This lemma emphasizes the purpose of these superadiabatic projectors: they allow to diagonalize the operator ${\rm op}_\eps( H)$ up to the correction $\eps\,{\rm op}_\eps(\Omega)+\eps^2\,{\rm op}_\eps (\Omega^{(2)})$ which is of lower order in $\eps$ (recall that $\Omega= i (B_++B_-)$ is self-adjoint).

\begin{proof}
The proof comes from the symbolic calculus of Proposition~\ref{prop:symbol} and Remark~\ref{rem:calculpseudo}, keeping in mind that 
we have $|w(x)|>\delta$ on the support of the cut-off functions. 
We observe 
\begin{align*}
{\rm op}_\eps (\Pi^\eps_\pm\chi_\pm ^\delta(t) )&{\rm op}_\eps ( H )
= \,{\rm op}_\eps \left( (\Pi_\pm^\eps \chi_\pm^\delta(t) )\sharp_\eps H   \right)=  \,{\rm op}_\eps \left(\chi_\pm^\delta(t)  (\Pi_\pm^\eps \sharp_\eps H )  \right) + {\rm op}_\eps (b_\pm^\eps) 
\end{align*}
with $b_\pm^\eps= (\Pi^\eps_\pm\chi_\pm^\delta )\sharp_\eps H-\chi_\pm^\delta (\Pi^\eps_\pm\sharp_\eps H)$ depending linearly on derivatives of $\chi_\pm^\delta$ of order larger or equal to~$1$. 
Using~\eqref{eq:sharp}, we obtain 
\begin{align*}
{\rm op}_\eps (\Pi^\eps_\pm\chi_\pm ^\delta(t) )&{\rm op}_\eps ( H )
= \,{\rm op}_\eps \left(   \chi_\pm^\delta(t) ( (h_\pm +\eps \Omega+\eps^2\Omega_\pm^{(2)} ) \sharp_\eps \Pi^\eps_\pm \right)+ {\rm op}_\eps (b_\pm^\eps)
+\eps^3 {\rm op}_\eps (\chi^\delta_\pm(t) R_\eps)\\
&=  \,{\rm op}_\eps \left( (\chi^\delta_\pm(t)h_\pm + \tilde \chi _\pm^\delta(t) \chi_\pm^\delta(t) ( \eps \Omega+\eps^2\Omega_\pm^{(2)} ) \sharp_\eps \Pi^\eps _\pm\right)+ {\rm op}_\eps (b_\pm^\eps)
+\eps^3 {\rm op}_\eps (\chi^\delta_\pm(t) R_\eps)\\
&=  \,{\rm op}_\eps \left( (\chi^\delta_\pm(t)H^\eps_{\rm adiab,\pm}(t) ) \sharp_\eps \Pi^\eps _\pm\right)+ {\rm op}_\eps (b_\pm^\eps)
+\eps^3 {\rm op}_\eps (\chi^\delta_\pm(t) R_\eps)
\end{align*}
where we have used in the last equation that $\tilde \chi^\delta_\pm(t)$ is identically equal to $1$ on the support of~$\chi^\delta_\pm(t)$. 
Then, we can write 
\begin{align*}
{\rm op}_\eps (\Pi^\eps_\pm\chi_\pm ^\delta(t) )&{\rm op}_\eps ( H )
=
\,{\rm op}_\eps \left(H^\eps_{\rm adiab,\pm}(t) \right) {\rm op}_\eps \left(\chi^\delta_\pm (t)\Pi^\eps_\pm\right) 
+  {\rm op}_\eps (\tilde b_\pm^\eps )+ \eps^3\,{\rm op}_\eps (\chi ^\delta_\pm(t) R_\eps)
\end{align*}
where $\widetilde b_\pm^\eps=b^\eps_\pm + (\chi ^\delta_\pm H^\eps_{\rm adiab,\pm}(t))\sharp_\eps \Pi^\eps_\pm - H^\eps_{\rm adiab,\pm}(t)\sharp _\eps(\Pi^\eps_\pm\chi^\delta_\pm) $ satisfies the same properties as $b_\pm^\eps$. 
Using \eqref{estimate:Reps}, we have 
$${\rm op}_\eps (\chi ^\delta_\pm(t) R_\eps)=\mathcal O(\delta^{-5}).$$
 Besides,  on the support of $\check\chi^\delta_+(t)$, the functions $\partial_z^\alpha \tilde \chi^\delta_+(t)$, $\partial_z^\alpha  \chi^\delta_+(t)$, and thus $\widetilde b_+^\eps$ and its derivatives,  are all  identically equal to~$0$ for any $\alpha\in\N^{2d}$, and similarly for the $minus$-mode. Therefore, using  Remark~\ref{rem:locpseudo}, we obtain
$${\rm op}_\eps (\check \chi^\delta_\pm(t)) {\rm op}_\eps (\widetilde b_\pm^\eps)= \mathcal{O}(\eps^{N+1} \delta^{-3-N})$$
because for $\gamma\in\N^{2d}$, using Remark~\ref{rem:C3} (the worst term being $\Pi^\eps_\pm$), we have
$$N^\eps_d ( \widetilde b_\pm^\eps) = \mathcal O( \eps \delta ^{-|\gamma|-3})\;\;\mbox{and}\;\; 
N^\eps_d(\partial_z^\gamma \check \chi^\delta_\pm)= O( \delta ^{-|\gamma|}).$$ One then concludes by choosing $N=2$.
 \end{proof}

We can now perform the proof of Proposition~\ref{prop:propagationt_out}.

\smallskip

\begin{proof}[Proof of Proposition~\ref{prop:propagationt_out}]
Without loss of generality, we can  reduce to only one mode  and we can assume  $v^\eps_+(s_1)=0$, what we do from now on. Indeed, the same scheme of proof  then extends to the other mode and one gets  the general case  because of the linearity of the equation. It is also enough to prove 
$$\left\|  \Pi_\pm \psi^\eps(s_2) -\psi^\eps_{\pm,{\rm app}}(s_2)
 \right\| _{\Sigma^k_\eps} \leq C_k (1+ | \ln \delta|) \left( {\eps^{3/2} \over \delta^4}  +{ \sqrt \eps\over \delta} \right),$$
 where $\psi^\eps_{\pm,{\rm app}}$  has been defined in~\eqref{def:psiepsapp}. Indeed, 
 the same argument will be valid for any $s^*\in[s_1,s_2]$, with the same constant $C_k$ because that constant will only depend on the sup-norm of quantities that are continuous functions in $\{|w(x)|>\delta\}$. 
  
 \smallskip

We choose $\delta$ such that $\sqrt \eps\delta^{-1}\leq 1$  and consider $\chi_\pm^\delta(t)$, $\tilde \chi^\delta_\pm(t)$ and $\check \chi^\delta_\pm(t)$
as in the preceding section (see Remark~\ref{rem:chidelta} and Lemma~\ref{lem:superadiab}); they enjoy the following relations: 
\begin{align*}
0\leq \check  \chi^\delta_\pm (t)& \leq  \chi^\delta_\pm (t)\leq \tilde  \chi^\delta_\pm (t)\leq 1\\
\tilde \chi^\delta _\pm(t)= 1\; \mbox{on}\; {\rm supp }\, \chi^\delta_\pm (t)&\; \;\mbox{and}\;\; \chi^\delta _\pm(t)= 1\; \mbox{on}\; {\rm supp }\,  \check \chi^\delta_\pm (t) .
\end{align*}
We additionally require
$$\partial_t \check \chi^\delta_\pm(t)=
\{ h_\pm ,\check \chi^\delta_\pm(t)\}.$$
and 
$ \check\chi^\delta_+(s_2)=\check \chi^\delta_-(s_2)=: \check\chi^\delta(s_2)$. 
Then, the functions $\chi_\pm^\delta(s_2)$ localize close to the point $z_-(s_2)$ while for $t\in[s_1,s_2]$, $\chi_\pm^\delta(t)$ localize on separated points, $\Phi^{t,s_2}_\pm (z_-(s_2))$, and similarly for $\check\chi^\delta_\pm$ and $\tilde \chi^\delta_\pm$. 
We set for $t\in[s_1,s_2]$
$$w^\eps_-(t) ={\rm op}_\eps (\check \chi^\delta_-(t)) \left( {\rm op}_\eps (\chi^\delta_-(t)\Pi_-^\eps ) \psi^\eps(t) - {\rm op}_\eps (\chi^\delta_-(t))\psi^\eps_{-,\rm app} (t) \right)$$
and 
$$ w^\eps_+(t)={\rm op}_\eps (\check \chi^\delta_+(t)) \, {\rm op}_\eps (\chi^\delta_+(t)\Pi_+^\eps ) \psi^\eps(t).$$
The crucial point of the proof is to establish the equation satisfied by $w^\eps_\pm(t)$. 
\begin{lemma}\label{lem:eqweps}
Let $k\in\N$, $\delta\in (0,1)$ with $\sqrt \eps\delta^{-1}\ll 1$. For $t\in[s_1,s_2]$, we have in $\Sigma^k_\eps$,
\begin{align*}
i\eps \partial_t w^\eps_+ &= -\frac{\eps^2}2 \Delta w^\eps_+ + \lambda_+ w^\eps_+  +\eps\, {\rm op}_\eps\left(\tilde \chi^\delta_\pm(t)(\Omega +\eps \Omega_+^{(2)})\right)w^\eps_+ + \mathcal{O}(\eps^3\delta^{-5} )
\\
i\eps \partial_t w^\eps_- &= -\frac{\eps^2}2 \Delta w^\eps_- +\lambda_- w^\eps_- +\eps \,{\rm op}_\eps\left( \tilde\chi^\delta_\pm(t)(\Omega+\eps\Omega_-^{(2)})\right)w^\eps_-
+ \mathcal{O}((\eps^{5/2} \delta^{-4} + \eps^{3/2}\delta^{-1})(1+ |\ln \delta|))
\end{align*}
 with initial data 
 $w^\eps_\pm (s_1)= O(\sqrt\eps). $
\end{lemma}

\begin{proof}[Proof of Lemma~\ref{lem:eqweps}]
Let us begin with $w^\eps_+(t)$. We have 
\begin{align*}
i\eps \partial_t w^\eps_+(t)=& \,{\rm op}_\eps (\check \chi^\delta_+(t)){\rm op}_\eps (\chi^\delta_+(t)\Pi_+^\eps ) {\rm op}_\eps (H)\psi^\eps(t)\\
&+i\eps   {\rm op}_\eps (\partial_t \check \chi^\delta_+(t)){\rm op}_\eps (\chi^\delta_+(t)\Pi_+^\eps )\psi^\eps(t)
+ i\eps\,  {\rm op}_\eps (\check \chi^\delta_+(t)){\rm op}_\eps (\partial_t \chi^\delta_+(t)\Pi_+^\eps )\psi^\eps(t).
\end{align*}
Using $\partial_t \check \chi^\delta_+(t)=\{h_+,\check \chi^\delta_+(t)\}$ and the fact that $\partial_z \chi^\delta_+(t)=0$ on the support of $\check \chi^\delta_+(t)$, we obtain by Remark~\ref{rem:locpseudo} as in the proof of Lemma~\ref{lem:superadiab},
$$\eps\, {\rm op}_\eps (\check \chi^\delta_+(t)){\rm op}_\eps (\partial_t \chi^\delta_+(t)\Pi_+^\eps )\psi^\eps(t)=\mathcal{O}(\eps^{N+1}\delta^{-3-N})$$
and we choose as before $N=2$.
By  Lemma~\ref{lem:superadiab}, we are left with 
\begin{align}
\label{oct2}
i\eps \partial_t w^\eps_+(t)=\,& {\rm op}_\eps (\check \chi^\delta_+(t)) {\rm op}_\eps \left(H^\eps_{\rm adiab,+}
\right){\rm op}_\eps (\chi^\delta_+(t)\Pi_+^\eps ) \psi^\eps(t)
\\ \nonumber&
+i\eps  \, {\rm op}_\eps (\partial_t \check \chi^\delta_+(t))\, {\rm op}_\eps (\chi^\delta_+(t)\Pi_+^\eps )\psi^\eps(t)+\mathcal{O}(\eps^3\delta^{-5}).
\end{align}
We now take advantage of Remark~\ref{rem:calculpseudo} for writing
\begin{align*}
\left[ {\rm op}_\eps (\check \chi^\delta_+(t)) ,  {\rm op}_\eps (H^\eps_{\rm adiab,+})\right] =\,&
-i \eps  \,{\rm op}_\eps (\{ \check \chi^\delta_+(t) , H^\eps_{\rm adiab,\pm} \} ) +\mathcal{O}(\eps^3\delta^{-5}),
\end{align*}
where we have used the analysis of  the singularities of  $\Omega$ and $\Omega^{(2)}_+$ (see Lemma~\ref{lem:C2}).
We deduce 
\begin{align*}
{\rm op}_\eps (\check \chi^\delta_+(t))\,  {\rm op}_\eps (H^\eps_{\rm adiab,+})&\, {\rm op}_\eps (\chi^\delta_+(t)\Pi_+^\eps ) \psi^\eps(t) =\,
 {\rm op}_\eps (H^\eps_{\rm adiab,+} ) w^\eps_+ \\
&-i \eps\, {\rm op}_\eps (\{h_+,\check \chi^\delta_+(t)  \} ) {\rm op}_\eps (\chi^\delta_+(t)\Pi_+^\eps )\psi^\eps(t) +\mathcal{O}(\eps^3\delta^{-5}).
\end{align*}
Combining the latter with~\eqref{oct2} and the relation $\partial_t \check \chi^\delta_+(t)=
\{h_+ ,\check \chi^\delta_+(t)\}$, we obtain
\begin{align*}
i\eps \partial_t w^\eps_+(t)=& {\rm op}_\eps( H^\eps_{\rm adiab,+})w_+^\eps(t)
+\mathcal{O}(\eps^3\delta^{-5}) .
\end{align*}

For $w^\eps_-(t)$, the computation follows the same steps with the difference that there is an additional term due to the presence of~$\psi^\eps_{-,{\rm app}}$. Using Lemma~\ref{lem:ansatz}, an additional remainder in 
$$\mathcal{O}((\eps^{5/2}\delta^{-4} +\eps^{3/2}\delta^{-1})(1+ |\ln \delta|)),$$
is generated, which is much larger than $\mathcal O (
\eps^3\delta^{-5})$ (again because of $\sqrt\eps \delta^{-1} \leq 1$).
\end{proof}

We can now conclude the proof of Proposition~\ref{prop:propagationt_out}.
Using 
 Lemma~\ref{lem:eqweps}, and by the properties of the unitary propagator associated with the operators 
$${\rm op}_\eps(H^\eps_{\rm adiab,\pm})=  -\frac{\eps^2}2 \Delta + \lambda_\pm  +\eps {\rm op}_\eps(\Omega\tilde\chi^\delta_\pm(t))+\eps^2 {\rm op}_\eps (\Omega^{2}_\pm\tilde\chi^\delta_\pm(t)),$$
(see~\cite{MaRo}),
we obtain the existence of a constant $C_k$ such that
\begin{equation} \label{claim:weps}
\| w^\eps_+(s_2) \|_{\Sigma^k_\eps} + \| w^\eps_-(s_2)\|_{\Sigma^k_\eps} \leq C_k((\eps^{3/2} \delta^{-4}+\sqrt\eps\delta^{-1})(1+ |\ln\delta|)).
\end{equation}
Equivalently, using $ \check \chi^\delta(s_2)= \check \chi^\delta(s_2) \chi^\delta_\pm (s_2)$, 
$${\rm op}_\eps (\Pi^\eps_++\Pi^\eps_-)= {\rm Id} + \mathcal O(\eps^2 \delta^{-4})$$
(see Remark~\ref{rem:C3}), 
 and  the localization properties of $\psi^\eps_{-,\rm app}$ (see Lemma~\ref{lem:localisation} (2)), the latter relation writes 
$$ {\rm op}_\eps ( \check \chi^\delta(s_2) ) \psi^\eps(s_2) =  \psi^\eps_{-,\rm app} (s_2)+\mathcal{O}((\eps^{3/2} \delta^{-4}+\sqrt\eps\delta^{-1})(1+ |\ln\delta|))$$
in $\Sigma^k_\eps$.
The argument could have been worked out between $s_1$ and any $s\in [s_1,s_2]$. Therefore, at this stage of the proof, varying the function $\check \chi_\delta$, we have obtained that for any $t\in [s_1,s_2]$ and any  cut-off function $ \chi_\delta$ supported in $\{ |w(x)|>\delta\}$, we have in $\Sigma^k_\eps$, 
\begin{equation}\label{oct3} 
{\rm op}_\eps ( \chi_\delta (t) ) \psi^\eps(t) =  \psi^\eps_{-,\rm app} (t)+ \mathcal{O}((\eps^{3/2} \delta^{-4}+\sqrt\eps\delta^{-1})(1+ |\ln\delta|)).
\end{equation}

\smallskip

We now want to extend this approximation to $\psi^\eps(t)$ itself. We define $\theta^\delta_-$ localizing close to the trajectory $z_-(t)$ and in $\{|w(x)|>\delta\}$ (we denote it $\theta_-^\delta$  to emphasize that it is independent of the functions $\chi^\delta_-$ used before). 
The analysis performed above applies to the special case of $\theta^\delta_-$ and  we have in $\Sigma^k_\eps$
and for $t\in[s_1,s_2]$
$$ {\rm op}_\eps (\theta^\delta_-(t) ) \psi^\eps(t) =  \psi^\eps_{-,\rm app} (t)+ \mathcal{O}((\eps^{3/2} \delta^{-4}+\sqrt\eps\delta^{-1})(1+ |\ln\delta|)).$$
 We 
 study 
 $$w^\eps(t)=  {\rm op}_\eps (1-\theta^\delta_-(t)  ) \psi^\eps(t)$$
 and aim at proving that $w^\eps(s_2)$ is negligible, which is the case for 
 $w^\eps(s_1)$. Moreover, for $t\in  [s_1,s_2]$,
 \begin{equation} \label{oct3}
 i\eps \partial_t w^\eps = -\frac{\eps^2} {2} \Delta w^\eps+V w^\eps +\frac 12 \left[ {\eps^2}  \Delta,  {\rm op}_\eps ( \theta^\delta_-(t)) \right] \psi^\eps. 
 \end{equation}

Let us study the source term.  By symbolic calculus (see Remark~\ref{rem:calculpseudo}), we have 
 $$  \left[- \frac{\eps^2} {2} \Delta,  {\rm op}_\eps (\theta^\delta_-(t) )\right] =\eps\, {\rm op}_\eps (\chi^\delta(t)) +
\mathcal O(\eps^3\delta^{-3}) $$
where $\chi ^\delta=\xi\cdot \nabla_x \theta^\delta_-\in{\mathcal C}^\infty(\R^{2d+1})$ is supported in $\{ |w(x)|>c \delta \}$ for some $c>0$, with $\chi^\delta(t)$ identically equal to $0$ in a neighborhood of $\Phi^{t,s_2} (z_-(s_2))$ and $|\partial^\alpha \chi^\delta (t,x,\xi)|\leq C\delta^{-1-|\alpha|} $  for all $\alpha\in\N^{2d}$. 

\noindent We deduce from~\eqref{oct3}  and from 
 (2) of Lemma~\ref{lem:localisation} that for $N\in\N^*$, $t\in[s_1,s_2]$ and in~$\Sigma^k_\eps$
\[
  {\rm op}_\eps (\chi^\delta(t)) \psi^\eps(t)
  =  \mathcal{O}\left(\delta^{-1}(\sqrt\eps\delta^{-1}) ^N (1+ |\ln\delta|)\right) +  \mathcal{O}\left((\eps^{3/2} \delta^{-4}+\sqrt\eps\delta^{-1})(1+ |\ln\delta|)\right).\]
  Therefore,  equation~\eqref{oct3}  gives in $\Sigma^k_\eps$
  $$w^\eps(s_2) = w^\eps(s_1)+ \mathcal{O}\left(\delta^{-1}(\sqrt\eps\delta^{-1}) ^N (1+ |\ln\delta|)\right)+ \mathcal{O}\left((\eps^{3/2} \delta^{-4}+\sqrt\eps\delta^{-1})(1+ |\ln\delta|)\right)+\mathcal O (\eps^2\delta^{-3}) .$$
 By choosing $N=3$ and using $\eps^2\delta^{-3}\ll\eps^{3/2}\delta^{-4}$,   we deduce 
 $$\psi^\eps(s_2) =  \psi^\eps_{-,\rm app} (s_2)+ \mathcal{O}\left((\eps^{3/2} \delta^{-4}+\sqrt\eps\delta^{-1})(1+ |\ln\delta|)\right),$$  
 whence Proposition~\ref{prop:propagationt_out}.
\end{proof}


\section{Passing through the gap region }\label{sec:trans}

At this stage of the proof, we have obtained an approximation of the solution as long as the trajectories do not enter in the region $\{ |w(q)|
\leq c \delta\}$,  for some $c>0$ fixed, i.e. in a neighborhood of the crossing set $\Upsilon$. We now focus on trajectories that reach their minimal gap inside this region and enter in the region at time $t^\flat -\delta$ and leaves it at time $t^\flat +\delta$.

\smallskip

The strategy is the following.
\begin{enumerate}
\item We first perform a change of time and unknown in order to reduce the system~\eqref{system} into a Landau-Zener model in the region  $\{ |w(q)|\leq c \delta\}$.
\item We identify the ingoing wave packet in the new coordinates, i.e. the function $\psi^\eps(t^\flat-\delta)$ that satisfy in $L^2(\R^d)$, 
$$\psi^\eps(t^\flat-\delta)=\psi^\eps_{\rm app} (t^\flat-\delta)+\mathcal O((\sqrt\eps\delta^{-1}+\eps^{3/2} \delta^{-4})\left(1+ |\ln\delta| )\right) .$$
\item We prove that we can use the resolution of the Landau-Zener model to obtain an approximation of the solution at time $t^\flat +\delta$.
\end{enumerate}

\subsection{Reduction to a Landau-Zener model}

To pass through the region $\Upsilon$, following ideas from~\cite{Hag94}, we use a Taylor approximation along
 the trajectory~$\Phi_0^{t,t^\flat } (z^\flat)  =\left(q_0 (t),p_0 (t)\right)$ introduced in Section~\ref{sec:actionproof}. 
   We make the time-scaling $t=t^\flat +s\sqrt\eps$ and consider the new unknown function $u^\eps(s)\in L^2(\R^d,\C^2)$ defined by
\begin{equation}\label{eq:ueps}
\psi^\eps(t)= {\rm e}^{\frac i \eps S_0(t,t^\flat, z^\flat)} {\rm WP}^\eps_{\Phi_0^{t,t^\flat}(z^\flat)} (u^\eps(s)),\;\;t=t^\flat +s\sqrt\eps
\end{equation}
where the action $S_0(t,t^\flat, z^\flat)$ is associated with~$h_0$, defined in \eqref{def:h0}, and $\Phi_0^{t,t^\flat } (z^\flat)$ as introduced in Lemma~\ref{lem:action}.

\begin{remark}\label{rem:nov7}
\begin{enumerate}
\item Note  that when $t=t^\flat-\delta$, then  $s=-s_0:=-\delta/\sqrt\eps$  and when $t=t^\flat+\delta$, then $s=s_0=\delta/\sqrt\eps$. Since we have assumed $\sqrt\eps \delta^{-1}\ll 1$ in the preceding section, we will have $s_0\gg 1$.
Through the change of variable~\eqref{eq:ueps}, for $k\in\N$ and $s\in [-s_0,s_0]$, there exist constants $c,C$ such that 
$$c \| u^\eps  (s)\| _{\widetilde\Sigma^k_\eps}\leq \|\psi^\eps(t) \|_{\Sigma^k_\eps}\leq C \| u^\eps  (s)\| _{\widetilde\Sigma^k_\eps}$$
with 
\begin{equation}\label{def:normnov7}
\| f \| _{\widetilde\Sigma^k_\eps}= \sup_{|\alpha|+|\beta|\leq k}  \eps^{\frac{|\alpha|+|\beta|}2} \| f\|_{\Sigma^{|\alpha|+|\beta|}}.
\end{equation}
Therefore, it is natural to use these sets $\widetilde \Sigma^k_\eps$ for estimations. 
\end{enumerate}
\end{remark}

\begin{lemma}\label{lem:nov1}
Let $k\in\N$. The family $(u^\eps(s))_{\eps>0}$
 satisfies for all $(s,y)\in\R^{2d+1}$
\begin{equation}\label{eq:system}
i\partial_s  u^\eps= A\left( s \, r\omega  +dw(q^\flat )y \right)u^\eps+\sqrt\eps \left(-\frac 12 \Delta u^\eps +B^\eps(s,y) u^\eps\right) 
\end{equation}
where $B^\eps$ is a smooth hermitian matrix valued potential with the following properties: there exist constants $C_0,C_1>0$  such that for all  $s\in [-s_0,s_0]$ and $y\in\R^d$, 
$$\left|  B^\eps(s,y) \right|\leq C_0 \left( s^2 \langle \sqrt\eps |y| \rangle + |y|^2\right),\;\;
\left| \nabla B^\eps(s,y) \right|\leq C_1 \left( \sqrt \eps |y|^2 + |y| +\sqrt\eps s^2  \right)
$$
and for  all  $|\beta|\geq   2$, there exists $C_\beta>0$ such that  for all  $s\in [-s_0,s_0]$ and $y\in\R^d$, 
$$\left| \partial^\beta _y B^\eps(s,y) \right|\leq C_\beta \eps^{ \frac{|\beta|-2}2} \langle \sqrt\eps y \rangle ^2 .$$
\end{lemma}

\begin{remark}\label{rem:gap}
When $(t^\flat, z^\flat)$ is the point of the trajectory $\Phi_-^{t,t_0}(z_0)$ where the quantity $\left|w\left( \Phi_-^{t,t_0}(z_0)\right)\right|$ (called the gap) is minimal, a similar analysis yields to the system 
\[
i\partial_s  u^\eps= A\left(  \frac {w(q^\flat)}{\sqrt\eps} + s \, r\omega  +dw(q^\flat )y \right)u^\eps+\sqrt\eps \left(-\frac 12 \Delta u^\eps +B^\eps(s,y) u^\eps\right). 
\]
This observation gives a starting point for the analysis of the propagation of a wave packet passing close to a crossing point, while no exactly through it. The size of the gap comparatively to $\sqrt\eps$ then is a crucial point of the description. 
\end{remark}

Recall that $r\omega= dw(q^\flat) p^\flat$ and that $w(q^\flat)=0$. 
We shall set in the following 
$\eta(y) :=  dw(q^\flat) y $
and compare $u^\eps$ with the solution $u$ of the equation
$$i\partial_s  u= A\left( s \,r\omega  +\eta(y) \right)u .$$
The important point to note here is that the leading part $A\left( s \,r\omega +\eta(y)   \right)$ of the system has the same structure as  the well-known Landau-Zener system (see references~\cite{La,Ze} and equation~\eqref{eq:LZ} below). 
The latter is well understood as it will be detailed in the next sections. We shall use the initial data at time $s=-s_0$ with 
$$s_0= \delta /\sqrt\eps.$$  The time $-s_0$  corresponds to $t=t^\flat-\delta$, i.e. to the ingoing solution, and we shall deduce the value of the outgoing solution at time $t=t^\flat +\delta$ or equivalently $s=+s_0$. This will be done assuming $\delta \gg\sqrt\eps$, thanks to the scattering result of the next section.

\begin{proof}
We use  the formalism of Section~\ref{sec:profileproof}, together with the observation of Appendix~\ref{appB}.
The first step consists in observing  that
\begin{align*}
&\left(i\eps \partial_t  +\frac{\eps ^2} 2\Delta -v(x)\right) \psi^\eps(t,x)
= {\rm e}^{\frac i\eps S_0(t,t^\flat,z^\flat)}
{{\rm e}^{\frac i\eps p_0(t)\cdot (\sqrt{\eps}y)}}
\Bigl(
i\sqrt\eps \partial_s u^\eps(s,y)
+\frac \eps 2 \Delta_y  u^\eps(s,y)\\
&\qquad \qquad\qquad \;\; - \left( v(q_0(t)+\sqrt\eps y)-v(q_0(t))  -y\sqrt\eps dv(q_0(t))\right) u^\eps (s,y)\Bigr)
\Bigr|_{y=\frac{x-q_0(t)}{\sqrt\eps}}\\
&\qquad= {\rm e}^{\frac i\eps S_0(t,t^\flat,z^\flat)}{{\rm e}^{\frac i\eps p_0(t)\cdot (\sqrt{\eps}y)}}
\Bigl(
i\sqrt\eps \partial_s u^\eps(s,y)
+\frac \eps 2 \Delta_y  u^\eps(s,y) + 
\eps \mathcal W^\eps (t,y) u^\eps(s,y)\Bigr)
\Bigr|_{y=\frac{x-q_0(t)}{\sqrt\eps}}
\end{align*}
where we have used  Lemma~\ref{lem:localisation} (1) and the definition of the action. Besides, 
$$\mathcal W^\eps(s,y)= R_0(t,y\sqrt\eps)y\cdot y,\;\;R_0(t,y\sqrt\eps ) = \int_0^1 {\rm Hess} \,  v(q_0(t)+\sqrt\eps \theta y)  (1-\theta) d\theta, $$
and $R_0$ is bounded with bounded derivatives according to~\eqref{rangeV}.
Similarly, we have  
$$
A(w(x)) \psi^\eps(t,x) =  {\rm e}^{\frac i\eps S_0(t,t^\flat,z^\flat)}\Bigl({{\rm e}^{\frac i\eps p_0(t)\cdot (\sqrt{\eps}y)}} 
A(w(q_0(t)+\sqrt\eps y)) u^\eps(s,y)  \Bigr)
\Bigr|_{y=\frac{x-q_0(t)}{\sqrt\eps}}.$$
Therefore, 
 Equation~\eqref{system} becomes
\[
i\sqrt\eps \partial_s u^\eps +\frac\eps 2 \Delta u^\eps  + \eps \mathcal W^\eps (s,y) u^\eps(s,y)
= A(w(q_0(t)+\sqrt\eps y)) u^\eps(s,y) .
\]
Writing $w(q_0(t)+\sqrt\eps y)=w(q_0(t)) +\sqrt\eps  dw(q_0(t))  y +\eps R_1(t,y\sqrt\eps)y\cdot y $ 
\[
i\sqrt\eps \partial_s u^\eps= A \left( w(q_0(t^\flat + s\sqrt\eps)) +\sqrt\eps dw ( q_0(t^\flat + s\sqrt\eps )) y \right)
u^\eps+\eps R_2(s\sqrt\eps , y\sqrt\eps)y\cdot y\, u^\eps(s,y),
\]
for some bounded smooth matrix $R_1$ and tensor $R_2$, with bounded derivatives coming from~\eqref{rangeV}.
We conclude by performing a Taylor expansion in $s$, writing 
\begin{align*}
 q_0(t^\flat +s\sqrt\eps)&= q^\flat +\sqrt\eps s p^\flat+ \eps s^2 R_3(s\sqrt\eps)\\
\;\;\mbox{ and} \qquad
  w(q_0(t^\flat + s\sqrt\eps)) &+\sqrt\eps dw ( q_0(t^\flat + s\sqrt\eps )) y \\
  & =\sqrt\eps s  dw( q^\flat) p^\flat +\sqrt\eps dw(q^\flat) y + \eps  R_4(s\sqrt\eps)s^2 + \eps^{3/2} s^2 R_5(s\sqrt\eps )y   
  \end{align*}
   for some smooth bounded vector-valued $R_3$ and $R_4$, and matrix-valued $R_5$, with bounded derivatives  because of  the assumption~\eqref{rangeV}. 
  
  The properties of $B^\eps(s,y)$ come from its expression in terms of the $R_j$, $j\in\{1,\cdots, 5\}$   
  $$B^\eps(s,y) = \mathcal W(t^\flat + s\sqrt\eps , y)+  s^2 A \left( R_4(s\sqrt\eps)+ \sqrt\eps R_5(s\sqrt\eps) y \right) + R_2(s\sqrt\eps , y\sqrt\eps)y\cdot y$$
  and the assumption~\eqref{rangeV} made on the potential.
\end{proof}

\subsection{The Landau-Zener model and the structure of the solutions}

The structure of the system~\eqref{eq:system} suggests that we consider the model problem 
\begin{equation}\label{eq:LZ}
\left\{\begin{array}{l}
i \partial_s u  = A( sr \omega +\eta)u ,\\
 u(0,\eta)= u_0(\eta)\in\C^2
\end{array}
\right.
\end{equation}
where $\eta\in\C^2$ is a parameter. As we shall see below, this problem can be turned into the following Landau-Zener problem by elementary computations 
\begin{equation}\label{LZversion0}
\frac 1i \partial_s u_{LZ}(s,z) = \begin{pmatrix} s+z_1 & z_2 \\ z_2 & -s-z_1\end{pmatrix}u_{LZ}(s,z).
\end{equation} 
Therefore, one can deduce the behavior of the solutions to \eqref{eq:LZ} from 
the asymptotics, as $s\rightarrow \pm\infty$, of the solutions to the Landau-Zener problem  \eqref{LZversion0}.
 Besides the historical references~\cite{La,Ze}, the reader can refer to~\cite{FG02} where an analysis of the behavior of the solutions of 
 the Landau-Zener model is given with a stationary phase approach; or to~\cite{Hag94} where the proof is given in terms of parabolic-cylinder functions. 
 We follow the results of the Appendix of~\cite{FG02} which  are obtained for~$\eta$ taken in  a fixed compact, while the  analysis in terms of the size $R$ of this compact is performed {in~[Appendix, \cite{FL08}]}: as $s\rightarrow \pm\infty $
\begin{equation}\label{asympt_LZ}
u_{LZ} (s) = {\rm e}^{i\frac {(s+z_1)^2}{2} +i\frac{z_2^2}{2} \ln |s+z_1| } \begin{pmatrix} u_1^\pm(z_2)\\ 0\end{pmatrix}
+  {\rm e}^{-i\frac {(s+z_1)^2}{2} -i\frac{z_2^2}{2} \ln |s+z_1| } \begin{pmatrix} 0\\ u_2^\pm(z_2)\end{pmatrix}
+\mathcal{O}(R^2|s|^{-1}),
\end{equation}
with 
$$u_1^+ = a(z_2) u_1^--\overline b(z_2) u_2^-,\;\;u_2^+ = b(z_2) u_1^-+ a(z_2) u_2^-$$
where the coefficients $a$ and $b$ are given by~\eqref{coef.scat}. 
It is then possible to derive the next proposition about solutions to~\eqref{eq:LZ} in which 
$(\vec V_\omega,\vec V_\omega^\perp)$  is a direct orthogonal basis of $\R^2$ as in \eqref{def:Vomega'} consisting of normalized real-valued eigenvectors of $A(\omega)$ satisfying
$$A(\omega)\vec V_\omega=\vec V_\omega\;\;\mbox{and}\;\;
A(\omega)\vec V_\omega^\perp=-\vec V_\omega^\perp.$$
Note that they are uniquely defined up to a sign. 
The next lemma gives the form of the asymptotics of $u(s,\eta)$ when $s\rightarrow \pm\infty$ in such a basis, together with  scattering relations.  

\begin{lemma}\label{scat:LZ} 
 There exists $\alpha_1^{\rm in},\alpha_2^{\rm in},\alpha_1^{\rm out},\alpha_2^{\rm out} \in{\mathcal S}(\R^d) $ such that as $s$ goes to $-\infty$ and for $|\eta|\leq R$,
$$u(s,\eta) = {\rm e}^{i\Lambda(s,\eta)}  \alpha_1^{\rm in}(\eta) \vec V_\omega^\perp +{\rm e}^{-i\Lambda(s,\eta) } \alpha_2^{\rm in}(\eta) \vec V_\omega + \mathcal{O}(R^3|s|^{-1}),$$
\footnote{check $R^3$}
and as $s$ goes to $+\infty$ and $|\eta|\leq R$ 
$$u(s,\eta) = {\rm e}^{i\Lambda(s,\eta) } \alpha_1^{\rm out}(\eta)  \vec V_\omega^\perp +{\rm e}^{-i\Lambda(s,\eta)  }\alpha_2^{\rm out}(\eta)  \vec V_\omega+ \mathcal{O}(R^3|s|^{-1}),$$
where 
\begin{equation}\label{def:Lambda}
\Lambda(s,\eta)= {1\over 2r} |\omega\cdot \eta+rs|^2 +{1\over 2r} |\omega^\perp \cdot \eta|^2 \,\ln (\sqrt{r}|s| ).
\end{equation}
Besides 
$$\begin{pmatrix}
\alpha^{\rm out}_1\\ \alpha_2^{\rm out} \end{pmatrix}= S(r^{-1/2}\omega^\perp\cdot \eta)
\begin{pmatrix}
\alpha_1^{\rm in}\\ \alpha_2^{\rm in} \end{pmatrix}$$
with 
$$S(\eta)=\begin{pmatrix}a(\eta)&-\overline b(\eta)\\
b(\eta)& a(\eta)\end{pmatrix},$$
where the coefficients $a$ and $b$ are given by~\eqref{coef.scat}. 
\end{lemma}

\begin{proof}
 For proving Lemma~\ref{scat:LZ}, we relate the solution $u$ of the  system~\eqref{eq:LZ} to $u_{LZ}$ 
 thanks to a change of variables
 via the rotation matrix
$\mathcal{R}(\theta)$ defined in~\eqref{def:Rtheta} and its property~\eqref{eq:matrixKR}.
Therefore, choosing $\theta\in\R$ such that $\omega_\theta=- \omega$, we have 
$$\mathcal{R}(\theta)^{-1}A(\eta+sr\omega)\mathcal{R}(\theta)= -
\begin{pmatrix}
 \eta\cdot \omega +sr  &  \eta \cdot \omega^\perp  \\
 \eta \cdot \omega^\perp &-\eta\cdot\omega- sr  \end{pmatrix}.$$
 We then write 
 \begin{align*}
 \frac 1i \partial_s ( \mathcal R(\theta)^{-1} u) & = \mathcal{R}(\theta)^{-1}A(-\eta-sr\omega)\mathcal{R}(\theta) \, ( \mathcal R(\theta)^{-1} u)\\
 &= \begin{pmatrix}
 \eta\cdot \omega +sr  &  \eta \cdot \omega^\perp  \\
 \eta \cdot \omega^\perp &-\eta\cdot\omega- sr  \end{pmatrix}( \mathcal R(\theta)^{-1} u).
 \end{align*}
 and we deduce that 
 $$v(s,\eta)=  \mathcal R(\theta)  ^{-1} u(sr^{-1/2}, r^{1/2}  \eta)$$
 solves 
 $$\frac 1 i \partial_s v(s,\eta) =\begin{pmatrix}
 \eta\cdot \omega +s  &  \eta \cdot \omega^\perp  \\
 \eta \cdot \omega^\perp &-\eta\cdot\omega- s \end{pmatrix} v(s,\eta),$$
 i.e. the equation~\eqref{LZversion0} for $z=(\eta\cdot\omega, \eta\cdot \omega^\perp)$ and we can write 
 $$
u(s,\eta)= \mathcal R(\theta)  u_{LZ}(sr^{1/2},r^{-1/2}z).$$
Then, Equation~\eqref{asympt_LZ} motivates the following: 
\begin{align*}
\Lambda(s,\eta)
&:= \frac 12 |sr^{1/2}+ r^{-1/2} \eta\cdot\omega|^2 + \frac 12 |r^{-1/2} \eta\cdot\omega^\perp|^2 \ln \left| sr^{1/2}+r^{-1/2} \eta\cdot\omega\right| \\
& = \frac 1{2r} |sr+  \eta\cdot\omega|^2 + \frac 1{2r} | \eta\cdot\omega^\perp|^2 \ln \left| sr^{1/2}\right| +  O(R^3|s|^{-1})
\end{align*}
where we have performed a Taylor expansion of $\ln \left|1+\dfrac{\eta \cdot \omega}{sr}\right|$ and used $|\eta|\leq R$.
As $s\rightarrow \pm \infty$, we deduce that Equation~\eqref{asympt_LZ} yields to
$$
u(s,\eta)  = {\rm e}^{i\Lambda(s,\eta)}u_1^\pm \mathcal R(\theta)  \begin{pmatrix}1\\ 0\end{pmatrix} +  {\rm e}^{-i\Lambda(s,\eta)}u_2^\pm \mathcal R(\theta)   \begin{pmatrix}0\\ 1\end{pmatrix} + \mathcal{O}(R^3|s|^{-1}).
$$
In view of 
$$\mathcal R(\theta)^{-1} A(\omega) \mathcal R(\theta) = \begin{pmatrix} -1 & 0 \\ 0 & 1 \end{pmatrix}$$
we deduce that there exists $\varsigma\in\{-1,+1\}$ such that
$$\vec V_\omega=\varsigma \mathcal  R(\theta) \begin{pmatrix}0\\ 1\end{pmatrix} ,\;\;
\vec V_\omega^\perp =\varsigma \mathcal  R(\theta) \begin{pmatrix}1\\ 0\end{pmatrix}$$
up to a sign. 
The result of Lemma~\ref{scat:LZ} then  follows with $\alpha_j^{\rm in}=\varsigma u_j^-$ and $\alpha_j^{\rm out}=\varsigma u_j^+$, $j\in\{1,2\}$.
 \end{proof}

\smallskip

In the following, we wish to compare $u^\eps$ with $u$ from \eqref{eq:LZ} with $\eta=\eta(y)$ and use Lemma~\ref{scat:LZ} to deduce the leading behavior of $u^\eps$ at $s_0=\delta/\sqrt\eps$ from information available at time $-s_0=-\delta/\sqrt\eps$. For that purpose,  it is required to identify the ingoing profiles  $\alpha_1^{in}$ and $\alpha_2^{in}$ related with the data $u(-s_0):=u^\eps(-s_0)$, that is known from Section \ref{sec:adiab}. We will do that in the next section and will make use of the following  property of $u^\eps(s,y)$.

\begin{lemma}\label{lem:sigmakLZ}
Assume $u(-s_0)\in \Sigma^k(\R^d)$, $\alpha,\beta\in\N^2$. Then, there exists a constant $C_{\alpha,\beta}>0$ such that the solution of~\eqref{eq:LZ} satisfies for 
for $s\in (-s_0,s_0)$ we have 
$$\| \eta ^\alpha \partial_\eta ^\beta u(s)\| _{L^2} \leq C_{\alpha,\beta}  \langle s\rangle ^{|\beta|} .$$
\end{lemma}

\begin{proof}
When $\beta=0$, one easily checks that the result holds (because $\eta^\alpha$ commutes with the equation). 
One then fixes 
$\alpha$, uses a recursive argument on the length of $\beta$, starting from the conservation of the $L^2$-norm ($\beta=0$) and based on the observation 
$$i\partial_s (\eta^\alpha \partial^\beta u)= A(sr\omega + \eta) (\eta^\alpha \partial^\beta u )+  \sum_{j=1,2}  {\bf 1} _{ \beta_j>0} \, c_j \, A( e_j) \eta^\alpha  \partial^{\beta-{\bf 1}_j}  u$$
where $c_j$ are universal constants and $(e_1,e_2)$ the canonical basis of $\R^2$. An energy inequality generates the growth in $s$. 
\end{proof}

\subsection{The ingoing wave packet}\label{subsec:ingoing}

Here we prove the following proposition.

\begin{proposition}\label{prop:ingoing}
With the assumptions of Theorem~\ref{theo:main}, the solution of~\eqref{system} satisfies~\eqref{eq:ueps} at time $t=t^\flat-\delta$, i.e. $s=-s_0=-\delta/\sqrt\eps$ with 
$$u^\eps(-s_0, y)= {\rm e}^{- i\Lambda(-s_0,\eta)}\alpha_2^{\rm in}(\eta(y)) \vec V_\omega 
+O\left((\sqrt\eps\delta^{-1}+\eps^{3/2} \delta^{-4} +\delta^3\eps^{-1})(1+ |\ln\delta|) \right)
\;
\mbox{in} \; \widetilde \Sigma^k_{\eps}$$
where $\Lambda(s,\eta)$ is defined in~\eqref{def:Lambda}, $\eta$ is given by 
$\displaystyle{\eta= dw(q^\flat) y}$
and
we have 
\begin{equation}\label{eq:alpha1}
\alpha_2^{in}(\eta)=
 {\rm Exp}\left(  \frac i\eps S^\flat _-+\frac i{4r} (\eta\cdot\omega^\perp)^2 \ln(\frac r {\eps})  + \frac i{2r} |\omega\cdot \eta|^2\right)u^{\rm in}_- (y),
\end{equation}
with $S^\flat_-= S_-(t^\flat, t_0,z_0)$.
\end{proposition}

\begin{remark}\label{rem:constraint2}
 This result suggests that $\delta$ has to be chosen so that 
$\delta^3 \ll \eps$, accordingly with  the constraints mentioned in Remark~\ref{rem:alphachoice} and fits with the choice of $\delta= \eps^{\frac 5{14}}$.
\end{remark}

\begin{proof}
We start from the estimate obtained for  $t\leq t^\flat-\delta$, namely 
\begin{align*}
\psi^\eps(t,x)= & \eps^{-d/4}{\rm e}^{\frac i\eps S_-(t,t_0,z_0) +\frac i\eps p_-(t)(x-q_-(t))}  \vec V_-(t, \Phi^{t,t_0}_-(z_0)) \\
&\qquad \times u_-\left(t, \frac{x-q_-(t)}{\sqrt\eps} \right) +\mathcal{O}\left((\sqrt\eps\delta^{-1}+\eps^{3/2} \delta^{-4}) (1+|\ln\delta| )\right)
\end{align*}
in $\Sigma^k_\eps$. We fix $k\in \N$ and prove the estimates in this set.

\smallskip

 We begin by considering the  phase.
The asymptotics of Lemma~\ref{prop:traj_asymp} and Lemma~\ref{lem:action} imply
that 
when $t=t^\flat +\sqrt\eps s$ with $s<0$ and $x=q_0(t)+\sqrt\eps y$, we have the pointwise estimates
$$\frac i\eps S_-(t,t^\flat,z^\flat)=\frac{i}{\eps}S_0(t,t^\flat,z^\flat)- ir s^2+\mathcal{O}(\sqrt\eps s^3)$$
and 
\begin{align*}
\frac i\eps p_- (t)\cdot (x-q_- (t))=&\frac i\eps \left(p_0(t)-  \sqrt\eps s \, ^tdw(q^\flat) \omega +\mathcal{O}( \eps s^2) \right) 
 \cdot \left(x-q_0(t)  + \frac \eps 2 s^2 \, ^tdw(q^\flat)\omega  +\mathcal{O}(\eps^{3/2} s^3)\right)\\
=&\frac{i}{\sqrt\eps}p_0(t)\cdot y- i s\omega \cdot dw(q^\flat) y  + \frac{i}{2}s^2 \omega\cdot dw(q^\flat)p_0(t) 
+ \mathcal{O}(\sqrt \eps s^2|y|) +\mathcal{O}(\sqrt \eps s^3).
\end{align*}
We observe that 
$$\omega \cdot dw(q^\flat) p_0(t)= \omega \cdot dw(q^\flat) p^\flat +\mathcal{O}(s \sqrt\eps )= r+\mathcal{O}(s\sqrt\eps).$$
Therefore
\begin{align*}
\frac{i}{\eps}S_-(t,t^\flat,z^\flat)+\frac{i}{\eps}p_-(t)\cdot (x-q_-(t))&= \frac{i}{\eps}S_0(t,t^\flat,z^\flat) +\frac i {\sqrt\eps}  y\cdot p_0(t) \\
&-\frac{ i}{2}r s^2
-  i s\omega \cdot dw(q^\flat) y  +\mathcal{O}(\sqrt\eps s^2|y|))+\mathcal{O}(\sqrt\eps s^3)
\end{align*}
We now consider the profile and takes into account  Corollary~\ref{cor:profile}.  We obtain  the estimate in $\Sigma^k_\eps$
\begin{align*}
\psi^\eps(t,x) =&\, \eps^{-d/4} {\rm Exp}\left(\frac i\eps S_0(t^\flat+\sqrt\eps s,t^\flat, z^\flat) + \frac i{ \sqrt\eps }p_0(t)  \cdot y\right)\\
&\; \times \vec V_\omega\,  
{\rm Exp}\left(- {\frac i 2\Gamma_0 y\cdot y \ln |s \sqrt\eps|  -\frac i2 rs^2- is  \, \omega \cdot dw(q^\flat) y+ O(s^3\sqrt\eps)}  \right) \\
&\; \times  {\rm e}^{\frac i\eps S^\flat_-}\,
u^{\rm in}_-\left(  y  + y^\eps(s)\right) +\mathcal{O}\left((\sqrt\eps\delta^{-1}+\eps^{3/2} \delta^{-4} +\delta )(1+ |\ln\delta|) \right)
\end{align*} 
where we have approximated $u_-(t)$ by $u^{\rm in}_-$ and  $ \vec Y_-(t) $ by $\vec V_\omega(t)$ for~$t$ close to $t^\flat$ ($|t-t^\flat|\sim \delta$) and $y^\eps(s)$ satisfies the pointwise estimate $y^\eps(s)= \mathcal{O}(s^2\sqrt\eps \langle y\rangle))$. 
We deduce from the fact that $u^{\rm in}_-\in\mathcal S(\R^d)$
\begin{align*}
u^\eps(-s_0,y) &=  \vec V_\omega\,   {\rm e}^{\frac i\eps S^\flat_-}\,
{\rm Exp}\left(-\frac i 2\Gamma_0 y\cdot y \ln |s \sqrt\eps|  -\frac i2rs^2- is  \,  \omega \cdot dw(q^\flat)y \right)
u^{\rm in}_-\left(y\right)\\
&\qquad +\mathcal{O}\left((\sqrt\eps\delta^{-1}+\eps^{3/2} \delta^{-4}+\delta +\delta ^3\eps^{-1} )(1+ |\ln\delta| )\right) 
\end{align*}
where we have used that $\frac{\delta^2}{\sqrt\eps} \ll \frac {\delta^3} {\eps} $ since $\sqrt\eps \ll \delta$.
Given the definition of $\Lambda(s,\eta)$ in~\eqref{def:Lambda} with $\eta= dw(q^\flat) y $, we obtain 
$$
\Lambda(s,\eta) = \frac r 2 s^2 + s\omega\cdot dw(q^\flat)y+  \frac{1}{2r}|\omega\cdot\eta|^2+\frac{1}{2r}|\omega^\perp\cdot\eta|^2\ln(\sqrt r|s| ).
$$
Moreover 
 $$\Gamma_0 y\cdot y = r^{-1} \left( ({\rm Id}_{\R^2} -\omega\otimes\omega) dw(q^\flat) y\right) \cdot \left(dw(q^\flat) y\right) =r^{-1} (dw(q^\flat) y\cdot\omega^\perp)^2.$$
 Therefore,  in $\widetilde \Sigma^k_{\eps}$
 \begin{align*}
 u^\eps(-s_0,y) 
&=   \vec V_\omega\,  {\rm Exp}\left(- i\Lambda(s,y) \right)
  {\rm Exp}\left(  \frac i\eps S^\flat _-+\frac i{4r} (\eta\cdot\omega^\perp)^2 \ln(\frac r {\eps})  + \frac i{2r} |\omega\cdot \eta|^2\right)u^{\rm in}_- (y)\\
&\qquad \;\; + \mathcal{O}\left((\sqrt\eps\delta^{-1}+\eps^{3/2} \delta^{-4} +\delta+\delta ^3\eps^{-1})(1+ |\ln\delta| )\right),
\end{align*}
which concludes the proof in view of~\eqref{eq:alpha1}. 
\end{proof}

\subsection{The outgoing solution}
We now compare $u^\eps(s)$ with a solution to the Landau-Zener model problem. 
Let $u$ be the solution of~\eqref{eq:LZ} for $\eta=\eta(y)$ and the initial data 
$$u(-s_0)={\rm e}^{ i\Lambda(-s_0,\eta)}\alpha_2^{\rm in}(\eta)  \vec V_\omega $$
where  $\alpha_2^{\rm in}(\eta)$ is given by~\eqref{eq:alpha1}, $R>0$, $\eta=\eta(y)$. 
We consider $\chi_0\in{\mathcal C}_0^\infty(\R^d,[0,1])$ such that $|\eta(y)| \leq c R$ when $y/R\in  \, {\rm supp} \chi_0$.
We consider 
the function 
$u_R(s)= \chi_0(y/R) u(s)$. Then, $u_R$ is the solution to~\eqref{eq:LZ} for $\eta=\eta(y)$ and the initial data 
\begin{equation}\label{def:u(s0)}
u_R(-s_0)={\rm e}^{ i\Lambda(-s_0,\eta)}\alpha_2^{\rm in}(\eta) \chi_0(y/R) \vec V_\omega  
\end{equation} 
This cut-off  allows us to use the scattering results of Lemma~\ref{scat:LZ} for $u_R$.
As noticed in Remark~\ref{rem:nov7}, we shall use the norms  $\widetilde \Sigma^k_\eps$ introduced in~\eqref{def:normnov7}.

\begin{lemma}\label{lem:compuepsandu}
Let $u_R(s)$ be  the solution of the Landau-Zener model problem~\eqref{eq:LZ} for $\eta=\eta(y)$ and the initial data $u_R(-s_0)$ given by~\eqref{def:u(s0)}, and $u^\eps(s)$ be the solution of~\eqref{eq:system}.  
Let $k\in\N$, then for all $N_0\in\N$
 and for all $s\in [-s_0,s_0]$, for all $R\geq 1$ with $R^2\sqrt\eps \ll 1$, $R\delta\ll 1$ and $R\eps^2\delta^{-4} \ll 1$,  in $\widetilde \Sigma^k_\eps(\R^d)$,
$$u^\eps(s)-u_R(s) = \mathcal{O}\left((\sqrt\eps\delta^{-1}+\eps^{3/2} \delta^{-4} +\delta+R\delta^3\eps^{-1}  +R^{-N_0})(1+ |\ln\delta| )\right).$$
\end{lemma}

\begin{remark}\label{rem:remark48}
We are going to take $\alpha = \frac 5{14}$ as in Remark~\ref{rem:alphachoice}, which implies $\delta^3\eps^{-1}=\eps^{\frac1{14}}$. We choose 
$R=\eps^{-\beta}$ with $\beta\in (0,\frac 1{14})$ small enough so that $R^2\sqrt\eps \ll 1$, $R\delta\ll 1$ and $R\eps^2\delta^{-4} \ll 1$. Since $R$ produces an error of size $R^3 \sqrt\eps \delta^{-1} \ll 1$ by Lemma~\ref{scat:LZ}, we additionally ask $R^3 \sqrt\eps \delta^{-1}\leq \eps^{\frac 1{14}}$. We choose 
 $N_0$ as large as necessary to ensure $R^{-N_0}\leq \eps^{\frac 1 {14}}$. We are then left with an approximation of order $ O(\eps^{\frac 1{14}-\beta})=O\left(\eps^{{\frac 1{14} }^-}\right)$.
\end{remark}

\begin{proof}
We set $r^\eps(s)=u^\eps(s)-u_R(s)$. 
We observe that 
using that $u^{\rm in}_-\in\mathcal S(\R^d)$ (see Proposition~\ref{prop:ingoing}), we deduce that 
we have in $\widetilde \Sigma^k_\eps$ and for any $N_0\in\N$, 
$$r^\eps(-s_0)=\mathcal{O}\left((\sqrt\eps\delta^{-1}+\eps^{3/2} \delta^{-4} +\delta + \delta^3\eps^{-1} +R^{-N_0})(1+ |\ln\delta|) \right)=\mathcal O(\varsigma_\eps) .$$
where we set for short $\mathcal O(\varsigma_\eps)=\mathcal{O}\left((\sqrt\eps\delta^{-1}+\eps^{3/2} \delta^{-4} +\delta+R\delta^3\eps^{-1}  +R^{-N_0})(1+ |\ln\delta| )\right)$. 
Besides, we have (with the notations of Lemma~\ref{lem:nov1}) 
$$i\partial_s r^\eps (s) - P^\eps(s)  r^\eps(s) = \sqrt\eps f^\eps(s,y)$$
where 
\begin{align*}
P^\eps(s)&= A(sE+dw(q^\flat)y) +\frac{\sqrt\eps}{2} \Delta + \sqrt\eps B^\eps(s,y)\\
f^\eps(s)& = 
\frac{1}{2} \Delta u_R(s) + B^\eps(s,y) u_R(s).
\end{align*}
We shall use the two following properties:
\begin{itemize}
\item[(i)]
By Lemma~\ref{lem:nov1}, there exist constants $C_0, C_1, C_\beta$, $|\beta|\geq 2$, such that  on the support of $u_R$ (where $|y|\leq c' R$, $c'>0$), and for $s\in[-s_0,s_0]$,  we have 
$$|B^\eps(s,y)|\leq  C_0 (Rs^2+|y|^2),\;\; |\nabla B^\eps(s,y)|\leq  C_1 ( R |y| + \delta s),\;\; |\partial_y^\beta B^\eps(s,y)|\leq C_\beta R^2 \eps^{\frac{|\beta|-2}2}.$$
\item[(ii)]
By  Lemma~\ref{lem:sigmakLZ}, $f^\eps$ satisfies the following: for all $\alpha,\beta\in\N^d$, there exists $C_{\alpha,\beta}$ such that 
\begin{equation}\label{fepsestimate}
\| y^\alpha \partial_y^\beta  f^\eps(s,y) \|_{L^2} \leq C_{\alpha,\beta} (R s_0^{|\beta|+2} +R^2 s_0^{(|\beta|-2)})\leq C_{\alpha,\beta} R s_0^{|\beta|+2}  
\end{equation}
\end{itemize}
where we used $R s_0^{-4}\leq 1$.
We prove by a recursive argument that  
\begin{equation}\label{rec:nov7} \sup_{s\in[-s_0,s_0]} \| y^\alpha \partial_y^\beta  r^\eps(s,y) \|_{L^2}=\mathcal O( \sqrt\eps\, R s_0^{|\alpha| + |\beta|+3}  ) + \mathcal O(\varsigma_\eps)
\end{equation}
which implies the Lemma since
$$\eps^{\frac{|\alpha|+|\beta|}2  } R\sqrt\eps  s_0 ^{|\alpha| + |\beta|+3} = R \sqrt\eps s_0^3=R\frac{\delta^3}\eps.$$
\smallskip

$\bullet$ $k=0$. 
An energy estimate gives
\begin{equation}\label{energynov7}
\| r^\eps(s) \|_{L^2} \leq C \sqrt\eps  \int_{-s_0}^s \| f^\eps(s') \|_{L^2} ds' + \mathcal O(\varsigma_\eps)\leq C R\sqrt\eps s_0 ^3 +\mathcal O(\varsigma_\eps) ,
\end{equation}
whence~\eqref{rec:nov7} for $k=0$. \\
\smallskip

$\bullet$ $k=1$. Using the equation satisfied by $r^\eps$, we write for $j\in\{1,\cdots , d\}$
\begin{align*}
(i\partial_s -P^\eps(s)) (y_j r^\eps)&= \sqrt\eps (y_j f^\eps) + \sqrt \eps ( \partial_{y_j} r^\eps),\\
(i\partial_s -P^\eps(s)) (\partial_{y_j} r^\eps)&= \sqrt\eps (\partial_{y_j} f^\eps) + \sqrt \eps ( \partial_{y_j }B^\eps  r^\eps) + A(\partial_{x_j}w(q^\flat) )r^\eps.
\end{align*}
Note that $|  \partial_{y_j }B^\eps  r^\eps|\leq C(R|y| +\delta s) |r^\eps|$ and $R\geq 1$.
 Using Point (i) above,
 an energy argument gives for some constant $c_1$
\begin{align*}
&\|  y r^\eps(s)\|_{L^2} + \| \nabla_y r^\eps(s)\|_{L^2} \leq \mathcal O(\varsigma_\eps)+ c_1 (R\sqrt\eps) \int_{-s_0}^s\left(  \|  y r^\eps(s')\|_{L^2} + \| \nabla_y r^\eps(s')\|_{L^2}\right) ds'  \\
&\;+  \sqrt\eps  \int_{-s_0}^s\left(  \|  y f^\eps(s')\|_{L^2} + \| \nabla_y f^\eps(s')\|_{L^2}\right) ds' ++c_1    \delta \sqrt \eps \int_{-s_0}^s \langle s'\rangle \|r^\eps(s')\|_{L^2} ds' + c _1\int_0^s \|r^\eps(s')\|_{L^2} ds'
\\
&\leq 2 c_1 R \sqrt\eps  s_0 \sup_{s\in[-s_0,s_0]} \left( \|  y r^\eps(s)\|_{L^2} + \| \nabla_y r^\eps(s)\|_{L^2}\right)  + c _1\int_0^s \|r^\eps(s')\|_{L^2} ds'
\\
&\qquad +  \sqrt\eps  \int_{-s_0}^s\left(  \|  y f^\eps(s')\|_{L^2} + \| \nabla_y f^\eps(s')\|_{L^2}\right) ds' 
+c_1   \delta \sqrt \eps    s_0  \int_{-s_0}^s  \|r^\eps(s')\|_{L^2} ds' + \mathcal O(\varsigma_\eps)
\end{align*}
where we have used $\sqrt\eps s_0\leq \delta$. 
Using $ R \sqrt\eps s_0\leq  R\delta \ll 1$,~\eqref{fepsestimate} and~\eqref{energynov7},  we deduce (changing the constant $c_1$ as necessary)
$$\sup_{s\in[-s_0,s_0]} \left( \|  y r^\eps(s)\|_{L^2} + \| \nabla_y r^\eps(s)\|_{L^2}\right) 
\leq c_1 (R \sqrt\eps  {s_0^4} + \delta \eps  R s_0^5) +  \mathcal O(\varsigma_\eps) \leq  c_1 R \sqrt\eps  {s_0^4} +  \mathcal O(\varsigma_\eps) ,
$$
where we have used $s_0\sqrt\eps \leq \delta $, whence~\eqref{rec:nov7} for $k=1$. \\
\smallskip 
$\bullet$ $k\rightarrow k+1$. 
We assume that there exists some $k\in\N$ such that for all $\ell\in\{0,\cdots , k\}$ 
$$\sup_{|\alpha|+|\beta|=\ell}\, \sup_{s\in[-s_0,s_0]}  \| y^\alpha \partial_y^\beta  r^\eps(s)\|_{L^2}\leq  c_k R \sqrt\eps s_0^{\ell+3}  $$
and that  for any  term of the form $y^\alpha \partial_y^\beta r^\eps$ with $|\alpha|+|\beta|=k$, we have   
$$(i\partial_s -P^\eps(s)) (y^\alpha \partial_y^\beta r^\eps)= \sqrt\eps y^\alpha \partial_y^\beta f^\eps 
+  \sum_{\ell=0}^{k}  \sum _{|\alpha'|+|\beta'|=\ell}c_{\alpha,\beta} ^\eps (s,y) y^{\alpha'} \partial_y^{\beta'} r^\eps 
$$
for some smooth  functions  $c_{\alpha,\beta}^\eps $ bounded together with their derivatives uniformly in $\eps$ with 
\begin{align*}
&| c_{\alpha',\beta'}^\eps(s,y) | +| \nabla _yc_{\alpha',\beta'}^\eps(s,y) | =\mathcal O(\sqrt\eps ( \langle s\rangle +|y|) +1)\;\; \mbox{for}\;\;
|\alpha'|+|\beta'|<k\\
&| c_{\alpha',\beta'}^\eps(s,y) | +| \nabla _yc_{\alpha',\beta'}^\eps(s,y) | =\mathcal O(R\sqrt\eps )  \;\; \mbox{for}\;\;
|\alpha'|+|\beta'|=k.
\end{align*}
Multiplying the equation by $y_j$ and applying $\partial_{y_j}$ for all $j\in\{1,\cdots, d\}$, one obtains that the form of the equation passes to the $(k+1)$-th step, which gives the norm estimate   by an energy argument. 
This concludes the proof. 
\end{proof}


\section{Proof of the main results}

\subsection{Proof of Theorem~\ref{theo:adiabatic}}
When the trajectory $\Phi^{t,t_0}_-(z_0)$ remains in the domain $\{|w(q)|> \delta\}$, the results of Proposition~\ref{prop:propagationt_out} apply and imply 
Theorem~\ref{theo:adiabatic}. 

\subsection{Proof of Theorem~\ref{theo:main}} 
 Inside the gap region, for $t\in [t^\flat-\delta,t^\flat+\delta]$, we apply Lemma 4.3 to pass through it. Then over the time $[t^\flat+\delta,t_0+T]$ we use Proposition 3.1 again to propagate further, with initial data found from the resulting solution $\psi^\epsilon(t^\flat+\delta)$ from Lemma 4.3 in the gap region. Then, we optimize $\delta$ and $R$ to get the best approximation in terms of $\eps$ according to Remarks~\ref{rem:alphachoice} and~\ref{rem:remark48}.
 
\subsubsection{Away from the gap region}
Given the initial assumptions of the theorem, we start at time $t_0$ far from the crossing point with initial data $\psi_0^\eps$ satisfying~\eqref{eq:data}. We consider the trajectory $\Phi_-^{t,t_0}(z_0)$ and the classical quantities that are associated with it. Applying Proposition 3.1 on $[t_0,t^\flat-\delta]$ with $u_+(t_0)=0$ and $u_-(t_0)=a$, we propagate the solution up to the gap region: at $t=t^\flat -s_0\sqrt\eps=t^\flat -\delta $, we have in $L^2(\R^d)$
\begin{align*}
\psi^\eps(t,x)= & \eps^{-d/4}{\rm e}^{\frac i\eps S_-(t,t_0,z_0) +\frac i\eps p_-(t)(x-q_-(t))}  \vec V_-(t, \Phi^{t,t_0}_-(z_0)) \\
&\qquad \times u_-\left(t, \frac{x-q_-(t)}{\sqrt\eps} \right) +\mathcal{O}\left((\sqrt\eps\delta^{-1}+\eps^{3/2} \delta^{-4})(1+ |\ln\delta| )\right).
\end{align*}
Using the minimal gap of the avoided crossing, $\delta_c\gg\sqrt\eps$, we are left with error terms $o(1)$.

\subsubsection{Passing through the gap region}
In this section, we compute an approximation of $\psi^\eps(t^\flat +\delta)$, thanks to the representation of $\psi^\eps$ as~\eqref{eq:ueps} which reduces the analysis to one of function $u^\eps(s)$ satisfying~\eqref{eq:system}.   Then, by Lemma~\ref{lem:compuepsandu}, it is possible to use  Lemma~\ref{scat:LZ} to link $u^\eps(+s_0)$ and $u^\eps(-s_0)$. Proposition~\ref{prop:ingoing} allow to identify the entering data at time $s=-s_0$ that we use in  Lemma~\ref{scat:LZ}  :   
 $\alpha_1=0$ and $\alpha_2$ satisfying~\eqref{eq:alpha1}.
We define  $(\alpha_1^{\rm out},\alpha_2^{\rm out})$ as
\begin{align}\label{def:omega}
\alpha_1^{\rm out}(\eta)&= 
-\overline b(r^{-1/2} \eta\cdot\omega^\perp) \alpha_2^{\rm in} (\eta) \\
\nonumber
\alpha_2^{\rm out}(\eta)&=
a(r^{-1/2} \eta\cdot\omega^\perp) \alpha_2^{\rm in}(\eta).
\end{align}
This follows from the formula giving  $\begin{pmatrix} \alpha_1^{\rm out} \\ \alpha_2^{\rm out} \end{pmatrix}$ in Lemma ~\ref{scat:LZ}.
Besides, we know that when $t=t^\flat +\delta =t^\flat +s_0\sqrt\eps$, $\psi^\eps(t)$ satisfies~\eqref{eq:ueps} with 
\begin{align*}
u^\eps (s_0,y)=  {\rm e}^{i\Lambda(s,\eta) } \alpha_1^{\rm out}(\eta)  \vec V_\omega^\perp &+{\rm e}^{-i\Lambda(s,\eta)  }\alpha_2^{\rm out}(y)  \vec V_\omega\\ &\qquad +\mathcal{O}\left((\sqrt\eps\delta^{-1}+\eps^{3/2} \delta^{-4} +\delta +R\delta^3\eps^{-1})(1+ |\ln\delta| )\right).
\end{align*}
This implies that for $t=t^\flat +\delta = t^\flat +\sqrt\eps s_0$, 
$$\psi^\eps(t,x)= \psi^\eps_+(t,x) +\psi^\eps_-(t,x) + \mathcal{O}\left((\sqrt\eps\delta^{-1}+\eps^{3/2} \delta^{-4} +\delta +R\delta^3\eps^{-1})(1+ |\ln\delta| )\right)$$
with 
\begin{align}
\label{eq:psi+}
\psi^\eps_+(t,x) &=  {\rm e}^{\frac i\eps S_0(t,t^\flat,z^\flat) +\frac i \eps (x-q_0(t))\cdot p_0(t) } \left. \left( {\rm e}^{-i\Lambda (s_0,\eta(y)) } \alpha^{\rm out}_2(\eta)\right)\right|_{y=\frac {x-q_0(t)}{\sqrt\eps}} \vec V_\omega,\\
\label{eq:psi-}
\psi^\eps_-(t,x)& =  {\rm e}^{\frac i\eps S_0(t,t^\flat,z^\flat) +\frac i \eps (x-q_0(t))\cdot p_0(t)} \left. \left( {\rm e}^{ + i\Lambda (s_0,\eta(y)) } \alpha^{\rm out}_1(\eta)\right)\right|_{y=\frac {x-q_0(t)}{\sqrt\eps}} \vec V_\omega^\perp
\end{align}
in $L^2(\R^d)$.
It remains to see why the functions $\psi^\eps_\pm(t,x)$ can be approximated by wave packets associated with the curves $\Phi^{t,t^\flat}_\pm(z^\flat)$ respectively. For this, we study  the asymptotics of the phase and of the profiles for $t>t^\flat$, as we did   in Section~\ref{subsec:ingoing} for times $t<t^\flat$.

\smallskip

Let us begin with the phases. 
We observe that the asymptotics of Lemma~\ref{prop:traj_asymp} and Lemma~\ref{lem:action} imply
that 
when $t=t^\flat +\sqrt\eps s$ with $s>0$ and $x=q_0(t)+\sqrt\eps y$, we have the pointwise estimates
$$\frac i\eps S_\pm(t,t^\flat,z^\flat)=\frac{i}{\eps}S_0(t,t^\flat,z^\flat)\mp ir s^2+\mathcal{O}(\sqrt\eps s^3)$$
and
\begin{align*}
&\frac i\eps p_\pm(t)\cdot (x-q_\pm(t))=\frac i\eps \left(p_0(t)\mp \sqrt\eps s ^tdw(q^\flat) \omega +\mathcal{O}( \eps s^2) \right)
 \cdot \left(x-q_0(t)  \pm \frac \eps 2 s^2 \, ^tdw(q^\flat)\omega  +\mathcal{O}(\eps^{3/2} s^3)\right)\\
&\qquad =\frac{i}{\sqrt\eps}p_0(t)\cdot y\mp i s\omega \cdot dw(q^\flat) y+ \mathcal{O}(\sqrt \eps s^2|y|)
\pm \frac{i}{2}s^2 \omega\cdot dw(q^\flat)p_0(t)+\mathcal{O}(\sqrt \eps s^3)+\mathcal{O}(\eps s^3)\\
\end{align*}
We observe that 
$$\omega \cdot dw(q^\flat) p_0(t)= \omega \cdot dw(q^\flat) p^\flat +\mathcal{O}(s \sqrt\eps )= r+O(s\sqrt\eps).$$
Therefore
$$ \frac i\eps p_\pm(t)\cdot (x-q_\pm(t))=
\frac i {\sqrt\eps}  y\cdot p_0(t) \mp i s\omega \cdot dw(q^\flat) y \pm \frac i2  r s^2 +\mathcal{O}(\sqrt\eps s^2|y|))+\mathcal{O}(\sqrt\eps s^3).$$
Then,
$$\frac{i}{\eps}S_\pm(t,t^\flat,z^\flat)+\frac{i}{\eps}p_\pm(t)\cdot (x-q_\pm(t))=\frac{i}{\eps}S_0(t,t^\flat,z^\flat) +\frac i {\sqrt\eps}  y\cdot p_0(t)  
$$
$$\mp\frac{ i}{2}r s^2\mp i s\omega \cdot dw(q^\flat) y  +\mathcal{O}(\sqrt\eps s^2|y|))+\mathcal{O}(\sqrt\eps s^3)$$
Given the definition of $\Lambda(s,\eta)$, ~\eqref{def:Lambda}, 
$$ i\Lambda(s,\eta) = \frac{i}{2r}|\omega\cdot \eta+rs|^2+\frac{i}{4r}|\omega^\perp\cdot\eta|^2\ln(rs^2),$$
we obtain
$$i\Lambda(s,\eta)=\frac{i}{2r}\left(|\omega\cdot\eta|^2+ 2r s\omega\cdot dw(q^\flat)y+r^2s^2\right)+
\frac{i}{4r}|\omega^\perp\cdot\eta|^2\ln(rs^2)
$$
$$=\frac{i}{2r}|\omega\cdot\eta|^2+i s\omega\cdot dw(q^\flat)y+\frac{i}{2}rs^2+
\frac{i}{4r}|\omega^\perp\cdot\eta|^2\ln(rs^2)$$
Using all of these ingredients together, we have the pointwise estimate
\begin{align*}
\frac{i}{\eps}S_\pm(t,t^\flat,z^\flat)&+\frac{i}{\eps}p_\pm(t)\cdot (x-q_\pm(t))
= \frac{i}{\eps}S_0(t,t^\flat,z^\flat) \\
&+\frac i {\sqrt\eps}  y\cdot p_0(t)  
\mp i\Lambda(s,\eta)
\pm \frac{i}{2r}|\omega\cdot\eta|^2
 \pm \frac{i}{2r}|\omega^\perp\cdot\eta|^2\left(\ln(\sqrt rs)\right)\\
&+\mathcal{O}(\sqrt\eps s^2|y|))+\mathcal{O}(\sqrt\eps s^3)
\end{align*}

\smallskip

At this stage of the proof, we are able to see the wave packet structure of the functions $\psi^\eps_\pm(t^\flat +\delta)$ defined in~\eqref{eq:psi+} and~\eqref{eq:psi-}.
Let us study more precisely $\psi^\eps_-(t,x)$, the computation for the other mode being similar. In view of the relations stated above,  we have in $L^2(\R^d)$
$$\psi^\eps_-(t,x) =  {\rm e}^{\frac i\eps S_-(t,t^\flat,z^\flat) 
+\frac i \eps (x-q_-(t))\cdot p_-(t)  
+  \frac{i}{2r}|\omega\cdot\eta|^2
+ \frac{i}{2r}|\omega^\perp\cdot\eta|^2\ln(\sqrt r s)  } \alpha^{\rm out}_1(\eta) \vec V_\omega^\perp$$
$$+\mathcal{O}((\sqrt\eps s^2 |y|+\sqrt\eps s^3)(1+|\ln (s\sqrt r)|)).$$
Here again 
\[
\frac 1r (\omega^\perp \cdot \eta(y) )^2 \ln (s\sqrt r) 
= \Gamma_0  y\cdot   y \ln (s\sqrt r)= \Gamma_0  y\cdot   y \ln (s\sqrt \eps) + \frac 1{2r} (\omega^\perp \cdot \eta(y) )^2 \ln (\frac r \eps ) 
.
\]
and we obtain for $t=t^\flat +\delta =t^\flat +s_0 \sqrt\eps $
\begin{align*}
&\psi^\eps_-(t,x) =V_\omega^\perp\,  {\rm Exp}\left(\frac i\eps S_-(t,t^\flat,z^\flat) +\frac i \eps (x-q_-(t))\cdot p_-(t)\right) \\
&\times \Bigl( {\rm Exp}\left(\frac i{2} \Gamma_0  y \cdot   y \ln (s\sqrt\eps) \right) {\rm Exp}\left( \frac i{4r } (\omega^\perp \cdot \eta(y) )^2 \ln (r/\eps )  +\frac i{2r} (\omega \cdot \eta)^2 \right) \alpha^{\rm out}_1(\eta)  \Bigr)\Bigr|_{y=\frac{x-q_0(t)}{\sqrt\eps}}  \\
&+\mathcal{O}((\sqrt\eps s^2 |y|+\sqrt\eps s^3)(1+|\ln (s\sqrt r)|))
\end{align*}
in $L^2(\R^d)$. 
Using the regularity of $\alpha^{\rm out}_1$, we deduce  $ \alpha^{\rm out}_1\left(\frac{x-q_0(t)} {\sqrt\eps}\right)= \alpha^{\rm out}_1\left(\frac{x-q_-(t)} {\sqrt\eps}\right)+\mathcal{O}( \sqrt\eps s^2)$ with $\mathcal{O}( \sqrt\eps s^2)=\mathcal{O}(\delta^2 \eps^{-1/2})$, we identify a wave packet approximation in $L^2(\R^d)$
$$\psi^\eps_-(t,x) =  {\rm e}^{\frac i\eps S_-(t,t^\flat,z^\flat) }
{\rm WP}^\eps_{ \Phi_-^{t,t^\flat}(z^\flat)} \left( 
{\rm e}^{\frac i{2} \Gamma_0  y \cdot   y \ln (s\sqrt\eps )}
{\rm e}^{ \frac i{4r } (\omega^\perp \cdot \eta(y) )^2   \ln (r/\eps  ) 
+ \frac i{2r} (\omega \cdot \eta )^2 } \alpha^{\rm out}_1(\eta)\right) V_\omega^\perp $$
$$+ \mathcal{O}(\sqrt\eps s^2 (1+|y|))+\mathcal{O}(\sqrt\eps s^3).$$
For $t\in [t^\flat -\delta, t^\flat+\delta]$, $  \mathcal{O}(\sqrt\eps s^2 (1+|y|))+\mathcal{O}(\sqrt\eps s^3)=  \mathcal{O}(\eps^{-1/2} \delta^2 (1+|y|))+\mathcal{O}(\eps^{-1} \delta^3)$. Using 
$y=\frac{x-q(t)}{\sqrt\eps}$ for this region, we are left with the error terms $\mathcal{O}(\delta^3\eps^{-1})$.
In view of~\eqref{eq:profout}, this suggests that we set
\begin{align*}
u_-^{\rm out} (y)&  = {\rm Exp} \left(\frac i{4r}(\omega^\perp \cdot \eta(y) )^2   \ln(r/\eps )  +\frac i{2r} |\omega\cdot \eta|^2
\right)\alpha_1^{\rm out}(\eta)\\
&= - {\rm Exp} \left(\frac i{4r}(\omega^\perp \cdot \eta(y) )^2  \ln(r/\eps )  +\frac i{2r} |\omega\cdot \eta|^2
\right)  \overline b(r^{-1/2} \eta\cdot\omega^\perp) \alpha_2^{\rm in}(\eta) 
\end{align*}
A similar computation for the $plus$-mode gives 
\begin{align*}
u_+^{\rm out} (y)&= {\rm Exp} \left(- \frac i{4r} (\omega^\perp \cdot \eta(y) )^2   \ln(r/\eps)  -\frac i{2r} |\omega\cdot \eta|^2
\right)\alpha_2^{\rm out}(\eta) ,\\
& = {\rm Exp}\left(- \frac i{4r} (\omega^\perp \cdot \eta(y) )^2   \ln(r/\eps)  -\frac i{2r} |\omega\cdot \eta|^2
\right)
a(r^{-1/2} \eta\cdot\omega^\perp) \alpha_2^{\rm in} (\eta).
\end{align*}
In view of~\eqref{eq:alpha1}, we deduce 
\begin{align*}
u_+^{\rm out} (y)&=  {\rm e}^{\frac i\eps S^\flat _-} a(r^{-1/2}\eta\cdot\omega^\perp) 
u_-^{\rm in} (y),
\\
u_-^{\rm out} (y)&= - {\rm e}^{\frac i\eps S^\flat _-}  {\rm Exp} \left(\frac i{2r}(\omega^\perp \cdot \eta(y) )^2  \ln(r/\eps )  +\frac i{r} |\omega\cdot \eta|^2
\right)  \overline b(r^{-1/2}\eta\cdot\omega^\perp)u_-^{\rm in} (y).
\end{align*}
which is equivalent to~\eqref{transition}.

\subsubsection{Leaving the gap region}
We define $u_\pm(t,y)$ for $t\geq t^\flat+\delta $ as the solution of~\eqref{def:profile} satisfying~\eqref{eq:profout}. Then, we have~\eqref{eq:appt>tflat} when $t=t^\flat +\delta$ and the result for $[t\in t^\flat+\delta, T]$  comes by applying Proposition~\ref{prop:propagationt_out}.

\subsection{Proof of Corollary~\ref{cor:wigner}}
Since for $t\in (t^\flat, t^\flat +T)$, we have $\Phi^{t,t^\flat}_+(z^\flat)\not= \Phi^{t,t^\flat}_-(z^\flat)$, any Wigner measure of $(\psi^\eps(t))_{\eps>0}$ is of the form~\eqref{eq:mut}. Besides the coefficients $c_+$ and $c_-$ are limits in $\eps$ of $\| u_+(t)\|^2_{L^2}$ and $\| u_-(t)\|^2_{L^2}$ respectively. We focus on $c_+$ (the proof for $c_-$ is similar). We have 
$$\| u_+(t)\|^2_{L^2}= \| u_+^{\rm out}\|^2_{L^2}= \| a(\eta_2)u^{\rm in}_+\|^2_{L^2} + \| b(\eta_2)u^{\rm in}_-\|^2_{L^2}  -2{\rm Re}\left( {\rm e}^{ \frac i\eps (S^\flat_+-S^\flat_-)}\gamma_\eps\right)$$
with 
$$\gamma_\eps =\int_{\R^d} a(\eta_2(y)) b(\eta_2(y))u^{\rm in}_+(y) \overline {u^{\rm in}_-(y) }{\rm e}^{i\theta_\eps(\eta(y))}dy.$$
In view of $|b(\eta_2)|^2=1-a(\eta_2)^2$, we have $ \| b(\eta_2)u^{\rm in}_-\|^2= \| \sqrt{1-a(\eta_2)}u^{\rm in}_-\|^2$. Moreover, by~\eqref{def:theta} and using $b(0)=0$, the term
$\gamma_\eps$ writes
$$ \gamma_\eps= \int \eta_2(y) f(y) {\rm e}^{-\frac i{2r} \eta_2(y)^2 \ln \eps} dy$$
for some smooth function $f$. 
Together with $\eta_2(y)= \omega^\perp \cdot dw(q^\flat)y$ where $dw(q^\flat)$ of rank $2$, one  writes 
$$\int_{\R^d}  \eta_2(y) f(y) {\rm e}^{-\frac i{2r} \eta_2(y)^2 \ln \eps} dy=\frac {r}{i \ln \eps } \int_{\R^d} |\,^t dw(q^\flat) \omega^\perp|^{-2} \, ^t dw(q^\flat) \omega^\perp\cdot \nabla_y  f(y){\rm e}^{-\frac i{2r} \eta_2(y)^2 \ln \eps} dy,$$
which implies $c_+=  \| a(\eta_2)u^{\rm in}_+\|^2 + \| \sqrt{1-a(\eta_2)}u^{\rm in}_-\|^2 $.


\appendix

\section{Semi-classical pseudo-differential calculus}\label{app:pseudo}

This section contains results about semi-classical pseudo-differential operators.
We consider matrix-valued functions  $a\in\mathcal C^\infty(\R^{2d},\C^{2,2})$  which are bounded, as well as their derivatives.
Then, one defines the Weyl semi-classical pseudo-differential operator of symbol $a$ as 
\begin{equation}\label{def:oppseudo}
\op_\eps(a) f(x)= {(2\pi\eps)^{-d}}\int_{\R^{2d}} {\rm e}^{{i\over \eps} \xi\cdot (x-y)}  a\left({x+y\over 2},\xi\right)f(y)dy\, {d\xi},\quad \forall f\in{\mathcal S}(\R^d,\C^2).
\end{equation} 
The reader may found proofs of the results presented here in~\cite{Dimassi1999,Zwobook,F14}, for instance. 
In the following, we denote by $z=(x,\xi)\in \R^{2d}$ the variable of the functions $a\in\mathcal C^\infty(\R^{2d},\C^{2,2})$.

\smallskip

The Calder\' on-Vaillancourt Theorem \cite{CV}  ensures the existence of constants~$C_d, n_d>0$ such that for every $ a\in \mathcal C^\infty(\R^d,\C^{2,2})$, bounded with bounded derivatives, one has
\begin{equation}\label{eq:pseudo}
\| \op_\eps(a)\|_{{\mathcal L}(L^2(\R^d,\C^2))}\leq C_d\, 
N_d^\eps(a),
\end{equation}
where
$$
N_d^\eps(a):=\sum_{\alpha\in\N^{2d},|\alpha|\leq n_d} \eps^{|\alpha|\over 2}\sup_{\R^{2d}}|\partial_{z}^\alpha a|
$$
with $n_d = M d$ for some constant $M\geq 1$ (see \cite{Zwobook} for example). It is then easy to check that, since $\eps \in (0,1]$,
\begin{equation}
\label{Nd}
\sqrt\eps N_d^\eps (\partial_{z_j} a)\leq N^\eps_{d+1} (a) \;    \textrm{ for all } \; j\in \{1,\cdots, 2d\}.
\end{equation}

\smallskip

Matrix-valued pseudodifferential operators enjoy a symbolic calculus:

\begin{proposition}\label{prop:symbol}
Let $a,b\in\mathcal C_0^\infty(\R^d,\C^{2,2})$, then
$$\op_\eps(a)\op_\eps(b)  = \op_\eps(ab)+\eps R^{(1)}_\eps(a,b)= \op_\eps(ab)+ \frac{\eps}{2i} \op_\eps(\{a,b\})+\eps^2 R^{(2)}_\eps(a,b),$$
with $\{a,b\}=\sum_{j=1}^d\partial_{\xi_j} a\,  \partial _{x_j} b-\partial _{x_j}a\, \partial_{\xi_j} b$ and 
$$\| R^{(j)}_\eps(a,b)\|_{{\mathcal L}(L^2(\R^d,\C^2))}\leq C \,\sup_{|\alpha|+|\beta|=j} N_d^\eps(\partial_\xi^\alpha \partial_x^{\beta}  a) N_d^\eps(
 \partial_\xi^\beta \partial_x^{\alpha}  b),\quad j\in\{1,2\},$$
for some constant $C>0$ independent of $a$, $b$ and $\eps$.
\end{proposition}

\begin{remark}\label{rem:locpseudo}
When $a=1$ on the support of $b$, pushing the Taylor expansion at larger order, one gets for $N\in\N^*$, 
$$\op_\eps(a)\op_\eps(b)  = \op_\eps(b)+\eps^N  R^{(N)}_\eps(a,b)$$
with 
$$\| R^{(N)}_\eps(a,b)\|_{{\mathcal L}(L^2(\R^d,\C^2))}\leq C \,\sup_{|\gamma|=|\gamma'|=N} N_d^\eps(\partial_z^\gamma  a) N_d^\eps(
 \partial_z^{\gamma'}  b).$$
\end{remark}

\begin{remark} \label{rem:calculpseudo}
For general (non-commuting) symbols~$a$ and~$b$, Lemma~\ref{prop:symbol} implies
$$\left[\op_\eps(a),\op_\eps(b)\right] =\op_\eps([a,b])+ {\eps\over 2i}(\op_\eps(\{a,b\}) -\op_\eps(\{b,a\}))+\eps^2( R_\eps^{(2)}(a,b)-R_\eps^{(2)}(b,a)).$$
However, the term of order $\eps^2$ in this expansion  has symmetries so that if $a$ and $b$ commutes, for example because $a$ is scalar valued, 
$$[\op_\eps(a),\op_\eps(b) ] =\frac{\eps}{i} \op_\eps(\{a,b\})+\mathcal{O}\left(\eps^3
\sup_{|\gamma|=|\gamma'|=3} N_d^\eps(\partial_z^{\gamma}  a) N_d^\eps(
 \partial_z^{\gamma'}  b)
\right).$$
\end{remark}

Note also that for $1\leq j\leq d$ the commutation relations between $x_j$ or $\eps D_{x_j}$ and $\op_\eps(a)$ writes 
\begin{equation}\label{eq:bracketsxj}
[x_j,{\rm op}_\eps(a) ]= \eps i {\rm op}_\eps(\partial_{\xi_j} a)\;\;
\mbox{and}\;\;
[\eps D_{x_j},{\rm op}_\eps(a) ]= - \eps i {\rm op}_\eps(\partial_{x_j} a).
\end{equation}
Using these relations and the estimates in $L^2(\R^d)$, it is possible to prove estimates in $\Sigma^k_\eps$ that are uniform in $\eps$. 

\begin{lemma} 
Let $\eps\in (0,1]$ and $k\in \N$.
There exist constants $C_{d,k}$ and $c_k$ such that for all $a\in\mathcal C^\infty_0(\R^d)$, we have  in $\Sigma^k_\eps$:
\begin{equation}\label{est:CVSigmak}
\| \op_\eps(a)\|_{{\mathcal L}(\Sigma^k_\eps)}\leq C_{d,k} 
N_{d+k}^\eps(a).
\end{equation}
\end{lemma}

\begin{proof}
The proof is based on~\eqref{eq:bracketsxj} and  a recursive argument. For $a\in\mathcal C^\infty_0(\R^d)$, $f\in\mathcal S(\R^d)$ and $j\in\{1,\cdots d\}$,
\begin{align*}
\| x_j{\rm op}_\eps(a) f \| _{\Sigma^{k-1}_\eps}&\leq \| {\rm op}_\eps(a) (x_jf)\| _{\Sigma^{k-1}_\eps}+ \eps \| {\rm op}_\eps(\partial_{\xi_j} a) f\| _{\Sigma^{k-1}_\eps},\\
\|\eps  \partial_{x_j}({\rm op}_\eps(a) f) \| _{\Sigma^{k-1}_\eps}&\leq \| {\rm op}_\eps(a) (\eps \partial_{x_j}f)\| _{\Sigma^{k-1}_\eps}+ \eps \| {\rm op}_\eps(\partial_{x_j} a) f\| _{\Sigma^{k-1}_\eps}.
\end{align*}
Therefore, there exists a constant $c'$ such that 
\begin{multline*}
\| {\rm op}_\eps(a) f \| _{\Sigma^{k}_\eps}\leq c' \| {\rm op}_\eps(a) \|_{ \mathcal L(\Sigma^{k-1}_\eps)} \| f\| _{\Sigma^{k}_\eps} \\ + c' \sum_j \eps \left( \| {\rm op}_\eps(\partial_{\xi_j} a) \|_{ \mathcal L(\Sigma^{k-1}_\eps)}+\| {\rm op}_\eps(\partial_{x_j} a) \|_{ \mathcal L(\Sigma^{k-1}_\eps)}\right) \|f\| _{\Sigma^{k-1}_\eps}.
\end{multline*}
One then concludes by starting the recursive argument from~\eqref{eq:pseudo} and using \eqref{Nd}.
\end{proof}

  \section{Localization of wave packets}\label{appB}
  
  The wave packets defined in~\eqref{wpdef} enjoy localization properties. We use here the notations introduced in Appendix~\ref{app:pseudo} and we use the notation $\widehat a$  for denoting (non semiclassical) pseudodifferential operators, $\widehat a={\rm op}_1(a)$.

\begin{lemma}\label{lem:prelimWP}
Let $z_0=(q,p)\in\R^{2d}$, $\varphi\in{\mathcal S}(\R^d)$ and $a\in{\mathcal C}^\infty(\R^{2d})$.
 Then, 
$$ {\rm op}_\eps( a) \, {\rm WP} _{z_0}^\eps(\varphi)=  {\rm WP} _{z_0}^\eps\left(\widehat{a(z_0+\sqrt\eps z)}\,  \varphi\right).$$
\end{lemma}

\begin{proof} The result comes from change of variables.
\end{proof}

This Lemma has several important consequences.

\begin{lemma}\label{lem:localisation}
Let $\eps\in(0,1]$, $z_0=(q,p)\in\R^{2d}$, $\varphi\in{\mathcal S}(\R^d)$ and $a\in{\mathcal C}^\infty(\R^{2d})$ bounded together with its derivatives.
Then, we have the following properties:
\begin{enumerate}
\item 
For all $n_0,k\in\N$, there exists a constant $C_k$ such that 
$$\left\| {\rm op}_\eps( a) \, {\rm WP}  _{z_0}^\eps(\varphi)-  {\rm WP}  _{z_0}^\eps\left(\widehat{ P^{(n_0)} _a(z\sqrt\eps ) } \varphi\right)\right\|_{\Sigma^{k}_\eps} \leq C_k\,\eps^{\frac {n_0+1} 2} \, N^\eps_{d+k+n_0+1}(d^{n_0+1}a)
 \| \varphi\|_{\Sigma^{k+n_0+1}} $$
where $z\mapsto P^{(n_0)}_a(z)$ is the Taylor polynomial at order $n_0$ of $a$ in $z_0$: 
  $$P^{(n_0)}_a (z )=a(z_0)+\nabla a(z_0)\cdot z+ \frac{1}{2}\nabla^2 a(z_0) z\cdot z+
...+\frac{1}{(n_0)!} d^{n_0} a(z_0)[z]^{n_0}.$$
\item Moreover, assume that $a(z)=1$ for $|z-z_0|\leq 1$ and $a(z)=0$ if $|z-z_0|>2$. Then, for any $n\in\N$, there exists a constant $C'_{k,n}$ such that 
$$\| {\rm WP} _{z_0}^\eps(\varphi)- {\rm op}_\eps(a)\, {\rm WP} _z^\eps(\varphi)\|_{\Sigma^k_\eps} \leq C'_{k,n}\, \eps^{n/2} N^\eps_{d+k+n}(d^na)  \| \varphi\|_{\Sigma^{k+n}}.$$
\end{enumerate}
\end{lemma}

\begin{proof}
Let us prove Point (1). 
Applying Lemma~\ref{lem:prelimWP},
\begin{align*}
\|{\rm op}_\eps( a) \, {\rm WP}  _{z_0}^\eps(\varphi) & -{\rm WP}  _{z_0}^\eps(\widehat{ P^{(n_0)}_a(z\sqrt\eps ) } \varphi)\|_{\Sigma_\eps^k}
=  \|{\rm WP}  _{z_0}^\eps( (\widehat{a(z_0+\sqrt\eps z)}\,-\widehat{ P^{(n_0)}_a(z\sqrt\eps ) }) \varphi )\|_{\Sigma_\eps^k}.
\end{align*}
There exists a constant $C'_k$ such that for all profiles $\varphi\in \mathcal S(\R^d)$,
  $$\| {\rm WP}  _{z_0}^\eps(\varphi)\|_{\Sigma^k_\eps}\leq C'_k \|\varphi\|_{\Sigma^k}$$
hence
  $$\|{\rm WP}  _{z_0}^\eps((\widehat{a(z_0+\sqrt\eps z)}\,-\widehat{ P^{(n_0)}_a(z\sqrt\eps ) }) \varphi )\|_{\Sigma_\eps^k}\leq C'_k\| (\widehat{a(z_0+\sqrt\eps z)}- \widehat{ P^{(n_0)}_a(z\sqrt\eps ) }) \varphi\| _{\Sigma^k}.$$
We have
$$a(z_0+\sqrt\eps z)-  P^{(n_0)}_ a(z\sqrt \eps)=
  \eps^ {\frac {n_0+1}{2}}  r(\sqrt\eps  z)[z]^{n_0+1}$$
where $r\in{\mathcal C} ^\infty (\R^{2d})$ is a smooth tensor of order $n_0+1$ that is bounded with bounded derivatives
$$r(z)= \frac {1} {n_0!} \int_0^1 d^{(n_0+1)}a (z_0+sz)(1-s)^{n_0} ds.$$
We state the following auxiliary claim: 

\smallskip
`` \textit{Consider a smooth function $b$ that is smooth, bounded with bounded derivatives. 
Then, for all $k,n\in\N$ there exists a constant $c'_{k,n}$ such that for all $|\alpha|\leq n$,}
\begin{equation}\label{est:rest_taylor}
\| {\rm op}_1( b(\sqrt\eps  z)z^\alpha) \|_{\mathcal L(\Sigma_\eps^{k+n},\Sigma_\eps^k)} \leq c'_{k,n} N^\eps_{d+k+n} (b). \; \textrm{''}
\end{equation}
\smallskip
Applying the claim to  $r(\sqrt\eps z) [z]^{n_0+1}$, with $n=n_0+1$, we obtain 
$$\| {\rm op}_1( r(\sqrt\eps  z)[z]^{n_0+1}) \|_{\mathcal L(\Sigma_\eps^{k+n_0+1},\Sigma_\eps^k)} \leq c'_{k} N^\eps_{d+k+n_0+1} (d^{n_0+1}a),$$
which is enough to complete the proof of Point (1).
\smallskip

We now turn to the proof of the claim. It relies on a recursive argument on $n$. When $n=0$, the estimate~\eqref{est:CVSigmak} gives 
$$\| {\rm op}_1( b(\sqrt\eps  z)) \|_{\mathcal L(\Sigma_\eps^k)} \leq c_k\, N^1_{d+k} (b(\sqrt\eps \cdot))\leq c'_{k,0} N^\eps_{d+k} (b).$$
Let us now assume that we have proved the estimate~\eqref{est:rest_taylor} for all indices smaller than some $n\in\N$ and let us consider $\alpha\in \N^{2d}$ with $|\alpha|=n+1$. Then, $\alpha$ has at least one non-zero component. Let~$\alpha_{ j}$ be such a component, with $ j\in \left\lbrace 1, \cdots, 2d \right\rbrace$. Either $ j \in \left\lbrace 1, \cdots, d \right\rbrace$ and $z^{\bf 1_j}= x_j$, or $ j \in \left\lbrace d+1, \cdots, 2d \right\rbrace$ and $z^{\bf 1_j}=\xi_j$.
 We consider the first case and a similar argument will work in the other one.
For $\alpha\in\N^{2d}$, we have
$$z^{\alpha} = x_1^{\alpha_1} \dots x_d^{\alpha_d}\xi_1^{\alpha_{d+1}} \dots\xi_d^{\alpha_{2d}} $$
so that $z^\alpha= z^{\alpha-{\bf 1}_j}z^{\bf 1_j}=z^{\alpha-{\bf 1}_j} x_j$.
Using~\eqref{eq:bracketsxj} and Proposition \ref{prop:symbol}, we then write for $f\in\mathcal S(\R^d)$,
\[
{\rm op}_1( b(\sqrt\eps  z)z^{\alpha}) f  = {\rm op}_1( b(\sqrt\eps z)z^{\alpha-{\bf 1}_j})\,(  x_j f)-\dfrac{1}{2i}{\rm op}_1\left(\partial_{\xi_j} \left(b(\sqrt\eps  z) z^{\alpha-{\bf 1}_j} \right) \right)f.
\]
We deduce
\begin{align*}
\| {\rm op}_1( b(\sqrt\eps  z)z^{\alpha})  f\|_{\Sigma^k_\eps} \leq & \| {\rm op}_1( b(\sqrt\eps  z)z^{\alpha-{\bf 1}_j})\|_ {\mathcal L(\Sigma_\eps^{k+n},\Sigma_\eps^{k})}\| x_jf\|_{\Sigma^{k+n} _\eps} \\
&\;+ \frac 12 \sqrt\eps  \|  {\rm op}_1(\partial_{\xi_j} b(\sqrt\eps  z)z^{\alpha-{\bf 1}_j})\|_ {\mathcal L(\Sigma_\eps^{k+n},\Sigma_\eps^k)}\| f\|_{\Sigma^{k+n} _\eps}\\
&\; 
+ \frac 12  \|  {\rm op}_1( b(\sqrt\eps  z)z^{\alpha-{\bf 1}_j-{\bf 1}_{j+d}})\|_ {\mathcal L(\Sigma_\eps^{k+n},\Sigma_\eps^k)}\| f\|_{\Sigma^{k+n} _\eps}
 ,
\end{align*}
where the last term is there only if the $(j+d)$-th component of $\alpha-{\bf 1}_j$ is non zero. One then deduces the result from the recursive assumption, which concludes the proof of the claim, and thus of Point~(1).

\smallskip
Finally, to prove Point (2), we only need to observe that since $a$ is identically equal to $1$ close to $z_0$, its Taylor polynomial  $P^{(n)}_a(z)$ is  equal to $1$  for all $n\in\N$. We then apply Point (1) with $n_0=n-1$. 

\end{proof}

\section{Matricial relations}

For $w=(w_1,w_2)\in\R^2$ and $u=(u_1,u_2)\in\R^2$, the matrices $A(u)$ and $A(w)$ defined in~\eqref{def:A} satisfy
$$A(w)A(u) = \begin{pmatrix} w\cdot u & w\wedge u \\ -w\wedge u & w\cdot u\end{pmatrix}$$
(recall $w\wedge u=w_1 u_2 - w_2 u_1$).

\subsection{The $B_\pm$ matrices}
We look more closely at the matrices $B_\pm$ introduced in~\eqref{def:B+-}.
We recall that for $\xi\in\R^d$, $\xi\cdot\nabla$ denotes the (scalar) operator 
$\displaystyle{\xi\cdot\nabla =\sum_{j=1}^d \xi_j\partial_{\xi_j}.}$

\begin{lemma}\label{lem:computation}
For $\xi\in \R^d$ and $w\in {\mathcal C}^\infty(\R^d,\R^2)$,
$$\xi\cdot \nabla \Pi_+ = -\frac{\xi\cdot \nabla w (x) \wedge w(x)}{ 2|w(x)|^3} A( w^\perp(x)),\;\; w^\perp=(-w_2,w_1). $$
Therefore,
$\displaystyle{\Pi_-(x) \xi\cdot \nabla \Pi_+(x)-\Pi_+(x) \xi\cdot\nabla \Pi_+(x) = - \frac{\xi\cdot \nabla w (x) \wedge w(x)}{2|w(x)|^2 }
\begin{pmatrix}  0 & 1 \\ -1 & 0\end{pmatrix}.}$
\end{lemma}

\begin{proof}
Since $\Pi_+(x)=\frac 12 \left({\rm Id}_{\R^2} + A\left(\frac{w(x)}{|w(x)|}\right)\right)$, a straightforward computation gives
\begin{align*}
\xi\cdot \nabla \Pi_+(x)= & \frac {1}{2|w(x)|} \left( A(\xi\cdot \nabla w(x)\right)- \frac {w(x)\cdot (\xi\cdot\nabla w(x))}{|w(x)|^2} A(w(x))\\
= & \frac 1 {2|w(x)|^3}  (w_2\,\xi \cdot \nabla w_1 -w_1\, \xi\cdot \nabla w_2) A(w_2,-w_1)
\end{align*}
whence the first formula. Then, we write 
\begin{align*}
\Pi_-(x) \xi\cdot \nabla \Pi_+(x)- \Pi_+(x) \xi\cdot\nabla \Pi_+(x)
&=  -\frac{\xi\cdot \nabla w (x) \wedge w(x)}{2 |w(x)|^3} \left( \Pi_-(x) A( w^\perp(x))-\Pi_+(x) A(w^\perp(x))\right)\\
&=  -\frac{\xi\cdot \nabla w (x) \wedge w(x)}{2 |w(x)|^4}A(w(x)) A(w^\perp(x))\\
&= - \frac{\xi\cdot \nabla w (x) \wedge w(x)}{2|w(x)|^2 }\begin{pmatrix}  0 & 1 \\ -1 & 0\end{pmatrix}.
\end{align*}
\end{proof}

\subsection{Superadiabatic projectors}\label{sec:superadiab}

In this section we use the semi-classical pseudodifferential operators introduced in Appendix~\ref{app:pseudo} and we denote by $a\,\sharp_\eps \,b$ the symbol of the operator ${\rm op}_\eps(a)\circ {\rm op}_\eps(b)$. 

\begin{lemma}\label{lem:C2}
There exist matrix-valued functions $\mathbb P^{(1)}_\pm$, $\mathbb P^{(2)}_\pm$, $\Omega_\pm^{(1)}$ and $\Omega_\pm^{(2)}$, that are smooth outside $\Upsilon$ and  such that the function
 $$\Pi_\pm^\eps(x,\xi) = \Pi_\pm(x) +\eps \mathbb P_\pm^{(1)}(x,\xi) + \eps^2 \mathbb P_\pm^{(2)}(x,\xi),$$
satisfies
 \begin{equation}\label{eq:sharp}
 \Pi^\eps_\pm \sharp_\eps H = (h_\pm +\eps \,\Omega^{(1)}_\pm +\eps^2\,\Omega^{(2)}_\pm) \sharp_\eps \Pi^\eps_\pm + \eps^3 R_\eps(x,\xi).
\end{equation}
Besides, 
 for all  $\alpha,\beta \in \N^d$, there exists constants 
 $C_{\alpha,\beta},\;p_\alpha>0$ such that  for  all $(x,\xi)\in\R^{2d}\setminus \Upsilon$,
 \begin{equation}\label{estimate:Reps}
  |\partial_x^\alpha\partial_\xi^\beta R_\eps(x,\xi)| \leq C_{\alpha,\beta}
 \langle x\rangle^{(|\alpha|+3)(1+n_0)} |w(x)|^{-|\alpha|-5}
 \end{equation}
 (where $n_0$ controls the gap at infinity, see~\eqref{hyp:gapinfinity}).
 Moreover, the following properties hold
 \begin{enumerate}
\item One has $$\mathbb P^{(1)}_\pm (x,\xi)=\pm \mathbb P(x,\xi),\;\; \Omega_\pm^{(1)}(x,\xi)=\Omega(x,\xi)$$
where $\mathbb P$ and $\Omega$ are the linear functions in $\xi$ defined respectively  in~\eqref{def:bigP} and~\eqref{def:Omega}. They are homogeneous functions in $w$ of degree $-1$ and $-2$ respectively.
\item 
 The matrices $\mathbb P_\pm^{(2)}$ and $\Omega^{(2)}_\pm$ are polynomial functions of order $2$ of the variable $\xi$ and 
for $(x,\xi)\in\R^{2d}\setminus \Upsilon$, for all  $\alpha,\beta \in \N^d$, there exists 
 $C_{\alpha,\beta}>0$ such that  
 \begin{align*}
  |\partial_x^\alpha\partial_\xi^\beta  \mathbb P^{(2)}_\pm(x,\xi) | & \leq C_{\alpha,\beta}\langle \xi\rangle ^2
 \langle x\rangle^{(|\alpha|+2)(1+n_0)} |w(x)|^{-|\alpha|-4},\\
  |\partial_x^\alpha\partial_\xi^\beta \Omega ^{(2)}_\pm(x,\xi) | & \leq C_{\alpha,\beta}\langle \xi\rangle^2
 \langle x\rangle^{(|\alpha|+2)(1+n_0)} |w(x)|^{-|\alpha|-3}.
 \end{align*}
  \end{enumerate} 
\end{lemma}

\begin{proof}
We use the calculus of $a\sharp_\eps b$ detailed in Proposition~\ref{prop:symbol} and the observations of Remark~\ref{rem:calculpseudo}. We have 
\begin{align*}
&\Pi^\eps_\pm \sharp _\eps H= \Pi_\pm H +\eps ( \mathbb P^{(1)}_\pm H + \frac 1{2i} \{ \Pi_\pm, H\}) + \eps^2 (  \mathbb P^{(2)}_\pm H + \frac 1{2i} \{ \mathbb P^{(1)}_\pm, H\} + d_\pm) +\eps^3 r^1,\\
&(h_\pm +\eps \,\Omega^{(1)}_\pm +\eps^2\,\Omega^{(2)}_\pm) \sharp_\eps \Pi^\eps_\pm= h_\pm \Pi_\pm 
+ \eps (h_\pm \mathbb P^{(1)}_\pm  + \frac 1{2i} \{h_\pm, \Pi_\pm\} +\Omega^{(1)}_\pm \Pi_\pm)\\
&\qquad\qquad + \eps^2 ( h_\pm \mathbb P^{(2)}_\pm  + \frac 1{2i} \{h_\pm, \mathbb P^{(1)}_\pm\} +\frac 1{2i}\{\Omega^{(1)}_\pm,\Pi_\pm\} + \Omega^{(1)} _\pm\mathbb P^{(1)} _\pm +\Omega_\pm^{(2)}\Pi_\pm + d_\pm)
\end{align*}
where $r^1$ and $r^2$ involves derivatives of order 3 of $\Pi_\pm$, of order 2 of $\mathbb P_\pm^{(1)}$ and of order 1 of $\mathbb P_\pm^{(2)}$ and~$d_\pm$ comes from the computations 
$$   \frac{|\xi|^2}2\sharp_\eps \Pi_\pm=   
\frac{|\xi|^2}2 \Pi_\pm + \frac\eps {2i } \left\{ \frac{|\xi|^2}2, \Pi_\pm\right \}+\eps^2 d_\pm,\;\;
\Pi_\pm \sharp_\eps   \frac{|\xi|^2}2 =   \frac{|\xi|^2}2\Pi_\pm - \frac\eps{2 i} \left\{ \frac{|\xi|^2}2, \Pi_\pm \right\}+\eps^2 d_\pm.
$$
We 
deduce that in order to realize equation~\eqref{eq:sharp}, we only need to equalize the terms of order $\eps$ and~$\eps^2$ on both developments (indeed $\Pi_\pm H= h_\pm \Pi_\pm$). We obtain two equations that it is convenient to put on the form
\begin{align}
\label{eq:1}
[ \mathbb P^{(1)}_\pm , H]  -   (h_\pm -H) \mathbb P^{(1)}_\pm -  \Omega^{(1)}_\pm \Pi_\pm &= \mp i\xi\cdot \nabla \Pi_+ ,\\
\label{eq:2}
[ \mathbb P^{(2)}_\pm , H]  - (h_\pm -H) \mathbb P^{(2)}_\pm -  \Omega^{(2)}_\pm \Pi_\pm &=  F_\pm
\end{align}
where $F_\pm$ depends on $\mathbb P^{(1)}$ and $\Omega^{(1)}_\pm$
\[
F_\pm= \frac 1{2i} \{h_\pm, \mathbb P^{(1)}_\pm\} +\frac 1{2i}\{\Omega^{(1)}_\pm,\Pi_\pm\} + \Omega^{(1)} _\pm\mathbb P^{(1)} _\pm
- \frac 1{2i} \{ \mathbb P^{(1)}_\pm, H\} .
\]
For solving these equations, we multiply them on both sides by $\Pi_+$ or $\Pi_-$, which gives four relations each time. 

\smallskip

Let us perform the computation for the $plus$-mode. Multiplying~\eqref{eq:1} on the right by $\Pi_+$ and on the left successively by $\Pi_+$ and  $\Pi_-$, we obtain 
two relations 
\begin{align*}
\Pi_+ \Omega^{(1)}_+\Pi_+=0, \;\;&\;\; \Pi_-\Omega_+^{(1)} \Pi_+ = i \Pi_- \xi\cdot \nabla\Pi_+\Pi_+.
\end{align*}
Using that we want to find $\Omega_+^{(1)}$ self-adjoint, we deduce that we can choose
$$ \Omega^{(1)}_+= i \Pi_- \xi\cdot \nabla\Pi_+\Pi_+-i \Pi_+ \xi\cdot \nabla\Pi_+\Pi_-=\Omega.$$
Similarly, for the $minus$-mode 
$$ \Omega^{(1)}_-= -i \Pi_+ \xi\cdot \nabla\Pi_+\Pi_-+i \Pi_- \xi\cdot \nabla\Pi_+\Pi_+=\Omega.$$
Multiplying ~\eqref{eq:1} on the left by $\Pi_+$ and on the right by $\Pi_-$, we end up with
 \[
 (h_+-h_-) \mathbb P^{(1)}_+ = i \Pi_+ \xi\cdot \nabla\Pi_+\Pi_-.
 \]
Choosing $\mathbb P^{(1)}_+$ self-adjoint, we obtain 
 \[
 \mathbb P^{(1)} _+= \frac {i}{2|w(x)|}( \Pi_+\xi\cdot \nabla \Pi_+\Pi_- -  \Pi_-\xi\cdot \nabla \Pi_+\Pi_+)=\mathbb P.
 \]
 
 We argue in a similar way for the $minus$-mode and find 
 \[
 \mathbb P^{(1)} _-= -\frac {i}{2|w(x)|}( \Pi_-\xi\cdot \nabla \Pi_+\Pi_+ -  \Pi_+\xi\cdot \nabla \Pi_+\Pi_-)=- \mathbb P.
 \]
 
 Let us now determine $\mathbb P^{(2)}_+$ and $\Omega^{(2)}_+$. We first decompose $F_+$ as the sum of a self-adjoint matrix and a skew-symmetric one:
 $ F_+=F_{+, {\rm aa}}+F_{+,{\rm ss}}$
 with 
 \begin{align*}
 F_{+, {\rm aa}}&=\frac 12 (F_++F_+^*),
  \qquad
 F_{+,{\rm ss}}= \dfrac{1}{2} (F_+-F_+^*)\\
 F_+^*&= -\frac 1{2i} \{h_+,\mathbb P\} +\frac 1{2i} \{\Pi_+,\Omega\} +\Omega \mathbb P -\frac 1{2i} \{ H,\mathbb P\}.
 \end{align*}
 We have used  $\{ M, N\}^*=-\{N,M\}$ for smooth matrix-valued function $M$ and $N$. 
 We also obtain 
 \begin{equation}\label{condition}
 \Pi_\pm F_{+,{\rm ss}}\Pi_\pm=0,
 \end{equation}
 which is required from~\eqref{eq:2} (when multiplied on both side by $\Pi_\pm$).
 These relations come from  $\mathbb P\Omega=\Omega\mathbb P$,
 \begin{align*}
 &\Pi_\pm \begin{pmatrix} 0 & 1 \\ -1 & 0\end{pmatrix}\Pi_\pm =0_{\C^{2,2}}\;\;
 \mbox{and}\;\;A(u) \begin{pmatrix} 0 & 1 \\ -1 & 0\end{pmatrix} + \begin{pmatrix} 0 & 1 \\ -1 & 0\end{pmatrix} A(u)=0,\;\;\forall u\in\R^2.
 \end{align*} 
 Then, multiplying~\eqref{eq:2} by $\Pi_+$ on the right, we deduce
 $$ \Omega^{(2)} \Pi_+ = - F_+\Pi_+.$$
  One then chooses
  \begin{equation}\label{def:Omega2}
  \Omega^{(2)}_\pm =  - \Pi_+ F_{+, {\rm aa}} \Pi_+ - \Pi_- F_+ \Pi_+ - \Pi_+ F_+^* \Pi_-.
  \end{equation}
  For determining $\mathbb P_+^{(2)}$, we multiply~\eqref{eq:2} by $\Pi_-$ on the right 
  \[
(h_+-h_-) \mathbb P_+^{(2)}\Pi_-= -F_+\Pi_-,
\]
and we obtain 
\begin{equation}\label{def:bigP2}
\mathbb P_+^{(2)}= -\frac 1 {2|w(x)|} ( \Pi_+ F_+ \Pi_- + \Pi_- F_+^* \Pi_+ + \Pi_- F_{+, {\rm aa}} \Pi_-) .
\end{equation}
The polynomial features of these matrices in the variable  $\xi$  and their properties as functions of $w$ come from their explicit formula. These aspects determine their behavior at $\infty$ and close to~$\Upsilon$. 
\end{proof}

\begin{remark}
	\label{rem:C3}
As already observed in the literature (\cite{bi,MS,N1,N2,Te}, it is possible to push these asymptotics at any order by constructing a sequence of matrices $(\Omega^{(j)}_\pm, \mathbb P^{(j)}_\pm)_{j\in \N}$ that will satisfy controls of the form 
 \begin{align*}
  |\partial_x^\alpha\partial_\xi^\beta  \mathbb P^{(j)}_\pm(x,\xi) | & \leq C_{\alpha,\beta}\langle \xi\rangle ^j
 \langle x\rangle^{(|\alpha|+j)(1+n_0)} |w(x)|^{-|\alpha|-2j},\\
  |\partial_x^\alpha\partial_\xi^\beta \Omega ^{(j)}_\pm(x,\xi) | & \leq C_{\alpha,\beta}\langle \xi\rangle^j
 \langle x\rangle^{(|\alpha|+j)(1+n_0)} |w(x)|^{-|\alpha|-2j+1}.
 \end{align*}
 
As a consequence of the computations above, we also have the following result. 

\begin{lemma}\label{lem:Omega2}
Let $\Phi^{t^\flat,t}_\pm(z^\flat)$ be a trajectory reaching the point $z^\flat\in\Upsilon$ at time $t^\flat$ with the conditions of~\eqref{hypothesis}. Then, we have for $t$ close to $t^\flat$, 
\begin{align*}
\Omega^{(1)} _\pm(\Phi^{t^\flat,t}_\pm(z^\flat))=\mathcal{O}(1),\;\;\;\;&
\mathbb P^{(1)}_\pm (\Phi^{t^\flat,t}_\pm(z^\flat))=\mathcal{O}(|t-t^\flat|),\\
 \Omega^{(2)} _\pm(\Phi^{t^\flat,t}_\pm(z^\flat))=\mathcal{O}(|t-t^\flat|^2),\;\;&\;\;
  \mathbb P^{(2)}_\pm (\Phi^{t^\flat,t}_\pm(z^\flat))=\mathcal{O}(|t-t^\flat|^3).
 \end{align*}
\end{lemma}

\end{remark}
  
\section{Generalization to time-dependent Hamiltonian}\label{app:generalization}

We consider a Hamiltonian 
$$H(t,z)= v(t,z) {\rm Id} _{\C^2} + A(w(t,z)),\;\; w(t,z)=\,^t(w_1(t,z),w_2(t,z)) \in\R^2$$
with subquadratic growth and polynomial 
control  of the gap at infinity (as in~\eqref{hyp:gapinfinity}). 
The crossing set is the subset of $\R\times\R^d$ given by 
$$\Upsilon=\{ (t,z)\in \R\times \R^{2d},\;\; w(t,z)=0\}.$$
We denote as before by $h_+$ and $h_-$ the eigenvalues of $H$ and $\Pi_+$ and $\Pi_-$ the associated eigenprojectors. 
Following~\cite{FG03}, a point $(t^\flat, z^\flat)\in\Upsilon$ is a non-degenerate crossing point if and only if 
$${\rm Rk} \, dw(t^\flat, z^\flat)=2 \;\;\mbox{and}\;\; E(t^\flat, z^\flat) :=\partial_t w (t^\flat, z^\flat) + \{v,w\} (t^\flat, z^\flat) \not =0.$$
With such a point, we associate the vector  
$$\omega = \frac {E(t^\flat, z^\flat) } {|E(t^\flat, z^\flat) |}\;\;\mbox{and}\;\; r=|E(t^\flat,z^\flat)|.$$
By Proposition~1 in \cite{FG03}, there exists a pair of generalized trajectories passing through non-degenerate crossing points and we denote them by $\Phi^{t,t^\flat}_\pm(z^\flat)$.

\smallskip

\noindent{\bf Time-dependent eigenvectors along the trajectories.} 
Starting from a point $(t_0,z_0)\in\R\times \R^{2d}$ such that $\Phi^{t^\flat, t_0}_\pm (z_0)= z^\flat$, we associate with these trajectories time-dependent  eigenvectors by solving the differential equation 
$$\frac d {dt} \vec Y_\pm(t) = B_\pm (\Phi^{t,t^\flat}_\pm(z^\flat))\vec Y_\pm(t),\;\; \vec Y_\pm(t_0)= \vec Y_0$$
where 
$$B_\pm=\Pi_\mp( \partial_t \Pi_\pm(t,z)+ \{v, \Pi_\pm\}) \Pi_\pm,$$
and $\vec Y_0$ is an eigenvector of $H(t_0,z_0)$ for the $\pm$-mode. 
One can then prove that the vectors $\vec Y_\pm(t)$, $t<t^\flat$ can be continued up to  $t=t^\flat$.

\smallskip

\noindent{\bf Profile equations.}
The  profile equations associated with the trajectory $\Phi^{t,t_0}_\pm(z_0)$ write 
$$i\partial_t u_\pm = {\rm Hess}\, H(t,\Phi^{t,t_0}_\pm)
\begin{pmatrix} y\\ D_y\end{pmatrix}
\cdot \begin{pmatrix} y\\ D_y\end{pmatrix} u_\pm,\;\; u_\pm (t_0)=\varphi_\pm.$$
Close to $t^\flat$, we have the asymptotics 
$$ {\rm Hess}\, H(t,\Phi^{t,t_0}_\pm)\sim_{t\rightarrow t^\flat} \pm \frac 1 {r|t-t^\flat|} ( \, ^t dw(t^\flat, z^\flat) ({\rm Id}_{\R^2} - \omega \otimes \omega) dw(t^\flat, z^\flat))  $$
which allow to define 
 ingoing  profiles $u_\pm^{\rm in}$ by 
 $$u_\pm (t)\sim {\rm e} ^{ \pm \frac i2  \widehat {\Gamma_0} \ln |t-t^\flat| } u_\pm^{\rm in}\;\;\mbox{as}\;\; t\rightarrow t^\flat, \;\;t<t^\flat $$
 with 
 $$ \widehat {\Gamma_0} = \frac 1 {r}  \, ^t dw(t^\flat, z^\flat) ({\rm Id}_{\R^2} - \omega \otimes \omega) dw(t^\flat, z^\flat)  
 \begin{pmatrix} y\\ D_y\end{pmatrix}
\cdot \begin{pmatrix} y\\ D_y\end{pmatrix} .$$
Note that in the case we have studied, the function $w$ only depends on~$x$ and thus the operator $\widehat {\Gamma_0} $ is an operator of multiplication.

\smallskip

\noindent{\bf Transition formulas.}
The transitions formula are now operator-valued. The function $\eta(y)$ inside the coefficients of Theorem~\ref{theo:main} have to be replaced by the operator
$$\eta(y, D_y)=\left( \omega\cdot (d_zw(t^\flat, z^\flat) \, ^t(y,D_y))  , \omega^\perp\cdot (d_zw(t^\flat, z^\flat) \, ^t(y,D_y)) \right).$$
Then, the transition rules are the same as in Theorem~\ref{theo:main}.

\smallskip

\noindent{\bf The Hermitian case.}
Such an approach extends to Hermitian Hamiltonians with crossings that have the geometric feature of~\cite{FG03}, the so-called 
{\it generic involutive codimension 3 crossing} (see also~\cite{CdV2}).
Assume 
$$ \displaylines{
H(t,z)= v(t,z) {\bf 1}_{\C^2} +
\begin{pmatrix} w_1(t,z)  & w_2 (t,z) +iw_3(t,z)  \\ w_2(t,z) -iw_3(t,z)  & -w_1 (t,z) 
\end{pmatrix} ,\cr
\mbox{with}\;\; w(t,z)=\,^t(w_1(t,z), w_2(t,z) ,w_3(t,z))\in\R^3.\cr}$$
Set 
$$E(t,z)= (\partial_ t w(t,z) + \{ v,w\})(t,z)\;\; \mbox{and}\;\; B(t,z)= \,^t(\{ w_2,w_3\},  \{ w_3,w_1\},\{ w_1,w_2\})(t,z).$$
The strategy developed in this article  extends  to crossing points $(t^\flat,z^\flat)$ close to $\Upsilon$, where the latter is a codimension 2 or 3 manifold, with 
$ E(t,z)\cdot B(t,z)$ identically equal to $0$ in a neighborhood of $(t^\flat,z^\flat)$ and $|E(t^\flat,z^\flat)|>[B(t^\flat,z^\flat)|$. Even though this situation is not generic, it contains for example the case where $w=w(x)$.
More intricate phenomena appear in the generic setting (see~\cite{F2} and \cite{F1} for example). 
Note however that a special attention has to be attached to the diabatic basis used at the crossing point because the eigenvectors are now complex-valued.


\end{document}